\documentclass[12pt,reqno]{amsart}
\usepackage{fullpage}
\usepackage{amsfonts}
\usepackage{amsmath}
\numberwithin{equation}{section}
\usepackage{amsthm}
\usepackage{mathtools}
\usepackage{scalerel}
\usepackage{amssymb}
\usepackage{times}
\usepackage{graphicx}
\usepackage{subcaption}
\usepackage[export]{adjustbox}
 \usepackage{hyperref}
\usepackage{epsfig}
\usepackage{xcolor}
\usepackage{multirow}
\usepackage{enumerate}
\usepackage{booktabs}

\newtheorem{theorem}{Theorem}
\numberwithin{theorem}{section}
\newtheorem{proposition}[theorem]{Proposition}
\newtheorem{lemma}[theorem]{Lemma}
\newtheorem{corollary}[theorem]{Corollary}
{ \theoremstyle{remark}\newtheorem{remark}[theorem]{Remark} }
\newcommand{\diag}{\operatorname{diag}}

\newcommand{\tr}{\operatorname{Tr}}

\newcommand{\Res}{\operatorname{Res}}

\newcommand{\card}{\operatorname{card}}

\newcommand{\im}{\operatorname{Im}}

\newcommand{\Prob}{\operatorname{{Pr}}}

\newcommand{\E}{\operatorname{\mathbf{E}}}

\newcommand{\var}{\operatorname{Var}}

\newcommand{\Li}{\operatorname{Li}}

\newcommand{\sm}[1]{{\scriptscriptstyle#1}}

\begin{document}

\title[Article]
{Zeros of conditional Gaussian analytic functions, random sub-unitary matrices and q-series}

\author{Yan V. Fyodorov$^{1}$, Boris A. Khoruzhenko$^{2}$, and Thomas Prellberg$^{2}$}
\address{$^1$ King's College London, Department of Mathematics, London  WC2R 2LS, United Kingdom}
\address{$^2$ Queen Mary University of London, School of Mathematical Sciences, London E14NS, United Kingdom}

\begin{abstract}
%We investigate radial statistics of zeros of hyperbolic Gaussian Analytic Functions (GAF) of the form $\varphi (z) = \sum_{k\ge 0} c_k z^k$ given that $|\varphi (0)|^2=t$ and assuming coefficients $c_k$ to be independent standard complex normals. We obtain the full conditional distribution of $N_q$, the number of zeros of $\varphi (z)$ within a disk of radius $\sqrt{q}$ centred at the origin and find that at a suitable scale $N_q$ exhibits Gaussian fluctuations about the value $(1-q)^{-1}$ in the limit when $q\to 1^{-}$, the limit that captures the entire zero set of $\varphi (z)$.  We also determine the scaling law for the conditional probability $P_k(t;q)=\Prob\{ N_q=k \, | \, |\varphi(0)|^2=t\}$ for large values of $k$ proportional to $(1-q)^{-1}$.  Finally, we determine the asymptotic form of $P_k(t;q)$ in the limit when $k$ is kept fixed whilst $q$ approaches 1. To leading order, the hole probability $P_0(t;q)$ does not depend on $t$ for $t>0$ but yet is different from that of $P_0(t=0;q)$ and coincides with the hole probability for unconditioned hyperbolic GAF of the form $\sum_{k\ge 0} \sqrt{k+1}\, c_k z^k$.  We also find that asymptotically as $q\to 1^{-}$,  $P_k(t;q)= e^t P_{k}(0;q)$ for every fixed $k\ge 1$ with $P_{k}(0;q)=\Prob\{ N_q\!=\!k\!-\!1 \}$.\\[1ex]
%
%
We investigate radial statistics of zeros of hyperbolic Gaussian Analytic Functions (GAF) of the form $\varphi (z) = \sum_{k\ge 0} c_k z^k$ given that $|\varphi (0)|^2=t$ and assuming coefficients $c_k$ to be independent standard complex normals. We obtain the full conditional distribution of $N_q$, the number of zeros of $\varphi (z)$ within a disk of radius $\sqrt{q}$ centred at the origin, and prove its asymptotic normality in the limit when $q\to 1^{-}$, the limit that captures the entire zero set of $\varphi (z)$.  In the same limit we also develop precise estimates for conditional probabilities of moderate to large deviations from normality. Finally, we determine the asymptotic form of $P_k(t;q)=\Prob\{ N_q\!=\!k \, | \, |\varphi(0)|^2=t\}$  in the limit when $k$ is kept fixed whilst $q$ approaches 1. To leading order, the hole probability $P_0(t;q)$ does not depend on $t$ for $t\!>\!0$ but yet is different from that of $P_0(t\!=\!0;q)$ and coincides with the hole probability for unconditioned hyperbolic GAF of the form $\sum_{k\ge 0} \sqrt{k+1}\, c_k z^k$.  We also find that asymptotically as $q\!\to \!1^{-}$,  $P_k(t;q)\!= \!e^t P_{k}(0;q)$ for every fixed $k\!\ge\! 1$ with $P_{k}(0;q)\!=\!\Prob\{ N_q\!=\!k\!-\!1 \}$.

\end{abstract}

\maketitle

\section{Introduction}

\subsection{Gaussian power series}
Since the pioneering work of Paley and Wiener \cite{PW1934}, Kac \cite{Kac1959} and Rice \cite{Rice1945}, statistical properties of zeros of Gaussian Analytic Functions (GAF) have been studied in various contexts, see  \cite{Sodin05} for a brief overview. Among the many models of GAF, one model stands out because of its exact solvability. It is the Gaussian power series
\begin{align} \label{phi(z)}
\varphi (z)=\sum\nolimits_{k\ge 0}c_k z^k %c_0 + c_1 z +c_2z^2 + c_3 z^3 + \ldots ,
\end{align}
with independent standard complex normal coefficients $(c_k)_{k\ge 0}$. As was established by Peres and Virag \cite{PV04}, its zero set is a determinantal point process in  its natural domain of analyticity, the unit disc $\mathbb{D}=\{ z: |z|<1\}$, with the correlation kernel given by $K (z,w)=\pi^{-1} (1-z\overline{w})^{-2}$.
This property allowed one to study statistics of associated zeros at fine level of detail.

Although Peres and Virag discovered the determinantal structure of the correlations in the zero set of $\varphi(z)$ working directly with the power series, this feature is better elucidated by a link to random matrices \cite{K09}.  Let $U_n$ be a unitary matrix sampled uniformly from the unitary group $U(n)$ according to its Haar's measure and let $T_{n-1}$ be the sub-unitary matrix obtained from $U_n$ by removing its first row and first column. Then in the limit of large matrix dimension $n\to\infty$, the eigenvalues of $T_{n-1}$  converge, as a point process in $\mathbb{D}$, to the zero set of ${\varphi (z)}$.  The probability law of elements of the zero set (which we for brevity will simply call "zeros" henceforth) can then be inferred from that of the eigenvalues of $T_{n-1}$. One advantage of this approach is that the joint eigenvalue density of $T_{n-1}$ is known in closed form \cite{SZ00} from which one can easily obtain the determinantal kernel $K (z,w)$
and other statistical properties of the zeros.

Remarkably, the link to random matrices with ensuing integrability of the probability law of the zero set is preserved, albeit in a modified form, on freezing the free coefficient of ${\varphi (z)}$ and treating it as a parameter.
It was shown by Forrester and Ipsen \cite{FI2019} that
the eigenvalues of  the random matrices
\begin{align}\label{A}
A_{n,\tau}= U_n \diag \left( \sqrt{\tau}, 1, \ldots, 1\right), \quad U_n\in U(n),
\end{align}
with $U_n$ being drawn at random from the unitary group $U(n)$, converge, as a point process in the scaling limit
\begin{align}\label{sl}
n\to\infty,\,
 \tau n= t>0,
\end{align}
to the zero set $\mathcal{Z}$ of the Gaussian power series $\varphi (z)$ conditioned by the event that $|\varphi (0)|^2=t$. Note that because of the rotational invariance of the standard complex normal distribution, the conditional probability law of $\mathcal{Z}$ given  $|\varphi (0)|^2=t$ is the same as that of the zero set of
\begin{align}\label{g_a(z)}
\varphi_a (z) = a  +\sum\nolimits_{k\ge 1} c_k z^k\, .
\end{align}
for every complex $a$ such that  $|a|^2=t$.

The matrix ensemble $ A_{n,\tau}$ featuring in Eq.(\ref{A}) was introduced by one of us in \cite{F00} as a one-parameter extension
of the truncated unitary matrices studied by \.Zyczkowski and Sommers in \cite{SZ00}. It is the simplest example  of a more general ensemble of finite-rank multiplicative non-unitary deformations of $U_n$, see \cite{FS00, FS03}.
Similarly to $T_{n-1}$, the eigenvalue joint probability density of $A_{n,\tau}$  is known in closed form for every finite matrix dimension, although in this case the eigenvalue correlation functions are not determinantal \cite{F00}. By passing to the limit $n\to\infty$ in the expression for the eigenvalue correlation functions derived in \cite{F00}, Forrester and Ipsen \cite{FI2019} obtained the correlation functions of the zero set
$\mathcal{Z}_a=\{ z: z\in \varphi_a^{-1} (0) \} $.
It is worth mentioning that prior to Forrester's and Ipsen's work the zero set of $\varphi_t(z)$ was linked to a different but less tractable random matrix ensemble by Tao \cite{T2013}.
The essential difference to the unconditional Gaussian power series $\varphi(z)$  is that the process $\mathcal{Z}_a$  is not determinantal and its correlation functions have a complicated structure. Consequently,  the statistical properties of $\mathcal{Z}_a$ remain largely unexplored.

Our original interest in the zero set of \eqref{g_a(z)}  stemmed from the utility of the random matrix ensemble \eqref{A} for the analytic study of scattering in open chaotic system with one scattering channel \cite{F00}. We will further comment on this in Section \ref{subsection:1.3}. In the random matrix context, the zero set of \eqref{g_a(z)} models eigenvalue outliers of \eqref{A} in the scaling limit \eqref{sl}. It also models eigenvalue outliers of the circular law for i.i.d. matrices under small rank additive perturbations \cite{T2013} which otherwise are difficult to study analytically.

The aim of the present paper is to essentially advance the line of inquiry started by Forrester and Ipsen in \cite{FI2019}. We concentrate on characterising the point process $\mathcal{Q}_a$ of
 absolute squares of the zeros of  $\varphi_a(z)$ in the unit disc,
\begin{align}\label{Q_a}
\mathcal{Q}_a=\{ q: \,\, q=|z|^2, z\in \varphi_a^{-1} (0) \}.
\end{align}
This process is somewhat simpler than  the zero set of $\varphi_a(z)$ but still it encodes useful information about the latter. We  demonstrate that $\mathcal{Q}_a$ gives rise to non-trivial and seemingly new extreme value distributions.
The novelty can be traced to the eigenvalue product conservation law
\begin{align}\label{claw}
| \det A_{n, \tau}|^2= \tau
\end{align}
which induces correlations between absolute squares of zeros in the scaling limit \eqref{sl}. In contrast, the probability law of the set $\mathcal{Q}$ of absolute squares of  zeros for the unconditional power series  \eqref{phi(z)} is known to be the same as that of infinite sequence of independent samples $\beta_{1,k}$ from Beta$(1,k)$ distribution \cite{PV04},
\begin{align}\label{beta_1k}
\mathcal{Q} \overset{\mathrm{d}}{=} (\beta_{1,k})_{k\ge 1}.
\end{align}

Radial  statistics of complex eigenvalues in non-Hermitian random matrix ensembles including truncations of random unitary matrices, as well as in associated Coulomb plasma models, have attracted considerable interest over the years \cite{Shirai2006,akemann2009gap,Allez2014,akemann2014universal,cunden2016large,LACT2019,FL2022,BG-ZL2023,akemann2023universality,L2024,ACM2024}, see also Section 3 in \cite{BF-book2024} and references there.
Our work significantly extends the existing literature in two directions. Firstly,  we mainly focus on the distribution of spectral outliers in the lower end of the  spectrum of the eigenvalue moduli, whilst the main body of the existing literature focuses either on the spectral bulk or the upper end of the moduli spectrum.   Secondly, the point process $\mathcal{Q}_a$ cannot be represented by sequences of independent random variables. This should be contrasted with radial eigenvalue statistics of non-Hermitian random matrix models with matrix distributions invariant under the unitary conjugations, where the independence holds and frequently simplifies calculations considerably.  To illustrate this point,
the counting function $N_q$ of $\mathcal{Q}$ has moment generating function \cite{PV04}
\begin{align}\label{pgfNq}
\E_{\varphi}\! \left\{(1+x)^{N_q}\right\} = \prod\nolimits_{k\ge 1} (1+xq^k)
\end{align}
as a straightforward consequence of \eqref{beta_1k}. Such a product representation for the moment generating function comes in very handy as one can use the Euler–Maclaurin formula for asymptotic analysis in various limits of interest.
On the other hand, we will demonstrate that the conservation law \eqref{claw} leads, via the $q$-binomial theorem, to the moment generating function of the conditional distribution of $N_q$ in the form of a $q$-series
\begin{align}\label{mgf_cond}
\E_{\varphi}\! \left\{ \left. (1+x)^{N_q}  \right|  \varphi(0) \right\} = 1+ \sum_{k=1}^{\infty} \frac{x^k\, q^{\frac{k(k-1)}{2}} }{(1-q) \cdots (1-q^k)}\, e^{|\varphi(0)|^2(1-q^{-k})}  \, .
\end{align}
Note that on averaging this $q$-series over the distribution of $\varphi(0)$ one recovers, via Euler's partition identity \eqref{euler_p},
the product representation for the unconditional moment generating function:
\begin{align*}
\E_{\varphi} \!\left\{(1+x)^{N_q}\right\}= \E_{\varphi}  \left\{\E_{\varphi}  \!\left\{ \left. (1+x)^{N_q} \, \right| \,  \varphi (0) \right\}\right\}.
\end{align*}
However, for a fixed value of $\varphi(0)$ the $q$-series in \eqref{mgf_cond} does not fold into a product.

There is a significant body of literature on asymptotics of $q$-series going back to Ramanujan  \cite{RIV} but the series in \eqref{mgf_cond} appears to be new. In our work we develop a method for obtaining its asymptotic expansions in various limits of interest which may be of independent interest. The method is based on an integral representation for the $q$-series in \eqref{mgf_cond} and has proved to be fairly efficient. For example we obtain the large deviation form for the probability (conditional and not) for $N_q$ to be equal $k$
in the scaling limit $k\to\infty$ and $q=e^{-s/k}$ with a greater precision than one known for the Ginibre ensemble \cite{Shirai2006,LACT2019}.

\subsection{Main results} Throughout the paper we use the notation $(z;q)_k$ for the $q$-shifted factorial,
\begin{align*}
(z;q)_k=
\begin{cases}
(1-z)(1-zq) \cdots (1-zq^{k-1}), &  k=1,2,\ldots, \\
1 & k=0.
\end{cases}
\end{align*}
Denote by $\mathcal{Q}$ the point process of absolute squares of the zeros of
% the Gaussian power series
$\varphi (z)=\sum_{k\ge 0}c_kz^k$, and let $N_q$ be its counting function:
\begin{align}\label{N_q}
\mathcal{Q} =\{ |z|^2: z\in {\varphi}^{-1} (0) \}, \quad N_q= \card\left( \mathcal{Q} \cap [0,q] \right).
\end{align}
It is convenient to think of $\mathcal{Q}_a$ as of the point process $\mathcal{Q}$ conditioned by the event $\varphi(0)=a$.
Our first observation is that its probability generating functional can be computed in terms of a series.

\begin{theorem} \label{Thm:M} Assume that the test function $h(q)$ has bounded variation on $[0,1]$ and is such that
\begin{align}\label{cond}
h(1)=0 \quad \text{and}\,\, \int_0^1 \frac{dV_{h}(q)}{1-q} < \infty,
\end{align}
where $V_{h}(q)$ is the full variation of $h(q)$ on the interval $[0,q]$. Then,
\begin{align} \nonumber
\E_{\varphi} \left\{ \left. \prod\nolimits_{q\in \mathcal{Q}} \!\left(1+h (q)\right)   \right|  \varphi (0) \right\}
& = \\ \label{limPGFmr1}
  \MoveEqLeft[10]
1+ \sum_{k=1}^{\infty} (-1)^k \int_0^1 \ldots \int_0^1 \frac{e^{|{\varphi (0)}|^2\left(1-\frac{1}{q_1 \cdots q_k}\right)}}{q_1 \cdots q_k}\, \prod_{j=1}^k \frac{q_1 \cdots q_j}{1- q_1 \cdots q_j} \, dh(q_1) \ldots dh(q_k) \, .
\end{align}
\end{theorem}
The series representation in \eqref{limPGFmr1} leads to the conditional distribution of  $N_q$.
\begin{theorem}
\label{Thm2}
\begin{align}\label{fm8i}
\E_{\varphi}\! \left\{\left. (1+x)^{N_q} \right|  \varphi (0)\,  \right\} = \sum_{m=0}^{\infty} \frac{x^m \, q^{\binom m 2}}{(q;q)_m}  \, e^{|{\varphi (0)}|^2 ( 1- {q^{-m}})} \, \, .
\end{align}
In particular,
the conditional probability mass function of $N_q$ is given by
\begin{align}\label{pmfN_q}
\Prob \{  N_q\!=\!k  \big| \varphi (0)\!=\!a \} = \sum_{m=k}^{\infty} (-1)^{m-k} \binom  m k \frac{ q^{\binom m 2} }{(q;q)_m}\, e^{|a|^2( 1- {q^{-m}})} \, .
\end{align}
\end{theorem}

Theorems \ref{Thm:M}  and \ref{Thm2} are proved in Section \ref{Sec2.1}.

\begin{remark}\label{rem3}
The case of $\varphi(0)=0$ is special: the series ${\varphi (z)}=\sum_{k\ge 0} c_k z^k $
conditioned on the event that $\varphi (0)=0$ has non-random zero at the origin and the rest of its zeros have the same probability law as the zeros of the unconditional series.
Therefore the conditional distribution of $N_q$ given ${\varphi (0)}=0$ is the same as the distribution of $N_q+1$. Consequently,
\begin{align}\label{c_0=0}
\E_{\varphi}\! \big\{ \!(1+x)^{N_q} \big| {\varphi (0)}=0  \big\} = (1+x) \E_{\varphi} \!\big\{\!(1+x)^{N_q}\big\} = \scaleobj{.9}{\prod\nolimits_{k\ge 0}} (1+xq^k),
\end{align}
which, by Euler's partition identity \eqref{euler_p}, is equivalent to \eqref{fm8i} with $\varphi(0)=0$.
By a similar argument,
\begin{align}\nonumber
\E_{\varphi}\! \Big\{ {\prod\nolimits_{q\in \mathcal{Q}}} (1\!+\!h(q)) \big| {\varphi (0)}\!=\!0 \Big\}
& = \\ \label{PGF_t_0}
  \MoveEqLeft[5]
 (1\!+\! h(0)) \E_{\varphi} \Big\{{\prod\nolimits_{q\in \mathcal{Q}}} (1\!+\!h(q))\Big\} =
{\prod\nolimits_{k\ge 0}} \Big(1\!-\! \scaleobj{.9}{\int_0^1} q^{k} dh(q)\Big).
\end{align}
It can be verified, see proof of Theorem \ref{thm:7},  that under the assumptions of Theorem \ref{Thm:M} on $h$
 \begin{align*}
{\prod_{k=0}^{\infty}} \Big(1\!-\! \scaleobj{.8}{\int_0^1} q^k dh(q)\Big)
=
1+ {\sum_{k=1}^{\infty}} (-1)^k \scaleobj{.9}{\int_0^1 \ldots \int_0^1} \frac{1
}{q_1 \cdots q_k}\, {\prod_{j=1}^k} \frac{q_1 \cdots q_j}{1- q_1 \cdots q_j} \, dh(q_1) \ldots dh(q_k),
\end{align*}
and, hence, \eqref{PGF_t_0}  is equivalent to  \eqref{limPGFmr1} with $\varphi(0)=0$.
\end{remark}

\begin{remark} From \eqref{pmfN_q} we find that if $\varphi (0)=a$ then the probability for $\varphi (z)$ to have
no zeros in the disk $\mathbb{D}_r=\{z\in \mathbb{C}: \, |z| \le r \}$
equals
\begin{align}\label{P_0_t}
1+\sum_{k=1}^{\infty}  \frac{(-1)^{k}\, r^{k(k-1)} \, }{(1-r^2)(1-r^4) \cdots (1-r^{2k})} \exp\left[|a|^2\left(1-\frac{1}{r^{2k}}\right)\right], \quad 0<r<1.
\end{align}
It is evident  that  this probability is extremely small for every fixed $a\not= 0$ and decays at the rate of an essential singularity as $r\to 0$. This is markedly different to the unconditional probability for $\varphi (z)$ to have no zeros in the disk $\mathbb{D}_r$ which decays quadratically at $r= 0$.
As a consistency check, on averaging the expression in \eqref{P_0_t} over the standard complex normal distribution of $a$, one recovers the known expression $\prod\nolimits_{k\ge 1} (1-r^{2k})$ for the unconditional probability \cite{PV04}.
\end{remark}

\begin{remark}
Let $q_{{(1)}}$ be the smallest point of $\mathcal{Q}$, i.e., the smallest absolute square of the zeros of ${\varphi (z)}$. Obviously, $\Prob \{ q_{{(1)}} \!> \!q  | {\varphi (0)}\!=\!a \}=
\Prob \{ N_q \!=\! 0 | {\varphi (0)}\!=\!a \}$, hence, by differentiating \eqref{pmfN_q} in variable $q$, we find that the conditional probability density of $q_{{(1)}}$ given ${\varphi (0)}=a$ equals
\begin{align}\label{eq:24}
\sum_{k=1}^{\infty} (-1)^{k-1}\!
\left[k \left(\frac{|a|^2}{q^{k}} -1 \right) + \sum_{i=1}^k \frac{i}{1-q^{i}}\right] \frac{q^{\binom k 2 -1}  }{(q;q)_k} \, e^{|a|^2( 1- q^{-k})}.
\end{align}
Fig. \ref{Fig:1} displays this density for several values of $|a|^2$. From \eqref{pmfN_q} one also finds the conditional probability density of the $m$-th smallest point of $\mathcal{Q}$, see Remark \ref{rem:2.5}. The conditional \emph{joint} distribution of the extreme points of $\mathcal{Q}$ is determined by  \eqref{limPGFmr1}. We do not have a closed form expression for it.
\end{remark}

\begin{remark}
Equation \eqref{eq:24} defines a family of extreme value  distributions depending on the parameter $|a|^2$.
Evidently, it cannot be transformed to one of the three canonical forms, known as the Gumbel, Fr\'echet and Weibull types, characterizing extremes in  sequences of i.i.d. real random variables.
In fact, our next theorem asserts that this family interpolates between particular cases of two canonical forms, one of the Fr\'echet family and Gumbel.
\end{remark}

\begin{figure}
\includegraphics[width=.6\linewidth]{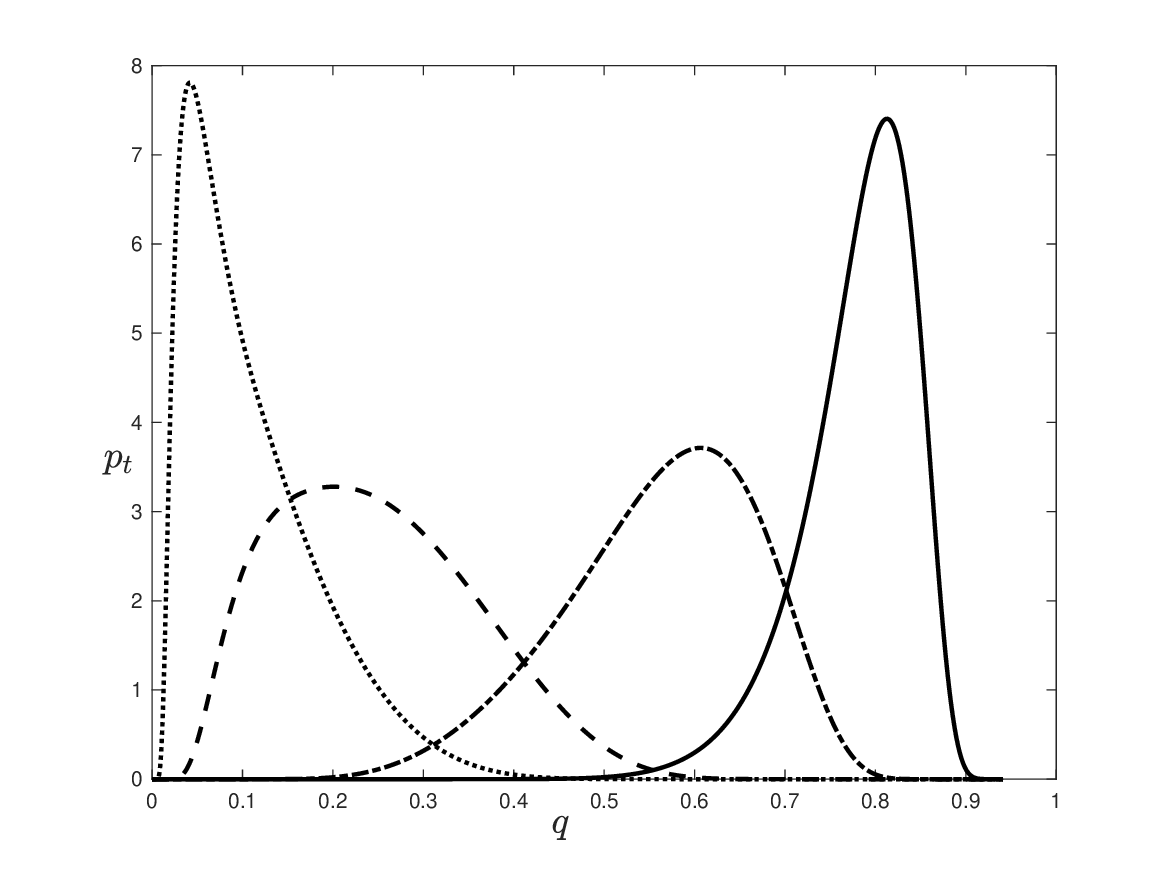}\\
\caption{The conditional probability density \eqref{eq:24} of the smallest absolute square of the zeros of $\varphi (z)$ \eqref{g_a(z)} given $\varphi (0)=a$ when $|a|^2=0.08$ (dotted line), $|a|^2=0.3$ (dashed line), $|a|^2=2$ (dotdashed line) and $|a|^2=8$ (solid line). }
\label{Fig:1}
\end{figure}

 \begin{theorem} \label{Thm:FG}
 \phantom{a}
 \begin{itemize}
  \item[(a)] The conditional distribution of $N_q$ given $\varphi (0)=a$, converges, in the scaling limit
 \begin{align}\label{sl_p1}
 a \to 0, \quad  q=|a|^2 u,
 \end{align}
  to Bernoulli ($e^{-1/u}$) for every $u>0$, which is a characteristic of the Fr\'echet family.
 \item[(b)] The conditional distribution of $N_q$ given $\varphi (0)=a$ converges, in the scaling limit
 \begin{align}\label{sl_p}
 a \to \infty, \quad \quad q=1- \frac{\log (|a|^2) - \log\log (|a|^2) + v}{|a|^2 },
\end{align}
to Poisson ($e^{-v}$) %with parameter $\lambda= e^{-v}$
for every $v\in \mathbb{R}$, which is a characteristic of the Gumbel family.
 \end{itemize}

 \end{theorem}

 We prove this theorem in Section \ref{Sec:3}.

\begin{remark}
Part (b) of Theorem \ref{Thm:FG} suggests that $\mathcal{Q}_{a}$ \eqref{Q_a} converges to Poisson process in the scaling limit \eqref{sl_p}. This is indeed so and we prove it in Section \ref{Sec:3}.
 \end{remark}

The rest of this section is concerned with the conditional distribution of $N_q$ in the limit when $q$ approaches 1, the limit that captures the entire zero set of $\varphi (z)$.

\smallskip

From \eqref{fm8i} we find that the conditional expected value of $N_q$ grows and becomes unbounded in this limit, whilst the \emph{relative} magnitude of fluctuations in values of $N_q$ vanishes:
\begin{align}\label{muqa}
\mu_q (a):= &\E_{\varphi}\! \left\{ N_q |  \varphi (0)\!=\!a \right\}
= \frac{e^{| a|^2( 1- {q^{-1}})} }{(1-q)}
= \frac{1}{1-q}+| a|^2
+O(1-q)\\
\label{sigmaqa}
\sigma^2_q( a):= &\var\! \left\{N_q |  \varphi (0)\!=\!a \right\}
= \frac{1}{2 (1-q)}+\frac{1}{4}+| a|^2+O(1-q)\, .
\end{align}
This asymptotic behaviour is very similar to that of the unconditional expected value and variance:
\begin{align*}
\E_{\varphi} \{N_q \} = \frac{q}{1-q}=\frac{1}{1-q} -1, \quad \var \{N_q\}= \frac{q}{1-q^2} = \frac{1}{2 (1-q)}+\frac{1}{4}+O(1-q).
\end{align*}
Therefore conditioned or not, the typical values of $N_q$ in the limit when $q$ approaches 1 concentrate about the value $(1-q)^{-1}$. In fact,

 \begin{theorem}
 \label{thm:1.9} For every complex $a$ it holds in the limit $q\to 1^{-}$ that
 \\[-1ex]
\begin{align*}
\Prob \Big\{ \Big| N_q\!-\!\frac{1}{1\!-\!q} \Big| \ge \frac{\xi}{(1\!-\!q)^{\alpha}} \Big| \varphi(0)\!=\!a \Big\}
\!=\!
\begin{cases}
\displaystyle{
\frac{1}{\sqrt{\pi}}\int_{|u|\ge \xi} e^{-u^2} du + o(1)
}\, , & \text{if $\alpha\!=\!\frac12$},
\\[4ex]
\displaystyle{
\exp\Big[\!-\frac{\xi^2\, (1\!+\!o(1))}{(1\!-\!q)^{2\alpha-1}}\Big]
} \, , & \text{if $\frac 12\!<\!\alpha\!<\! 1 $},
\\[4ex]
\displaystyle{
 \exp\Big[
\!-\frac{(\xi\!+\!1) \Psi (1\!+\!\xi) (1\!+\!o(1))}{(1\!-\!q)}
\Big]
} \, , & \text{if $\alpha\!=\!1$},
\end{cases}
\end{align*}
\\[0ex]
where in the last line
$\Psi(s)\!=\!s\!+\!W_0(-se^{-s}) \!+\!\frac{1}{s}\Li_2\! \big(1\!-\!e^{s+W_0(-se^{-s})}\big)$, $s\!>\!1$,
with $W_0$ and $\Li_2$  being, correspondingly, the principal branch of the Lambert-$W$ function and the dilogarithm function.
\end{theorem}
\smallskip

We see that the leading asymptotic form of the scaling law of fluctuations does not depend on the value of $\varphi (z)$ at the origin $z=0$.  In case when $\varphi (0)=0$, which corresponds to the unconditional case, convergence of the distribution of $N_q$ to normal was previously established by Peres and Virag \cite{PV04} and the large deviation rate function for $N_q$ was obtained by Waknin \cite{W23} (who also found the rate function for deviations on the scale exceeding  $\var \{N_q\}$). Both results are obtained by using representation of $N_q$ as sum of independent random variables, and the convergence to normal distribution in case when the fluctuations scale with $\sqrt{\var \{N_q\}}$  can also be inferred from the general central limit theorem due to Costin-Lebowitz \cite{CL95} and Soshnikov \cite{S00}. These tools are not available if $a\not= 0$ and we prove Theorem \ref{thm:1.9} in Section \ref{Sec:4} by making use of the $q$-series representation of the moment generating function of $N_q$ \eqref{fm8i}, see Theorem \ref{afixed_t}.

\smallskip
Now, we turn to investigating probabilities for $N_{q}$ to take finite values in the limit when $q$ is approaching 1. To simplify the notations, let
\begin{align*}%\label{P_k}
P_k(q)=  \Prob \{N_q\!=\!k \},  \quad  P_k(|a|^2; q)=  \Prob \{N_q\!=\!k \big| {\varphi (0)} \!=\!a\}.
\end{align*}
In the unconditioned case the hole probability, $P_k(q)$ with $k=0$, is given by \eqref{pgfNq} with $x=-1$. Its asymptotic form near $q=1$ can obtained by the use of the Euler - Maclaurin formula. To leading order one finds \cite{PV04}
\begin{align}\label{hp_0_a}
P_0(q)=\sqrt{\frac{2\pi}{1-q} }\, \exp\!\left[{-\frac{\pi^2}{6(1-q)}} + O(1) \right].
\end{align}
Similarly, using Euler - Maclaurin we find that
\begin{align*}%\label{hp_0_a}
 P_1(q)=P_0(q)\, \sum_{m=1}^{\infty} \frac{q^m}{1-q^m}  =
P_0(q)\, \frac{\log (1-q)}{\log q }\,   \left[1+ O \left(\frac{1}{\log (1-q)} \right)\right]\, .
\end{align*}
Extending this calculation to the values of $N_q$  other than 0 and 1 might be feasible but is cumbersome because of multiple sums due to higher order derivatives of \eqref{pgfNq}. An alternative calculation based on a contour integral representation of  the
moment generating function yields
\begin{theorem}
\label{thm:1.10}
In the limit when $q$ approaches 1 whilst $k$ is fixed,
\begin{align}\label{N_q=k}
P_k(q) \sim \frac{\lambda_q^k}{k!}\, P_0(q), \quad \lambda_q= \frac{\log (1-q)}{\log q },
\end{align}
and, moreover, for every $t>0$ and $k\ge 1$
 \begin{align}\label{N_q=k_cond}
\lim_{q\to 1^{-}} \frac{P_{k}(t;q)}{P_{k-1}(q) } =
 e^{t}.
\end{align}
\end{theorem}

\smallskip
This theorem is proved in Section  \ref{section5}.

\smallskip

 The conditional hole probability, $P_0(|a|^2; q)$, for general $a^2>0$  is not covered by Theorem \ref{thm:1.10} and hence requires a separate treatment. Forrester and Ipsen conjectured that its leading order term in the limit when $q$ approaches 1  is independent of $a$ and	then naturally concluded that it should be the same as that for the unconditional case, $a=0$, see Eq. (3.20) in \cite{FI2019}. Quite surprisingly, we find that despite the leading-order form being indeed independent of $a \ne  0$, it is still different from that for the unconditional probability result, in contrast to the conclusion in  \cite{FI2019}.

 Our next theorem furnishes an asymptotic expansion of the conditional hole probability and confirms that  to next-to-leading-order term $ \Prob \{ N_q\!=\!0 | {\varphi (0)} \!=\!a  \} $
does not depend on the value of $a\not=0$. However the expression is different from one for $a=0$ and actually coincides with that
 of the Gaussian power series  of the form $\sum_{k\ge 0} \sqrt{k+1}\, c_k z^k$, see
 \cite{BNPS2018} where the asymptotic form of the hole probability at $q=1$ was determined for the entire family of such hyperbolic GAF.
\begin{theorem}
 \label{Thm:hole_pr}
 \phantom{a}
 For every $t>0$ in the limit $\varepsilon\to 0^{+}$
  \begin{align}\label{gappr2m}
 \log P_0(t;e^{-\varepsilon})
\!=\!
-\frac{1}{\varepsilon} \left[ \frac{ w^2}{2}+w +\frac{\pi^2}{3} \right]  -\frac{1}{2} \log (\varepsilon t) + t  -\frac{t}{w}  +O(\varepsilon), \quad w=W_{-1}(-\varepsilon t)
\end{align}
where $W_{-1}$ is the lower branch of the Lambert-$W$ function.
 \end{theorem}

We prove this theorem in Section \ref{Sec:4}.

 \begin{remark}\label{remark_Lambert}
The Lambert-$W$ function is the multivalued inverse to the function $w\mapsto w e^w$  \cite{CGHJK96}. It has two real-valued branches on the interval $[-e^{-1},0)$. One branch has values
greater than or equal to $-1$, and the other has values less than or equal to $-1$. The former  is denoted by $W_0$ and is referred to as the principal branch,  and the latter is denoted by $W_{-1}$ and is often referred to as the lower branch, see Fig. \ref{Fig:6}. At zero, the principal branch vanishes and the lower branch has the asymptotic expansion \cite{CGHJK96}
\begin{align}\label{W_ae}
W_{-1}(-\varepsilon) \mathop{\sim}_{\,\,\,\,\varepsilon\to 0^+}
-\eta-\log\eta+\sum_{n=1}^{\infty}%
\frac{1}{\eta^{n}}\sum_{m=1}^{n}\genfrac{[}{]}{0.0pt}{}{n}{n-m+1}\frac{\left(-\log\eta\right)^{m}}{m!}, \quad \quad \eta = \log \frac{1}{\varepsilon},
\end{align}
where $\genfrac{[}{]}{0.0pt}{}{n}{n-m+1}$ are unsigned Stirling numbers of the first kind.
Using this expansion we find that for $t>0$
\begin{align*}
\MoveEqLeft
\log P_0(t;e^{-\varepsilon}) = \\
\nonumber & -\frac{ (\eta + \log \eta )^2}{2\varepsilon}  + \frac{(1+\log t)\eta}{\varepsilon} +\frac{(\log t )\log \eta}{\varepsilon}
-\frac{\log^2 t }{2\varepsilon} -\frac{\pi^2}{3 \varepsilon }+ \frac{1}{\varepsilon}O\left( \frac{\log \eta}{ \eta}\right), \quad \eta=\log \frac{1}{\varepsilon}.
\end{align*}
It is evident that the correct scale in the expansion \eqref{gappr2m}   is given in terms of the Lambert-$W$ function. For, truncating the asymptotic series for $W_{-1}(-\varepsilon t)$ introduces an error which renders the $O(\log \varepsilon)$ terms in \eqref{gappr2m} obsolete.
\end{remark}
\smallskip

From Theorems \ref{thm:1.9} and \ref{thm:1.10} one is naturally led to the problem of determining scaling law(s) far in the tail of the distribution of $N_q$. The theorem below determines the scaling law when deviations from the expected value are of the same order as the expected value itself or are at a smaller scale.

\begin{theorem}\label{maintheorem}
Fix $s>0$. Then in the limit $k\to\infty$ it holds that
\begin{align}\label{eq:mthm1}
P_k\big(e^{-s/k}\big)=\frac{A(s)}{\sqrt{k\pi}}\, e^{-k\Psi (s)} (1+O(1/k))
\end{align}
with
\begin{align*}
A(s)=\sqrt{\frac{\log z_s}{2(z_s-s)}}, \quad \Psi(s)=\log z_s+\frac{1}{s} \Li_2(1-z_s),
\end{align*}
where
$z_s=e^{s+w_s}$ and $w_s=W_{-1}(-se^{-s})$ when $s\in(0,1)$ and $w_s=W_{0}(-se^{-s})$ when $s\ge 1$.
Moreover, for $t>0$,
\begin{align}\label{eq:mthm2}
P_k\big(t; e^{-s/k}\big)=z_se^{t(1-z_s)} P_k\big(e^{-s/k}\big) \big(1+O(1/k)\big)\, .
\end{align}
\end{theorem}

This theorem is proved in Section  \ref{section5}.

\smallskip

\begin{remark}
Theorem \ref{maintheorem} is essentially a refinement of Theorem \ref{thm:1.9}. We demonstrate this at the end of Section \ref{section5}.
 \end{remark}

\smallskip

\begin{remark}
Theorem \ref{thm:1.10}  determines the asymptotic form of $P_k(q)$ and $P_k(t;q)$ in the limit when $q$ approaches 1 for $k=O(1)$ and Theorem \ref{maintheorem} does this for values of $k \propto 1/(1-q)$. A natural question is then whether the two results match asymptotically. We can show that they do
\emph{assuming uniformity of \eqref{eq:mthm1}  in $s$ for small values of $s$}.
Indeed,  consider the regime when $k$ is fixed and $q\to 1^{-}$, so that $s=-k\log q \to 0$.  Expanding the rate function $\Psi(s)$ and the pre-exponential factor $A(s)$ for small values of $s$ as in Appendix \ref{app:C}, one obtains  from \eqref{eq:mthm1} that
\begin{align}\label{match}
 \frac{A(s)}{\sqrt{k \pi}} \, e^{-k \Psi(s)}
 \mathop{=}_{\,\,\,\,q \to 1^{-}}
 \sqrt{\frac{\log(1-q)}{k\log q}}
 \frac{  e^{\frac{\pi^2}{ 6 \log q}}}{\sqrt{2 \pi k}\, k^{k} e^{-k} }
  \, \left[\frac{\log(1-q)}{\log q}\right]^{k} \left(1 + O\left(\frac{1}{\log q}\right)\right).
 \end{align}
Comparing the r.h.s. with the expression in \eqref{N_q=k}
gives a good agreement, with the main difference being that $k!$ in \eqref{N_q=k}  is replaced in by its  Stirling approximation for large $k$. Apart from that, \eqref{match} is off the exact asymptotic form  \eqref{N_q=k} by a factor  $\sqrt{-\log(1-q) / (2 k \pi)}$ which is not surprising if the asymptotics is not fully uniform. Similarly,
\begin{align*}
z_se^{t(1-z_s)} \sim e^t \, \frac{k \, \log q}{\log (1-q)}, \quad k=O(1), s=-k\log q \to 0,
\end{align*}
and we find that \eqref{eq:mthm2} and \eqref{N_q=k_cond} asymptotically match too.
 \end{remark}
 \smallskip

\begin{remark}
Now consider $N_q$ in the regime when $q$ is fixed,  $q\in (0,1)$. It follows from \eqref{fm8i}, see Proposition \ref{prop:5.3}, that the tail probability is given by
\begin{align} \label{P_k(a;q)}
P_k(t;q)&=\frac{q^{\binom k2}}{(q;q)_\infty}e^{t(1-q^{-k})}\left(1+O(q^k)\right), \quad q\in (0,1), k\to\infty.
\end{align}
\emph{Assuming uniformity of \eqref{eq:mthm1} in $s$ for large values of $s$} we find that the scaling law \eqref{eq:mthm1} asymptotically matches \eqref{P_k(a;q)} too. Indeed, setting $s=-k\log q$ we find in the limit $k\to\infty$ that
\begin{align*}
z_se^{t(1-z_s)}   \frac{A(s)}{\sqrt{k \pi}} \, e^{k \Psi(s)} = q^{\binom k 2} e^{t(1-q^{-k})} e^{-\frac{\pi^2}{6 \log q}} \sqrt{\frac{-\log q}{2\pi}} \, \Big(1+O(q^k)\Big).
 \end{align*}
Comparing this with the result in \eqref{P_k(a;q)} gives a remarkable agreement, with the only difference that the $q$-product is replaced by its expansion for $q\to 1$ (one could not hope for more here).
\end{remark}

\subsection{Random matrix perspective}
\label{subsection:1.3}
The analytic tractability of the random matrix ensemble $A_{n,\tau}$ defined in \eqref{A} is an important factor in determining the probability law of the zero set of the Gaussian power series \eqref{phi(z)} and \eqref{g_a(z)}, and in this context the random matrices $A_{n,\tau}$ are just a tool. Nevertheless, this ensemble being the simplest representative of finite-rank multiplicative non-unitary deformations away from the Circular Unitary Ensemble (CUE) \cite{FS00, FS03} is also interesting on its own merit. In the context of random matrix theory and its applications matrices of this type have been used as a model for studying the universal features of resonances in open chaotic maps, see the discussion and further references in, e.g., \cite{Novaes13}. In that context, an important characteristics known as the resonance widths are modelled by the distance between the complex eigenvalue and the unit circle, thus providing a motivation for studying statistics of the absolute values of eigenvalues.

In the limit of large matrix dimension the spectrum of $A_{n,\tau}$ undergoes an abrupt restructuring at $\tau=0$. For every fixed $\tau\in (0,1)$ the eigenvalues of $A_{n,\tau}$ form a narrow band near the unit circle. More precisely, for $n$ large, the bulk of the eigenvalues lie at a distance of order $1/n$ from the unit circle \cite{F00}. In fact, it can be shown that the smallest absolute square $q_{min}$ of the eigenvalues of $A_{n,\tau}$ satisfies
\begin{align}\label{AGum}
\lim_{n\to\infty}
\Prob \Big\{q_{min} \le  1- \frac{1-\tau}{\tau}\frac{\log n - \log(\log n) + y }{n} \Big\}=e^{-\frac{\tau}{1-\tau} e^{-y}}.
\end{align}
In contrast, if $\tau=0$ then, in addition to an eigenvalue band near the unit circle, $A_{n,0}$ has trivial eigenvalue at zero, and there are also eigenvalues deep inside the unit disk: the probability for $A_{n,0}$ to have at least two eigenvalues in the disk of radius $r$ about zero is $1-\prod_{k=1}^{n-1} (1-r^{2k})$ and, hence, is asymptotically non-zero  \cite{SZ00, HKPV2009}. This is illustrated in Fig. \ref{Fig:0}.
\begin{figure}
\begin{subfigure}[b]{0.36\textwidth} % width of upper left subfigure
    \includegraphics[width=\linewidth]{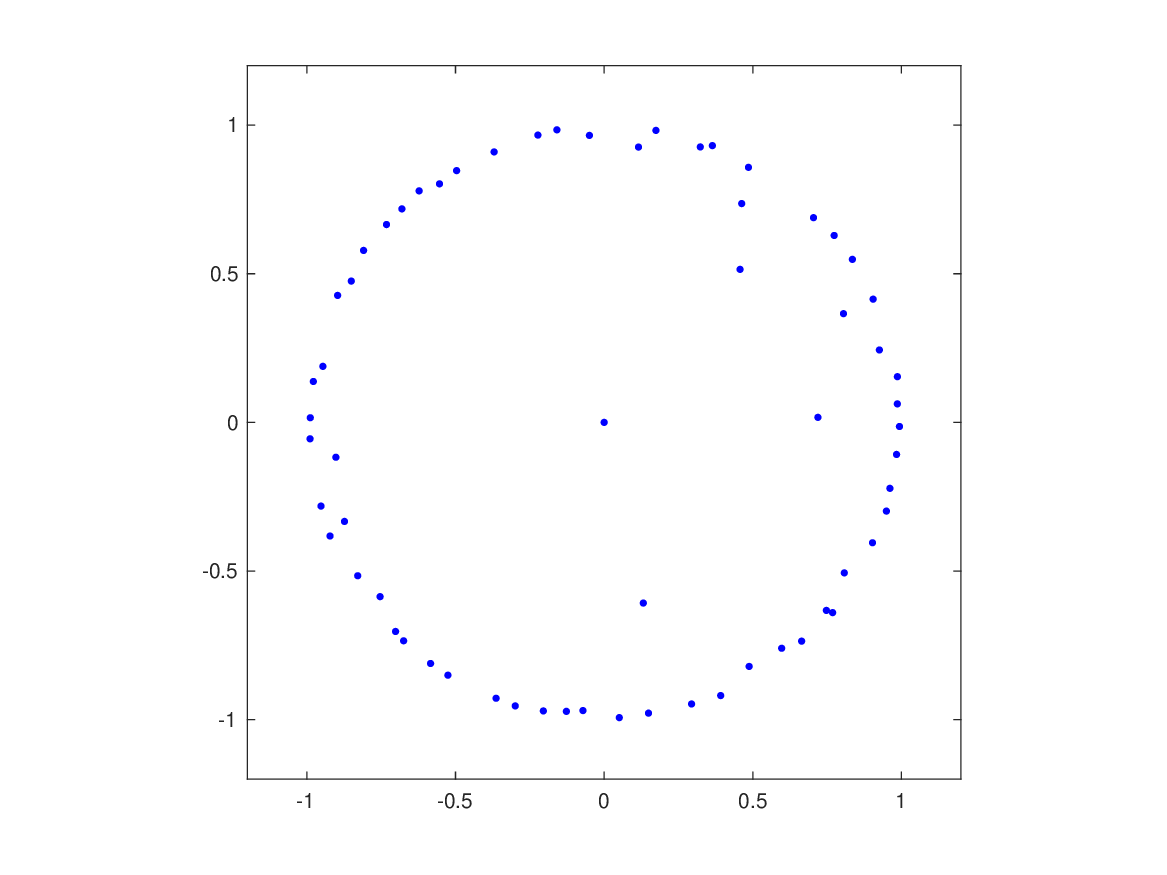}\hspace{-2.4em}%
    \caption{$\tau=0$} % subcaption
\end{subfigure}%
\hspace{-2.4em}%
\begin{subfigure}[b]{0.36\textwidth} % width of upper right subfigure
    \centering
    \includegraphics[width=\linewidth]{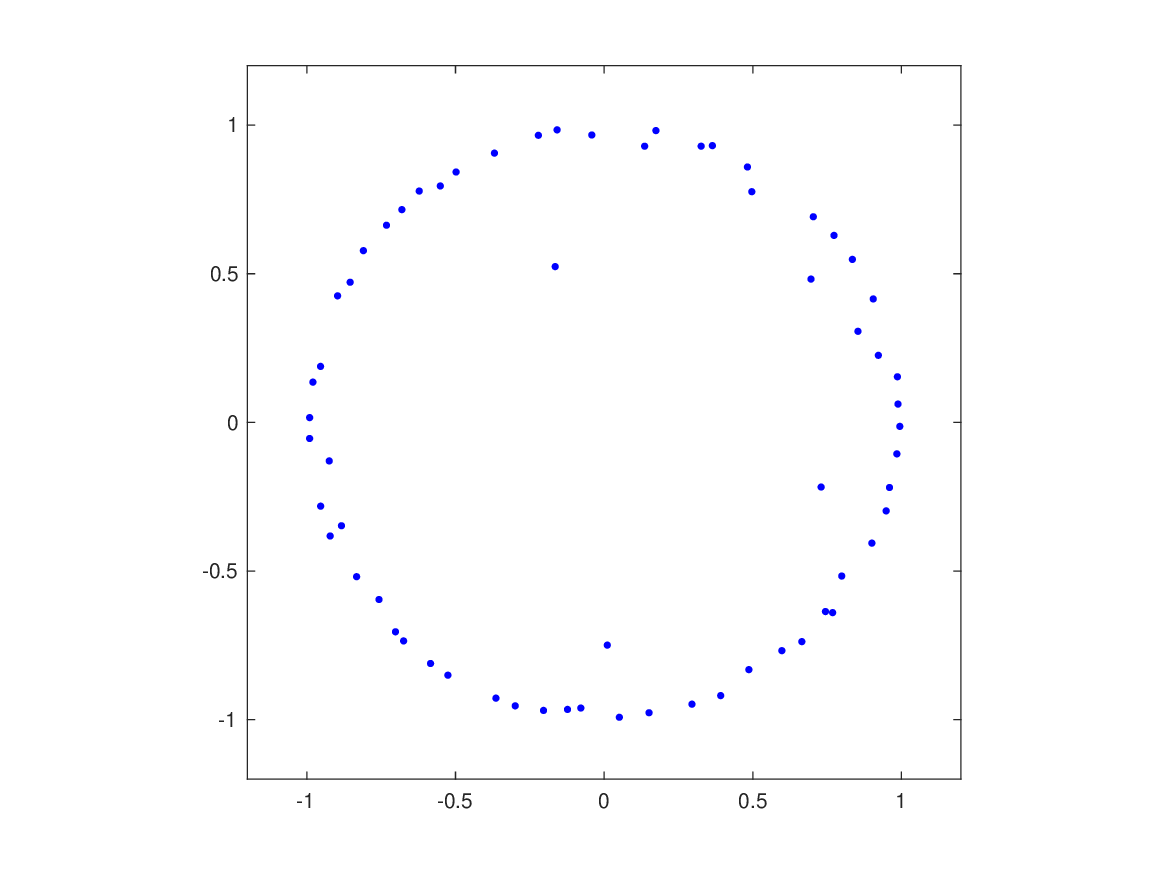}\hspace{-2.4em}%
    \caption{$\tau=0.0078$} % subcaption
\end{subfigure}%
\hspace{-2.4em}%
\begin{subfigure}[b]{0.36\textwidth} % width of lower left subfigure
    \includegraphics[width=\linewidth]{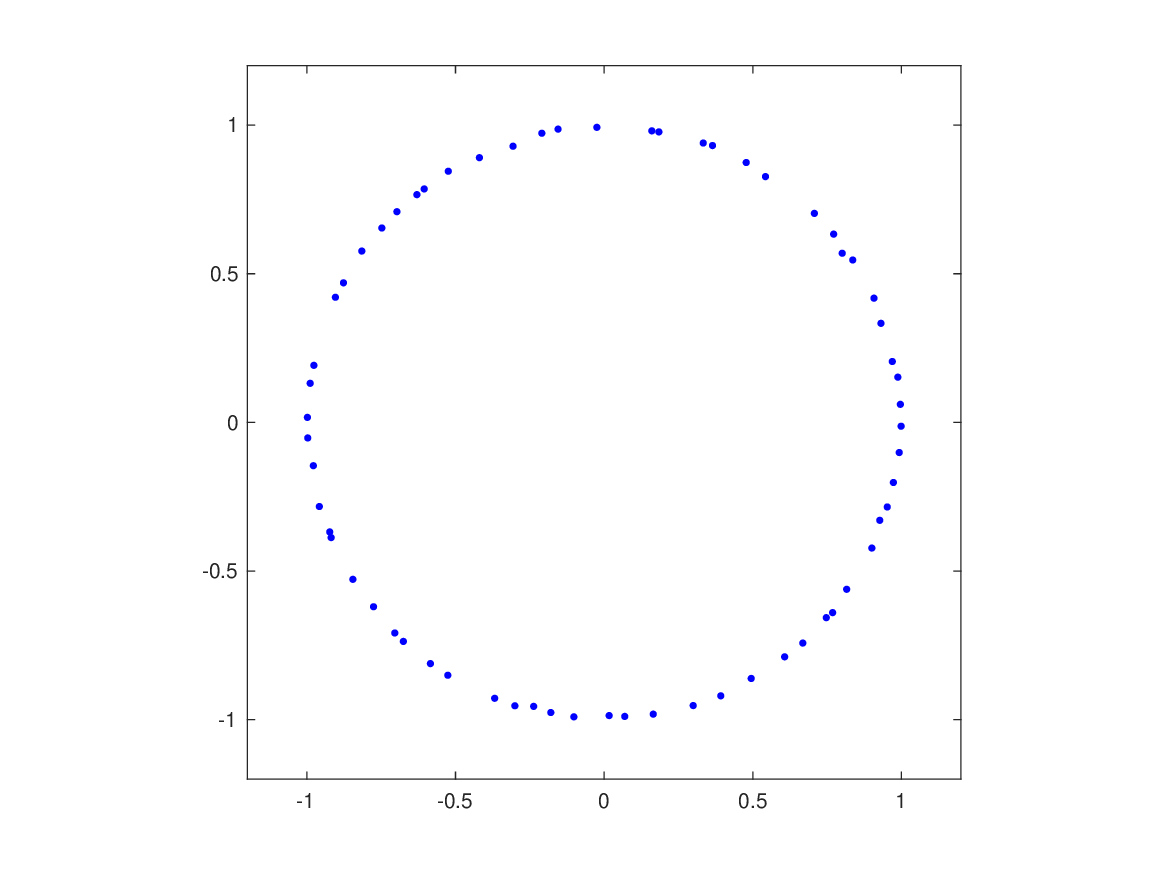}
    \caption{ $\tau=0.5$} % subcaption
\end{subfigure}%
\caption{Eigenvalues of random matrix \eqref{A} of dimension $n=64$. The three plots were produced by generating a single sample $U_n$ from the CUE and then numerically diagonalising the subunitary matrix $U_n \diag ( \sqrt{\tau}, 1, \ldots, 1)$ for three values of $\tau$: $\tau=0$ , $\tau=0.0078$, and $\tau=0.5$. }
\label{Fig:0}
\end{figure}

The abrupt spectral restructuring in the random matrix ensemble \eqref{A} poses natural questions about the critical scaling of $\tau$ that captures the emergence of outliers and about the distribution of extreme eigenvalues in the scaling limit. The first question was answered by Forrester and Ipsen in \cite{FI2019} who obtained the nonzero density of eigenvalues of $A_{n,\tau}$ and the eigenvalue correlation functions in the scaling limit \eqref{sl}, see also the recent work of Dubach and Recker  \cite{DR2023} for a different perspective on the emergence of outliers. And in this paper we obtain the (marginal) distributions of extreme eigenvalues in the scaling limit  \eqref{sl}.

A similar spectral restructuring occurs in the large-$n$ limit of rank-one additive non-Hermitian deviations from random Hermitian  matrices.
Namely, define
\begin{align}\label{J}
J_{n,\gamma} = H_n + i \Gamma_n, \quad \quad \Gamma_n=\gamma \diag (1, 0, \dots, 0),
\end{align}
where $\gamma >0$ and $H_n$ are $n\times n$ random Hermitian matrices normalised so that $\mathbf{E}\tr H_n^2=n$.  Under
suitable conditions on the distribution of $H_n$ the eigenvalues of $J_{n,\gamma}$, except possibly for a single outlier, form a narrow band just above the real axis for $n$ large.  If $\gamma <1$,
then up to small order corrections all eigenvalues of $J_{n,\gamma}$
lie within distance of $n^{-1}$ to the real axis \cite{RW2017}. For $\gamma>1$  there emerges exactly one eigenvalue lying away from this eigenvalue band, at a distance
asymptotically equal to $\gamma - \gamma^{-1}$ \cite{RR2014, RW2017, outlier}.
In the physics literature, the eigenvalues of $J_{n,\gamma}$  are associated with the zeroes of a (single-channel) scattering matrix in the complex energy plane, and their complex conjugates are associated with the poles of the same scattering matrix, known as ``resonances''. In that context the absolute value of the eigenvalue's imaginary part is associated with the ``resonance width'', so that the eigenvalues of $J_{n,\gamma}$ in the narrow band near the real axis are the ``narrow resonances'' and the outlier is the ``broad resonance''. As far as the spectral outliers are concerned, the two models \eqref{A} and \eqref{J} behave differently, see \cite{DE2023,FKP2023} for an analysis of the critical scaling regime in model \eqref{J}. At the same time, as far as the local statistics of the narrow resonances are concerned, the two models are completely interchangeable. Namely, if $H_n$ is chosen by sampling the Gaussian Unitary Ensemble (GUE) then, locally at energy $x\in (-2,2)$, the eigenvalue correlations of $J_{n,\gamma}$ near the real axis are functions of the combination
\begin{align*}
g(x)=\frac{\gamma +\gamma^{-1}}{2\pi \nu_{sc}(x)}
\end{align*}
where $\nu_{sc}(x)$ is the semicircular density of the (real) eigenvalues of $H_n$, provided the distances between the eigenvalues of $J_{n,\gamma}$ are measured in units of the mean separation $\Delta=(n \nu_{sc} (x))^{-1}$ between neighbouring eigenvalues of $H_n$ \cite{FK99}. Identifying  the radial eigenvalue deviation from the unit circle in the ensemble  $A_{n,\tau}$ with the imaginary parts of the eigenvalues of $J_{n,\gamma}$, and measuring the distance between the eigenvalues in units of the mean separation $\Delta =2\pi/n$ between neighbouring  eigenvalues of $U_n$, the eigenvalue correlations of $A_{n,\tau}$ near the unit circle have exactly the same functional form in the limit $n\to\infty$ as those of $J_{n,\gamma}$, provided one identifies $(1+\tau)/(1-\tau)$ with $g(x)$, see \cite{FS03}.

\section{Absolute squares of zeroes and their extremes}
\label{Sec:2}

Throughout the rest of this paper, $\theta(q)=1$ if $q>0$ and $\theta(q)=0$ if $q\le 0$.

\subsection{Probability generating functional of absolute squares and their counting function}
\label{Sec2.1}

Denote by $z_1, \ldots, z_n$ the eigenvalues of the matrix $A_{n,\tau}$ \eqref{A} and consider the set of their absolute squares,
\begin{align}\label{AS}
\mathcal{Q}_{n,\tau} = \left\{|z_j|^2: \,\, j=1, \ldots, n \right\}.
\end{align}
As $U_n$ is chosen at random from the unitary group $U(n)$, this is a point process in the interval $[0,1]$. Its probability law is determined by the probability generating functional (p.g.fl)
\begin{align}\label{PGF2mr}
G_{n, \tau} [h] = {\E}_{U(n)} \big\{ \scaleobj{.85}{\prod\nolimits_{q\in \mathcal{Q}_{n,\tau}}} \! \left[1+h(q)\right]\big\},
\end{align}
where $\E_{U(n)}$ is the average over the Haar measure on $U(n)$.

\begin{theorem} \label{thm:7}
Assume that the test function $h (q)$ satisfies the assumptions of Theorem \ref{Thm:M}. Then $G_{n, \tau} [h]$  converges, in the scaling limit $n\to \infty$, $n\tau=t>0$, to
\begin{align}\label{limPGFmr}
G_t[h] = 1+ \sum_{k=1}^{\infty} (-1)^k \int_0^1 \ldots \int_0^1 \frac{e^{t\left(1- \frac{1}{q_1 \cdots q_k} \right)}}{q_1 \cdots q_k}\, \prod_{j=1}^k \frac{q_1 \cdots q_j}{1- q_1 \cdots q_j} \, dh(q_1) \ldots dh(q_k) \, .
\end{align}
\end{theorem}
\begin{proof}
The Haar measure of $U(n)$ induces a probability distribution on the set of eigenvalues of $A_{n,\tau}$. Correspondingly,
\begin{align}\label{pgfl}
G_{n, \tau} [h]  = \int \ldots \int
p_n (z_1, \ldots, z_n) \prod_{k=1}^n  \! \left[1+h (|z_j|^2)\right] \prod_{k=1}^n d^2 z_k\, ,
\end{align}
where $p_n (z_1, \ldots, z_n)$ is the joint eigenvalue density which is known in the closed form \cite{F00,Kozhan07},
\begin{align}\label{eq:-2}
p_n(z_1, \ldots, z_n) = \frac{1}{n\pi^n (1-\tau)^{n-1}}\,   \delta \Big(\tau - \prod_{k=1}^n |z_k|^2\Big)
\!\!
\prod_{1\le i < j \le n} |z_j-z_i|^2\,  \prod_{k=1}^n  \theta (1-|z_k|)  \, .
\end{align}
The $\delta$-function in \eqref{eq:-2} can be handled by using a suitable integral transform in variable $\tau$. We use the Mellin transform
\begin{align*}
{\mathcal M} \left[ f( \tau); s  \right] = \int_0^{\infty}s^{\tau-1}  f(\tau)  \, d\tau \, .
\end{align*}
To this end we consider the sequence of functions of variable $\tau$:
\begin{align}\label{FPGF}
f_n(\tau) = \theta (1-\tau) (1-\tau)^{n-1} G_{n, \tau} [h] \, .
\end{align}
In the scaling limit \eqref{sl}, $G_{n,\tau}[h] = e^{-t}f_n\left(\frac{t}{n}\right)(1+o(1))$. The limiting form of $f_n\left(\frac{t}{n}\right)$ can be obtained from that of its Mellin transform,
\begin{align}\label{MTS}
{\mathcal M} \left[ f_n\left( \frac{t}{n}\right);s  \right] =
n^s{\mathcal M} \left[ f_n(\tau);s  \right] \, .
\end{align}
The Mellin transform on the r.h.s. can be found with help of the Andr\'eief integration formula \cite{Forrester_A}
\begin{align}\label{AI}
\int \ldots \int \det [\phi_j(z_i)]_{i,j=1}^n \det [\psi_j(z_i)]_{i,j=1}^n \prod_{k=1}^n d^2 z_k =n! \det \left[\int \phi_i(z) \psi_j(z) d^2 z \right]_{i,j=1}^n,
\end{align}
for a similar calculation see Corollary  2.2 in \cite{Dubach18}. We find that
\begin{align*}
{\mathcal M} \left[ f_n(\tau);s  \right]
  =
 (n-1)! \prod_{k=1}^{n} \int_0^1 (1+h (q))q^{k+s-2} dq\, .
\end{align*}
By our assumptions, $h(q)$ is bounded. Hence, all integrals in the above product are convergent for every $s$ in the half-plane $\Re s>0$ which is the characteristic strip of the Mellin transform in \eqref{MTS}. After first expanding the product of integrals into a sum as
\begin{align*}
\MoveEqLeft
\prod_{k=1}^{n} \int_0^1 (1+h (q))q^{k+s-2} dq =
\int_0^1 \ldots \int_0^1 \left[ 1+\sum_i h (q_i) + \sum_{i<j} h (q_i) h (q_j)   \right.
\\
&  \left. + \sum_{i<j<k} h (q_i) h (q_j) h(q_k) + \ldots
 + h (q_1)h (q_2)\cdots h (q_n)
 \right] q_1^{s-1} q_2^{s} \cdots q_n^{n+s-2} dq_1 \ldots dq_n\,
\end{align*}
and then using the relations
\begin{align*}
\int_0^1 q^{k+s-2} dq =\frac{1}{k+s-1}\quad \text{and} \quad
\int_0^1  h (q) \, q^{k+s-2}\, dq & = -\frac{1}{k+s-1}\int_0^1 q^{k+s-1} dh (q)\,
\end{align*}
in each term of the sum one can finally convert the resulting sum back into a product, obtaining
\begin{align}\label{MTPGF}
{\mathcal M} \left[ f_n(\tau);s  \right] &=
B(n,s) \prod_{k=1}^{n} \left[1- \int_0^1q^{k+s-1} dh (q) \right]
\end{align}
with the help of the beta function.
The product on the r.h.s. is convergent in the limit $n\to\infty$ absolutely and uniformly in any region in the complex $s$-plane where the series
$
\sum_k\int_0^1q^{k+s-1} dh (q)
$
is convergent absolutely and uniformly. Since for every $s$ in the half-plane $\Re s> 0$
\begin{align}\label{sf}
\sum_{k=1}^{\infty} \left| \int_0^1q^{k+s-1} dh (q) \right| \le \sum_{k=1}^{\infty} \int_0^1\left| q^{k+s-1} \right| dV_{h} (q)\le \sum_{k=1}^{\infty} \int_0^1q^{k-1} dV_{h} (q)  = \int_0^1\frac{dV_{h} (q)}{1-q} \, ,
\end{align}
we have the desired convergence by \eqref{cond}.
Furthermore,
\begin{align}\label{PGFMT}
\lim_{n\to\infty} {\mathcal M} \left[ f_n\left(\frac{t}{n}\right);s  \right] = \Gamma (s)\, \prod_{k=0}^{\infty} \left[1-\int_0^1 q^{k+s} dh (q) \right]
\end{align}
uniformly in $s$ on compact sets in the right half-plane. By interchanging the operations of taking the limit and integration in the Mellin transform inversion formula \begin{align*}%\label{IMT_t}
f_n\left( \frac{t}{n}\right) = \frac{1}{2\pi i } \int^{c+i\infty}_{c-i\infty}  {\mathcal M} \left[ f_n\left( \frac{t}{n}\right);s  \right] \frac{ds}{t^{s}}
\end{align*}
where the integration is along the vertical line $\Re s= c$ inside the characteristic strip,
we can recover  $f_n\left(\frac{t}{n}\right)$ in the scaling limit \eqref{sl} by finding the inverse Mellin transform of the expression on the r.h.s. in \eqref{PGFMT}.    The infinite product in \eqref{PGFMT} is absolutely convergent, and, hence, it may be rearranged as a series,
 \begin{align*}
 \MoveEqLeft
 1-\sum_{0\le k} \int_0^1 q^{k+s} dh (q) + \sum_{0\le k<j} \int_0^1q^{j+s} dh (q) \int_0^1q^{k+s} dh (q) \\ &-  \sum_{0\le k<j<l} \int_0^1q^{j+s} dh(q) \int_0^1q^{k+s} dh (q) \int_0^1q^{l+s} dh (q) +\ldots ,
 \end{align*}
On resumming the arising geometric series in the above expansion, one gets
 \begin{align*}
\prod_{k=0}^{\infty}\! \left[1\!-\!\int_0^1 q^{k+s} dh (q) \right] = {}&
 1 \!-\! \int_0^1 \frac{q_1^s}{1-q_1}  dh(q_1) \!+\! \int_0^1\!\int_0^1 \frac{(q_1q_2)^s}{1-q_1q_2}\frac{q_1}{1-q_1}  dh(q_1)dh(q_2)\\
 &{}  -\int_0^1\!\int_0^1\!\int_0^1  \!\frac{(q_1q_2q_3)^s}{1-q_1q_2q_3}\frac{q_1q_2}{1-q_1q_2}\frac{q_1}{1-q_1}  dh(q_1)dh(q_2)dh(q_2) + \ldots .
  \end{align*}
Now, recalling that
\begin{align*}%\label{imvm}
  {\mathcal M} \left[  \exp \left( -\frac{t}{q^{k}} \right); s\right] =  q^{ks} \, \Gamma (s), \quad \Re s>0 \, ,
\end{align*}
we easily Mellin-invert the r.h.s in \eqref{PGFMT}. This yields
\begin{align*}
\lim_{n\to\infty} f_n\left(\frac{t}{n}\right) = e^{-t} + \sum_{k=1}^{\infty} (-1)^k\int_0^1 \ldots \int_0^1 \frac{e^{-\frac{t}{q_1 \cdots q_k}}}{q_1 \cdots q_k}\, \prod_{j=1}^k \frac{q_1 \cdots q_j}{1- q_1 \cdots q_j} \, dh (q_1) \ldots dh (q_k)
\end{align*}
Finally, on multiplying this equation through by $e^{t}$ to compensate for the factor $(1-\tau)^{n-1}$ in the definition of $f_n(\tau)$, we establish the convergence of $G_{n,\, t/n} [h]$ to $G_t[h]$ given by \eqref{limPGFmr}. This complete our proof of Theorem \ref{thm:7} .
\end{proof}

Consider now the counting function of absolute squares of the eigenvalues of $A_{n,\tau}$
\begin{align}\label{N_tq}
N_{n, \tau}(q) = \card \left( \mathcal{Q}_{n,\tau}\cap [0,q] \right).
\end{align}
As an immediate corollary of Theorem \ref{thm:7} we have convergence in distribution.
\begin{corollary}
\label{prop:9}
In the scaling limit $n\to\infty$, $n\tau=t>0$,
\begin{align}\label{fm8}
\lim_{n\to\infty} {\E}_{U(n)} \big\{(1+x)^{N_{n, {t}/{n}}(q)}\big\} = \sum_{k=0}^{\infty}  \,
\frac{q^{\binom k 2}\, e^{t\left(1-\frac{1}{q^{k}}\right)}     }{(q;q)_k }\,   x^k, \quad \quad q\in (0,1).
\end{align}
\end{corollary}

\begin{proof}  Setting $h$ in \eqref{PGF2mr} to be a multiple of the 0-1 indicator function $\mathbf{1}_{[0,q]}$ of interval $[0,q]$,
\begin{align*}
G_{n,\tau}[x \mathbf{1}_{[0,q]}]={\E}_{U(n)} \big\{(1+x)^{N_{n,\tau}(q)}\big\}.
\end{align*}
Therefore, by Theorem \ref{thm:7}, ${\E}_{U(n)} \{(1+x)^{N_{n,\tau}(q)}\}$ converges
 to the infinite series on the r.h.s in \eqref{limPGFmr} evaluated at $h=x\mathbf{1}_{[0,q]}$. By repeatedly making use of $\int_0^1 f(q^\prime) d\mathbf{1}_{[0,q]} (q^\prime)= - f(q)$ in this series, one transforms it to \eqref{fm8}.
\end{proof}

\begin{proof}[Proof of Theorems \ref{Thm:M}  and \ref{Thm2}]
Theorem \ref{thm:7}  implies that the point process of absolute squares of eigenvalues of $A_{n,\tau}$ converges (in the sense of finite-dimensional distributions) in the scaling limit \eqref{sl}  to the point process defined by the generating functional \eqref{limPGFmr}. As Forrester's and Ipsen's work \cite{FI2019} implies that the limiting process is that of the absolute squares of the Gaussian power series \eqref{phi(z)} conditioned by the event $|\varphi(0)|^2=t$, this proves Theorems \ref{Thm:M} and \ref{Thm2} for $\varphi(0)\not=0$. In Remark \ref{rem3} we proved that equations \eqref{limPGFmr1} and \eqref{fm8i} also hold for $\varphi (0)=0$. This completes our proof of Theorems \ref{Thm:M}  and \ref{Thm2}.
\end{proof}

\subsection{Link to $q$-binomial theorem}
\label{Sec2.2}
Our proof of Theorem \ref{thm:7} obscures a link of our technique to $q$-binomial expansions. Here we  provide an alternative proof of Corollary \ref{prop:9} which makes this link more apparent.

The following expansion \cite{NIST:DLMF}
 \begin{align}\label{q-bin-thm}
\prod_{k=0}^{n-1} (1+x q^{k}) = \sum_{k=0}^{n}
\frac{(q;q)_n}{(q;q)_k (q;q)_{n-k}}\, q^{\binom k 2} x^k
\end{align}
is known as the $q$-binomial theorem. In the limit $q\to1^-$ the $q$-binomial coefficients
 \begin{align} \label{q-bin-c}
\frac{(q;q)_n}{(q;q)_k (q;q)_{n-k}} = \prod_{j=1}^k \frac{1-q^{n-j+1}}{1-q^j}, \quad k=1, \ldots, n,
 \end{align}
converge to the binomial coefficients and \eqref{q-bin-thm} transforms into the binomial theorem. On passing to the limit $n\to \infty$ in \eqref{q-bin-thm}, one obtains one of Euler's partition identities
 \begin{align}\label{euler_p}
\prod_{k=0}^{\infty} (1+x q^k) = \sum_{k=0}^{\infty} \,
\frac{q^{\binom k 2}}{(q;q)_k}\,  x^k.
\end{align}

By making use of the $q$-binomial theorem we obtain

\begin{proposition}
For every $q\in(0,1)$
\begin{align}\label{DNntau}
{\E}_{U(n)} \big\{\!(1+x)^{N_{n,\tau}(q)} \big\}= \sum_{k=0}^{n} \, \theta(q^k-\tau)
\frac{ (q;q)_n }{(q;q)_k (q;q)_{n-k}} \left( \frac{1 - \frac{\tau}{q^k} }{1-\tau} \right)^{\!\!n-1} \!\! \,
q^{\binom k 2} \, x^k\, .
\end{align}
\end{proposition}
\begin{proof}
Consider function $f_n(\tau) = \theta (1-\tau) (1-\tau)^{n-1} {\E}_{U(n)} \{(1+x)^{N_{n,\tau}(q)}\}$.
Its Mellin transform was calculated above and is given by \eqref{MTPGF} with $h =x\mathbf{1}_{[0,q]}$. By the $q$-binomial theorem,
\begin{align}\label{MTN}
{\mathcal M} \left[ f_n(\tau);s  \right] &=
B(n,s) \prod_{k=0}^{n-1} \left(1+x q^{k+s} \right) = B(n,s)\sum_{k=0}^{n}
\frac{(q;q)_n}{(q;q)_k (q;q)_{n-k}} \,  q^{\binom k 2}(xq^s)^k\, .
\end{align}
The Mellin transform in \eqref{MTN} can be inverted by noticing that for $q\in (0,1)$
\begin{align*}
{\mathcal M} \left[ \theta (q-\tau)\left(1-\frac{\tau}{q}\right)^{\!\!n-1}; s  \right] = q^s B(n,s)\, ,
\end{align*}
and we arrive at \eqref{DNntau}.
\end{proof}

To pass to the scaling limit $n\to\infty$, $n\tau=t$ in expression \eqref{DNntau}, we first note that
\begin{align}\label{22a}
{\E}_{U(n)} \big\{\!(1+x)^{N_{n,\tau}(q)}\big\} = \sum_{k=0}^{m(\tau,q)}
\prod_{j=1}^k \frac{1-q^{n-j+1} }{1-q^{j}}\!\left( \frac{1 - \frac{\tau}{q^k} }{1-\tau} \right)^{\!\!n-1} \!\! \!\!q^{\frac{k(k-1)}{2}} \, x^k,
\end{align}
where $m(\tau,q)$ is the (unique) integer $m$ such that $\tau^{\frac{1}{m}} \le q < \tau^{\frac{1}{m+1}}$.
Now,   since $n\tau=t$,  we have that $m(\tau, q)\sim \log n /\log (1/q)$ and, also,
\begin{align*}
\prod_{j=1}^k \frac{1-q^{n-j+1} }{1-q^{j}}\!\left( \frac{1 - \frac{\tau}{q^k} }{1-\tau} \right)^{\!\!n-1} \!\! \!\!q^{\binom k 2} \, x^k \sim
\frac{q^{\binom k 2}\, e^{t\left(1-\frac{1}{q^{k}}\right)}     }{(q;q)_k }\,   x^k\, .
\end{align*}
Since the series on the r.h.s. in \eqref{22a} is bounded by the convergent series
$
\sum_{k=0}^{\infty}   \frac{1 }{(q;q)_k }q^{\binom k 2}  x^k
$
we may pass to the scaling limit in \eqref{22a} and recover \eqref{fm8}.

\subsection{Marginal probability distributions of ordered absolute squares}
\label{Sec2.3}
Equation \eqref{fm8} determines the limiting distribution of  $N_{n, \tau}(q)$ in the scaling limit $n\to\infty$, $n\tau=t$:
\begin{align}\label{fm10}
\lim_{n\to\infty}\Prob\! \left\{ N_{n,\, t/n}(q)=m \right\} =  \delta_{m,0} + \sum_{k=1}^{\infty} (-1)^{k+m} \frac{(k)_m}{m!}\,  \frac{ q^{\binom k 2} \, e^{t\left(1-\frac{1}{q^{k}}\right)} }{(q;q)_k}, \quad m\ge 0,
\end{align}
where $(k)_m $ is the falling factorial: $(k)_0=1$ and $(k)_m=k(k-1)\cdots (k-m+1)$. In turn, this allows us to obtain the limiting distribution of the $m$-th smallest absolute square of the eigenvalues of $A_{n,\tau}$ which we denote by $q_{n,\tau}^{{(m)}}$.
\begin{proposition}
\label{prop:10} For every $m\ge 1$, the distribution of $q_{n,\tau}^{{(m)}} $  in the scaling limit $n\to\infty$, $n\tau =t>0$, is given by
\begin{align}\label{fm12a}
\lim_{n\to\infty} \Prob \big\{q_{n, \,t/n}^{(m)} > q \big\} = 1+  \sum_{k=m}^{\infty} (-1)^{k+m-1}  \binom{k-1}{m-1} \frac{q^{\binom k 2}\, e^{t\left(1-\frac{1}{q^{k}}\right)}  }{(q;q)_k}.
\end{align}
\end{proposition}
\begin{proof} The distribution of $q_{n,\tau}^{{(m)}} $ can be read off from the distribution of the counting function:
\begin{align}\label{fm11b}
\Prob \{ q_{n,\tau}^{{(m)}}  \!>\! q \} = \Prob \{N_{n,\tau}(q)\!=\!0 \} + \Prob\{ N_{n,\tau}(q)\!=\!1 \} \!+\! \ldots \!+\! \Prob\{ N_{n,\tau}(q) \!=\!m\!-\!1 \}.
\end{align}
 It follows from  \eqref{fm10} and \eqref{fm11b} that
\begin{align*}%\label{fm12}
\lim_{n\to\infty} \Prob \big\{ q_{n,\,t/n}^{{(m)}}  > q \big\}= 1+ \sum_{k=1}^{\infty} (-1)^{k} \,  \sum_{l=0}^{m-1} (-1)^l \, \frac{(k)_l}{l!} \, \frac{q^{\binom k 2}\, e^{t\left(1-\frac{1}{q^{k}} \right)}  }{(q;q)_k}.
\end{align*}
Now, \eqref{fm12a} follows from the identity
\begin{align}\label{fm11bb}
\sum_{l=0}^{m-1} (-1)^l \frac{(k)_l}{l!} =
\begin{cases}
\displaystyle{
(-1)^{m-1}\binom{k-1}{m-1}
}& \text{if }  k\ge m; \\
0 & \text{if }  k=1, \ldots m-1
\end{cases}
\end{align}
which can be proved by induction.
\end{proof}

\begin{figure}[t!]
\includegraphics[width=.45\linewidth]{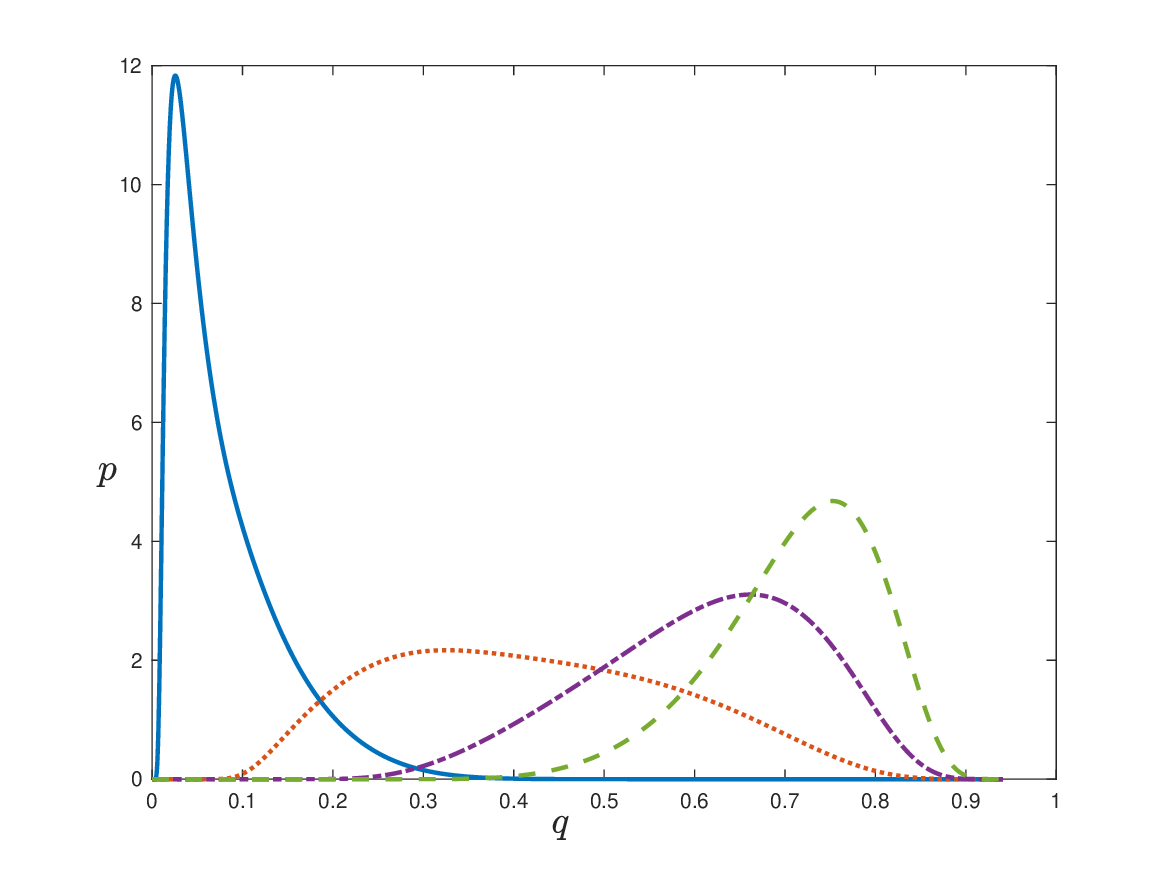}
\includegraphics[width=.45\linewidth]{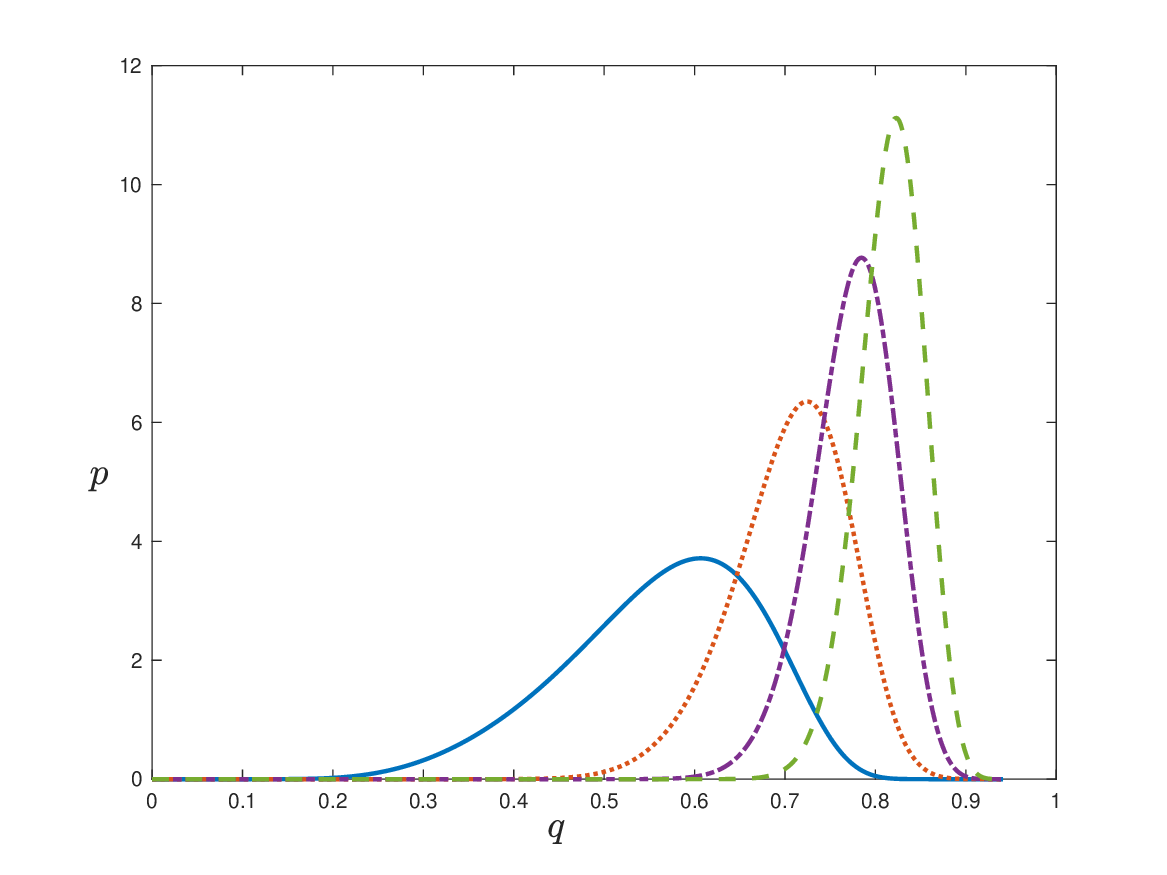}
\caption{Graphs of the probability density of the conditional marginal distribution of the four smallest absolute squares of the zeros of ${\varphi (z)}$ given ${\varphi (0)}=a$:
$q_{{(1)}}$ (solid line), $q_{{(2)}}$ (dotted line), $q_{{(3)}}$ (dash-dotted line) and $q_{{(4)}}$ (dashed line); $|a|^2=0.05$ in the plot on the left and $|a|^2=2$ in the plot on the right.}
\label{Fig:3}
\end{figure}

\begin{remark}
\label{rem:2.5}
 By virtue of Forrester's and Ipsen's work, $q_{n,\tau}^{(m)}$ converges in distribution, in the scaling limit $n\to\infty$, $n\tau=t$,  the $m$-th smallest absolute square $q_{(m)}$ of the zeros of the Gaussian power series ${\varphi (z)}$ conditioned by the event that $|{\varphi (0)}|^2=t$. The probability density of the conditional distribution of $q_{(m)}$ is obtained by differentiation from \eqref{fm12a}. For $m=1$ this function is given by the $q$-series in \eqref{eq:24}, and for $m>1$ a closed form expression for the probability density of $q_{(m)}$  is obtained by multiplying the  $k$-th term of the $q$-series in  \eqref{eq:24} by $(-1)^{m-1}\binom{k-1}{m-1}$. For example, the conditional probability density of $q_{{(2)}}$  given that $\varphi(0)=a$  equals
\begin{align}\label{eq:24_s}
\sum_{k=2}^{\infty} (-1)^{k}\, \frac{(k-1)\,  q^{\binom k 2 -1}\, e^{|a|^2\left(1-\frac{1}{q^{k}}\right)} }{(q;q)_k}
\left[k \left(\frac{|a|^2}{q^{k}} -1 \right) +\sum_{i=1}^k \frac{i}{1-q^{i}}\right].
\end{align}
Fig. \ref{Fig:3} displays graphs of the conditional probability density of the four left-most points of the absolute squares of the zeros of ${\varphi (z)}$ given ${\varphi (0)}=a$ for two values of $|a|^2$.
\end{remark}

\section{Limiting distributions of $\mathcal{Q}_a$}
\label{Sec:3}

Recall that $\mathcal{Q}_a$ is the random set of absolute squares of the zeros of $\varphi_a(z)= a+\sum_{k\ge 1}c_kz^k$. It is useful to think of the probability law of $\mathcal{Q}_a$ as the conditional law of the point process $\mathcal{Q}$  of the absolute squares of the zeros of the GAF ${\varphi (z)}=\sum_{k\ge 0}c_kz^k$ given ${\varphi (0)}=a$,
\begin{align*}
\mathcal{Q}_a = \left. \mathcal{Q}\, \right| \, {\varphi (0)}=a.
\end{align*}
 The p.g.fl of the conditional law is given in terms of an infinite power series in Theorem \ref{Thm:M}. Using this series one can easily obtain the conditional probability distribution of the counting function $N_q=\card(\mathcal{Q} \cap [0,q])$ for every $q\in(0,1)$, see Theorem \ref{Thm2}. However, for a given value of $a$ getting information about the correlation structure of $\mathcal{Q}_a$ looks less straightforward. In this section we will identify various scaling limits in which the probability law of $\mathcal{Q}_a$ takes familiar forms.

\subsection{Limits of small values of $a$} It is evident that on removing the trivial zero $z=0$ from the zero set of $\sum_{k\ge 1}c_kz^k$ one obtains a point process that is statistically indistinguishable from the zero set of the unconditioned power series ${\varphi (z)}$ . The latter point process is well understood.

Our next proposition asserts that the probability law of $\mathcal{Q}_a$ for small values of $a$ can be approximated by that of the point process $\{0\} \cup \mathcal{Q}$, although as usual, some care needs to be taken when dealing with densities, see e.g., Remark \ref{rem:14}.

\begin{proposition}\label{Prop:11}
For every test function $h(q)$ satisfying the assumptions of Theorem \ref{Thm:M},
\begin{align}\label{PGF_Lim_t_0}
\lim_{a\to 0} \E_{\varphi}  \Big\{ {\prod\nolimits_{q\in \mathcal{Q}}} \!(1+h (q)) \Big|   {\varphi (0)}=a \Big\} = (1+ h(0))  {\prod\nolimits_{k\ge 1}} \Big( \! 1- {\int_0^1} q^k dh (q)\!\Big),
\end{align}
where the infinite product on the r.h.s. is the unconditional p.g.fl $\E_{\varphi} \{ \prod\nolimits_{q\in \mathcal{Q}} \!\left(1+h (q)\right)\}$.
\end{proposition}

\begin{proof} The conditional p.g.fl on the l.h.s in \eqref{PGF_Lim_t_0} is given by the series  \eqref{limPGFmr1}. It is evident that this series, for every $a\in \mathbb{C}$,  is bounded in absolute value by the converging series
\begin{align*}
1+ \sum_{k=1}^{\infty} \int_0^1 \ldots \int_0^1 \frac{1}{q_1 \cdots q_k}\, \prod_{j=1}^k \frac{q_1 \cdots q_j}{1- q_1 \cdots q_j} \, d V_{h} (q_1) \ldots dV_{h}(q_k)
= \prod_{k=0}^{\infty} \Big(\! 1+ \int_0^1 q^k dV_{h} (q)\!\Big),
\end{align*}
Therefore we may take the limit $a\to 0$ inside the summation and integration in \eqref{limPGFmr1}. The limiting series is
\begin{align*}
\MoveEqLeft[10]
1+ \sum_{k=1}^{\infty} (-1)^k \int_0^1 \ldots \int_0^1 \frac{1}{q_1 \cdots q_k}\, \prod_{j=1}^k \frac{q_1 \cdots q_j}{1- q_1 \cdots q_j} \, d h (q_1) \ldots d h (q_k)
\\ & = \prod_{k=0}^{\infty} \Big( \! 1- \int_0^1 q^k dh (q)\!\Big) = (1+ h(0))
\prod_{k=1}^{\infty} \Big( \! 1- \int_0^1 q^k dh (q)\!\Big),
\end{align*}
as claimed in \eqref{PGF_Lim_t_0}.
\end{proof}

\begin{corollary}
\label{cor_a=0}
\phantom{a}
\begin{itemize}
\item[(a)]
In the limit $a\to 0$, $\left. N_q\, \right| {\varphi (0)}=a$ converges in distribution to $1+N_q$, where
\begin{align}\label{fm10tmr}
\Prob \{N_q\!=\!m \} =  \delta_{m,0}+ \sum_{k=m}^{\infty} (-1)^{k-m} \binom{k}{m} \frac{ q^{k(k+1)/2} }{(q;q)_k}, \quad m\ge 0.
\end{align}
\item[(b)] The conditional joint distribution of the first $m$ left-most points $(q_{(1)}, q_{(2)}, \ldots, q_{(m)})$ of $\mathcal{Q}$ given that ${\varphi (0)}=a$ converges, in the limit $a\to 0$, to that of  $(0, q_{(1)}, \ldots, q_{(m-1)})$ whose marginal distributions are given by
\begin{align}\label{fm12aT}
\Prob \{q_{(m)}\!>\!q \} =
\displaystyle{
1+ \sum_{k=m}^{\infty} (-1)^{k+m-1}  \binom{k-1}{m-1}\, \frac{q^{\frac{k(k+1)}{2}} }{(q;q)_k}
}, \quad q\in (0,1),\,\, m\ge 1.
\end{align}
\end{itemize}
\end{corollary}
\begin{proof}
The convergence in distribution follows from Proposition \ref{Prop:11}. Equation \eqref{fm10tmr} follows from the product representation of the moment generating function of $N_q$ \eqref{pgfNq} via Euler's partition identity \eqref{euler_p}, and \eqref{fm12aT} follows from  \eqref{fm10tmr}, \eqref{fm11b} and \eqref{fm11bb}.
\end{proof}

\begin{remark}
\label{rem:14} Although $q_{(2)}$ given ${\varphi (0)}=a$ converges in distribution to $q_{(1)}$ in the limit of small values of $a$, at the level of densities this convergence holds only on compact sets in $(0,1)$, and not at $q=0$. Indeed, the probability density of $q_{(1)}$,
\begin{align}\label{fm15mr}
p_{q_{(1)}} \!(q)  =
\sum_{k=1}^{\infty}  \frac{(-1)^{k+1} \, q^{\frac{k(k+1)}{2} -1}}{(q;q)_k}
 \sum_{j=1}^k
\frac{j}{1-q^{j}},
\end{align}
takes value 1 at $q=0$, whilst, as evident from \eqref{eq:24_s}, the density of the conditional distribution of $q_{(2)}$ given ${\varphi (0)}=a$ has an essential singularity at $q=0$ for every fixed $a\not=0$, see Fig \ref{Fig:4} for illustration.
\end{remark}
\begin{figure}[t!]
\includegraphics[width=.45\linewidth]{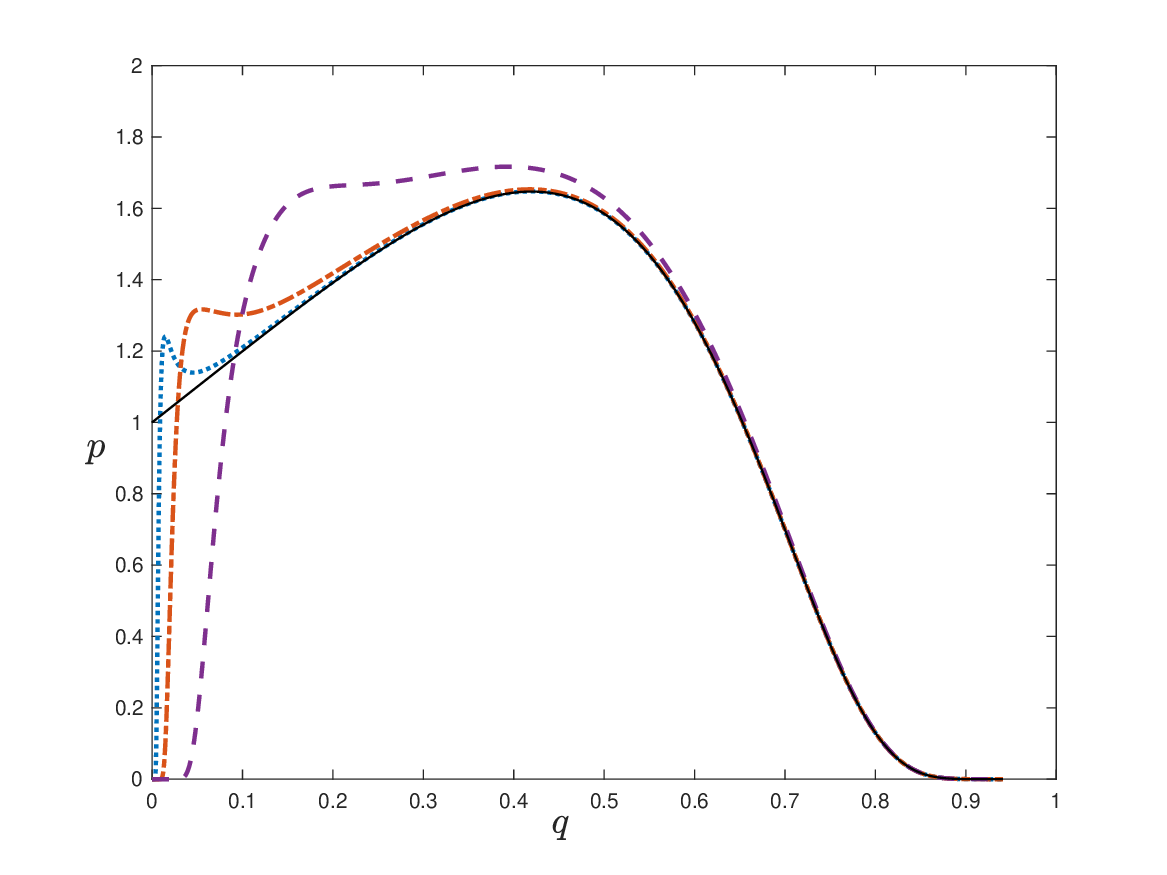}
\caption{Plots of the density of the conditional probability distribution of the second smallest absolute square of the zeros of ${\varphi (z)}=\sum_{k\ge 0} c_kz^k$ given that ${\varphi (0)}=a$  for $|a|^2=0.01$(dashed line), $|a|^2=0.001$ (dash-dotted line) and $|a|^2=0.0001$ (dotted line), and of the probability density of the smallest absolute square of the zeros of the unconditioned series ${\varphi (z)}$ (solid line).}
\label{Fig:4}
\end{figure}

Before proving Theorem \ref{Thm:FG}a we would like to explain the origin of the associated scaling limit. The Corollary \ref{cor_a=0} asserts that given ${\varphi (0)}=a$ and in the limit $a\to 0$,  the smallest absolute square $q_{(1)}$ `escapes'  from the crowd of larger absolute squares and moves towards zero.  To determine the scale of its deviations from zero, it is instructive to examine the conditional intensity function
\begin{align} \label{p_t1mr}
\rho_a (q) =  \frac{d}{dq} \E_{\varphi} \{  N_q | {\varphi (0)}\!=\!a \}=\frac{ q^2 + |a|^2 (1-q)}{q^2(1-q)^2} \, e^{|a|^2 \left( 1-\frac{1}{q}\right)}.
\end{align}
From Corollary \ref{cor_a=0}, we expect that the intensity develops a $\delta$-function peak as $a$ gets smaller and smaller, and that $\rho_a (q)$ approximates the probability density of $q_{(1)}$ around the peak, see Fig. \ref{Fig:1bis} for illustration. The position of the peak indicates the location of the smallest absolute square which can be determined from the equation $\frac{d}{dq} \rho_a (q) =0$. For $q\in(0,1)$,  this equation is equivalent to
\begin{align*}
{2 q^4-4 q^3  |a|^2+q^2  |a|^4+6q^2  |a|^2- 2 q  |a|^4 - 2 q  |a|^2+ |a|^4}
=0
\end{align*}
from which it is evident that in the limit of small values of $a$ the intensity $\rho_a(q)$ has a peak at $q=|a|^2/2$ of height $4e^{-2}/|a|^2$. Thus in the limit $a\to 0$ the smallest absolute square scales with $|a|^2$.

\begin{figure}[!htbp]
\includegraphics[width=.45\linewidth]{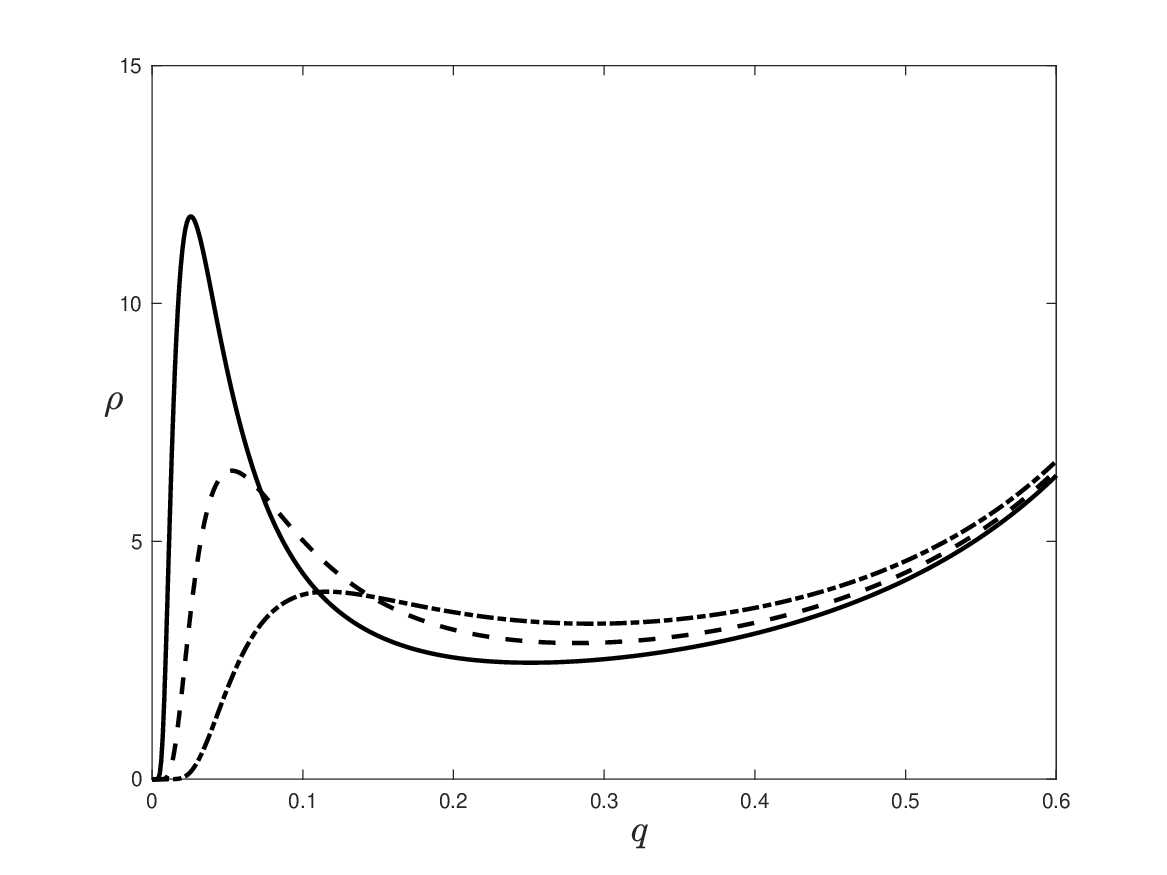}
\includegraphics[width=.45\linewidth]{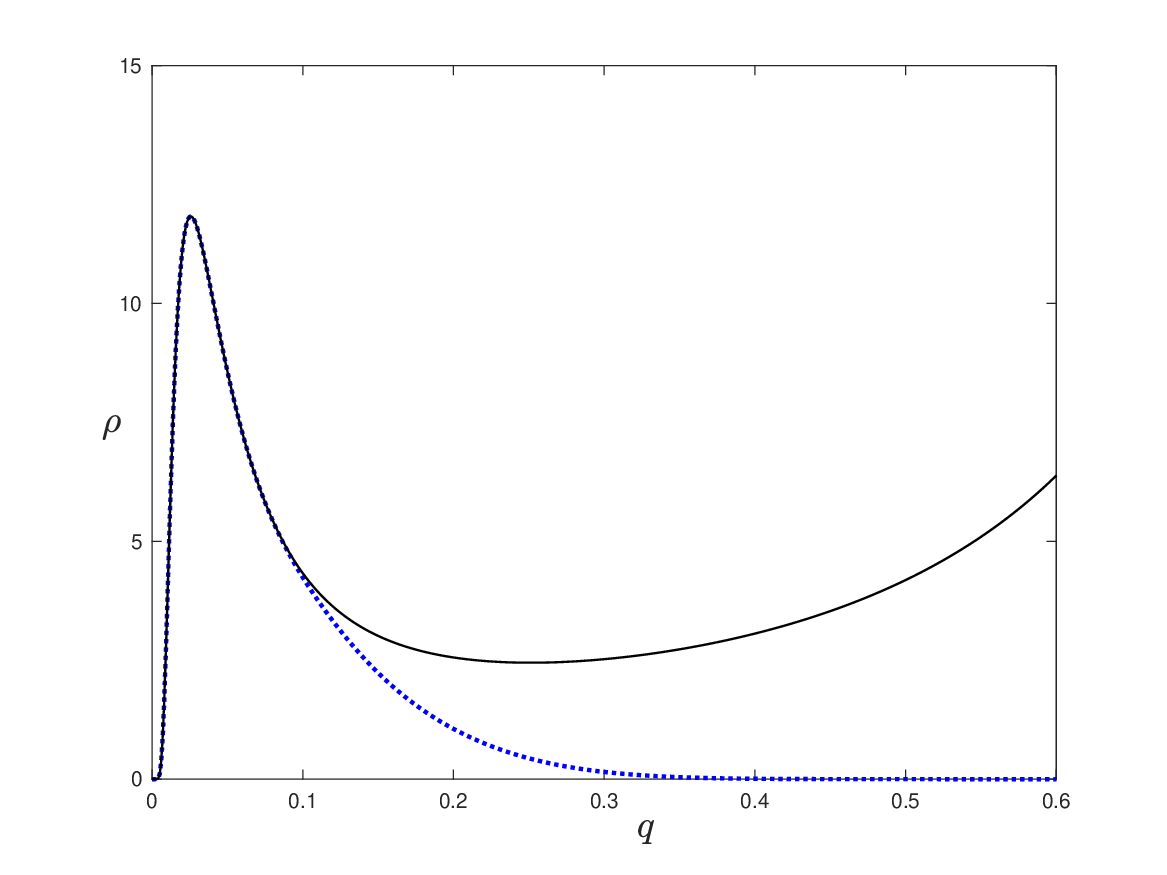}\\
\caption{Intensity $\rho_a(q)$ of the absolute squares of the zeros of ${\varphi (z)}=\sum_{k\ge 0} c_kz^k$ conditioned by ${\varphi (0)}=a$. The plot on the left depicts $\rho_a (q)$ for $|a|^2=0.05$ (solid line), $|a|^2=0.1$ (dashed line) and $|a|^2=0.2$ (dot-dashed line). The plot on the right depicts $\rho_a (q)$  (solid line) and the conditional density \eqref{eq:24} of the smallest absolute square (blue dotted line) for $|a|^2=0.05$.}
\label{Fig:1bis}
\end{figure}

\begin{proof}[Proof of Theorem \ref{Thm:FG}a]
By \eqref{fm8i},
\begin{align*}
{\E}_{\varphi} \big\{ (1\!+\!x)^{N_{ut}} \big| |{\varphi (0)}|^2\!=\!t  \big\} = 1+ \frac{x e^{t-\frac{1}{u}}}{1-ut}  +
 \sum_{k=2}^{\infty} \frac{(ut)^{\frac{k(k-1)}{2}} \, e^{t-\frac{1}{u^{k} t^{k-1}}}}{(1-ut) \cdots  (1-(ut)^{k})} \,  x^k \, .
\end{align*}
Evidently, in the scaling limit $t\to 0^{+}$, $q=ut$,
\begin{align*}
{\E}_{\varphi} \big\{ (1\!+\!x)^{N_{q}} \big| |{\varphi (0)}|^2\!=\!t    \big\} \to 1+ x e^{-\frac{1}{u}}
\end{align*}
for every $x$, and Theorem \ref{Thm:FG}a follows.
\end{proof}

It follows from Theorem \ref{Thm:FG}a and that the conditional distribution of the scaled smallest absolute square $q_{(1)}/ |a|^2$ given ${\varphi (0)}=a$ converges in the limit of small values of $a$ to the standard Fr\'echet distribution with shape parameter 1. We have
\begin{proposition} Write the $m$-th smallest absolute square as $q_{(m)} = |a|^2 u_{(m)}$. Then
\begin{align*}
\lim_{a\to 0} \Prob \{ u_{(m)} \!>\! u \, | \,  {\varphi (0)}\!=\!a \}=
\begin{cases}
1-e^{-\frac{1}{u}}, & \text{when } m=1; \\
1 & \text{when } m\ge 2;
\end{cases}
\end{align*}
\end{proposition}
\begin{proof} This proposition is an immediate consequence of Theorem \ref{Thm:FG}a and relation \eqref{fm11b}.
\end{proof}

\subsection{Scaling limit of large values of $a$} Now we turn to the opposite limit of large values of $a$. It is an immediate consequence of \eqref{fm8i} that for every fixed value of $q<1$, the probability for the disk $|z|^2\le q$ to contain zeros of the Gaussian power series $\varphi_a(z)=a+\sum_{k\ge 1} c_kz^k$ tends to zero in the limit $a\to+\infty$. An interesting question then arises about how fast $q$ should tend to 1 simultaneously with the increase of $a$ for the interval $[0,q]$ to catch up with the left tail of the set  $\mathcal{Q}_a$ of absolute squares of the zeros of $\varphi_a(z)$. Such a scale for $q$ is the natural habitat for the left-most extremes of $\mathcal{Q}_a$ in the limit of large $a$ and it can be determined by requiring that
\begin{align*}%\label{anz}
{\E}_{\varphi} \{N_q \, | \, {\varphi (0)}=a \}= \frac{e^{-|a|^2\frac{1-q}{q}}}{1-q}
\end{align*}
remains non-zero and finite as $a\to\infty$. This happens if one chooses
 \begin{align*}
q=1- \frac{\log(|a|^2) - \log\log(|a|^2) + v}{|a|^2}. %,\quad \quad t\to+\infty,
\end{align*}
This explains the origin of the scaling limit \eqref{sl_p} in Theorem \ref{Thm:FG}b.

\begin{proof}[Proof of Theorem \ref{Thm:FG}b]
From \eqref{fm8i} one finds the conditional factorial moments of $N_q$,
\begin{align}\label{fm8b}
\E_{\varphi} \big\{ N_q(N_q\!-\!1)\cdots (N_q\!-\!k\!+\!1) \big| \varphi (0)=a  \big\}
=  \frac{k!\, q^{\binom k 2}\, e^{|a|^2( 1- {q^{-k}})} }{(q;q)_k}.
\end{align}
It is straightforward to verify that in the scaling limit \eqref{sl_p} the expression on the r.h.s., for every $k \ge 0$, tends to $\exp(-kv)$, the factorial moments of Poisson($e^{-v}$).
\end{proof}

Turning to the conditional distribution of the $m$-th smallest absolute square $q_{(m)}$ of the zeros of ${\varphi (z)}=\sum_{k\ge 0} c_k z^k$ for large values of $c_0=a$, define random variables $v_t^{\sm{(k)}}$ by  the transformation
\begin{align*}
q_{(m)} = 1-\frac{\log(|a|^2) -\log\log(|a|^2) + v_{(m)}}{|a|^2}\, .
\end{align*}
It is an immediate consequence of Theorem \ref{Thm:FG}b and equation \eqref{fm11b} that
\begin{align*}
\lim_{a\to \infty} \Prob \{ v_{(m)} \!\le\!  v\, |\, {\varphi (0)}\!=\!a \}  = \frac{\Gamma (m, e^{-v})}{\Gamma (m)},
\end{align*}
where $ \Gamma (m, x)= \int_{x}^{+\infty} \lambda^{m-1} e^{-\lambda}\, d\lambda$ is the upper incomplete Gamma function. More generally,

\begin{theorem} \label{Prop4:13} Consider the rescaled point process $\mathcal{V}_a$ defined by the transformation
\begin{align*}
\mathcal{Q}_a=1- \frac{\log t - \log\log t + \mathcal{V}_a }{t}\,, \quad t=|a|^2,
\end{align*}
where $\mathcal{Q}_a$ is the set  of absolute squares of the zeros of $\varphi_a(z)=a+\sum_{k\ge 1} c_k z^k$. The rescaled process $\mathcal{V}_a$  converges in distribution, as   $a\to\infty$, to Poisson point process on $\mathbb{R}$ with intensity $\lambda (v)= e^{-v}$.
\end{theorem}

\begin{proof} It will suffice to prove convergence of the multipoint factorial moments of the counting function $\widetilde N_v^{(a)} = \card \left( \mathcal{V}_a \cap [v, +\infty)  \right)$
 to those of the Poisson process. To simplify the notations, denote by $N_q^{(a)}$ the counting fucntion of $\mathcal{Q}_a$, $N_q^{(a)}= \card \left( \mathcal{Q}_a \cap [0,q])  \right)$. Evidently, $\widetilde N_v^{(a)}= N_q^{(a)}$ with $v$ and $q$ being related as in \eqref{sl_p}.

 In Appendix \ref{appendix:A} we prove, following a similar calculation in \cite{PV04}, that
\begin{align*}%\label{n_sum}
\E_{\varphi} \left\{  N_q^{(a)}( N_q^{(a)}-1) \cdots (N_q^{(a)}-k+1)\right\} =
e^{t\left(1-\frac{1}{q_1\cdots q_k}\right)}\sum_{\sigma\in S_k} \prod_{\nu\in \sigma } (-1)^{|\nu|-1} \frac{1}{1-q_{\nu}}\, , \
\end{align*}
where the sum is over all permutations of $k$ indices, the product is over all disjoint cycles $\nu$ of permutation $\sigma$, $|\nu|$ is the length of $\nu$ and $q_{\nu}=\prod_{j\in \nu} q_j$. On substituting here $1- (\log t - \log\log t + v_j)/t$ for $q_j$, we find that for every cycle $\nu$
\begin{align*}
\frac{1}{1-q_{\nu}} = \frac{1}{|\nu|} \frac{t}{\log t} \left( 1+ O\left( \frac{\log\log t}{\log t }\right)\right),
\end{align*}
and, hence, the leading form of the above sum over permutations is given by the contribution from the identity permutation as it contains $k$ cycles, the maximal number possible:
\begin{align}\label{n_sum1}
\sum_{\sigma\in S_k} \prod_{\nu\in \sigma } (-1)^{|\nu|-1} \frac{1}{1-q_{\nu}} = \left( \frac{t}{\log t}\right)^k \left( 1+ O\left( \frac{\log\log t}{\log t }\right)\right).
\end{align}
As for the exponential factor, its leading order form is given by
\begin{align}\label{n_sum2}
\exp\left[t\left(1-\frac{1}{q_1\cdots q_k}\right) \right] \sim \left(\frac{\log t}{t}\right)^k \exp \left(-\sum_{j=1}^k v_j\right)
\end{align}
On multiplying \eqref{n_sum1} and \eqref{n_sum2}, and recalling that $\widetilde N_v^{(a)}= N_q^{(a)}$,  we conclude that for every $k$ and every $(v_1, \ldots, v_k)$ in $\mathbb{R}^k$,
\begin{align*}
\lim_{a\to\infty}\E_{\varphi} \left\{ \widetilde N_{v_1}^{(a)}(\widetilde N_{v_2}^{(a)}-1) \cdots (\widetilde N_{v_k}^{(a)}-k+1)\right\}  = \exp \left(-\sum_{j=1}^k v_j\right).
\end{align*}
The exponential on the r.h.s. is the multipoint factorial moment of the corresponding counting function of the Poisson point process with intensity $\lambda (v) = e^{-v}$.
\end{proof}

 \medskip

\section{Asymptotic analysis of moment generating function}
\label{Sec:4}

In this section we develop asymptotics of the $q$-series
\begin{align} \label{Fext}
    F(x,t;q):=\sum_{k=0}^\infty\frac{(-x)^kq^{\binom k2}}{(q;q)_k}\, e^{t(1-q^{-k})}\, , \quad 0 < q < 1, %, t\ge 0.
   \end{align}
in the limit $q\to 1^{-}$.
For every $q\in (0,1)$ this series converges absolutely and uniformly in $x$ on compact subsets of $\mathbb{C}$ and in $t \ge 0$.

When $t=0$, \eqref{Fext} reduces to the $q$-exponential
\begin{align}\label{E_q}
E_q(-x) = \sum_{k=0}^{\infty} \frac{q^{\binom k 2}(-x)^k}{(q;q)_k} =(x;q)_{\infty}
\end{align}
which is one of the two commonly used $q$-deformations of the usual exponential function, the other one being $e_q(x)=1/(x;q)_{\infty}$.
The $q$-exponentials are well known special functions
but the deformation \eqref{Fext} appears to be new. In the context of our work
\begin{align}\label{cmgf}
\E_{\varphi} \big\{ (1\!-\!x)^{N_q} \, \big| \, |\varphi(0)|^2=t \big\} = F(x,t;q).
\end{align}
The unconditional moment generating function can also be expressed via  $F(x,t;q)$,
\begin{align}\label{mgfa}
\E_{\varphi} \big\{ (1-x)^{N_q}\big\} = \frac{F(x,0;q)}{1-x}
=F(xq,0;q).
\end{align}
The factor $(1-x)^{-1}$ here accounts for the emergence of zero of ${\varphi_a (z)}=a + \sum_{k\ge 1} c_k z^k $ at the origin in the limit $a\to 0$.

Asymptotic expansions for $F(x,0;q)$
in the limit $q \to 1^{-}$ are well studied and can be obtained from the product representation $F(x,0;q)=(x;q)_{\infty}$. For example, for $|x|<1$,
\begin{align}\label{qdilog}
-\log F(x,0;q) =  \sum_{k=0}^{\infty} \log (1-xq^k) =   \sum_{k=1}^{\infty} \frac{x^k}{k(1-q^k)} \:.
\end{align}
On a formal level, when $q$ tends to 1, the sum on the right, which is known as quantum dilogarithm, can be replaced by $(1-q)^{-1} {\Li}_2 (x)$
where ${\Li}_2(x)$ is the dilogarithm function
\begin{align}\label{dilog}
{\Li}_2 (x)  = \sum_{k=1}^{\infty} \frac{x^k}{k^2}  = -\int_0^x \frac{\log(1-u)}{u}\, du.
\end{align}
This gives
\begin{align*}
\log F(x,0;q)  \sim - \frac{{\Li}_2(x)}{1-q}\:.
\end{align*}
More precise treatments of sums in \eqref{qdilog} leads to asymptotic expansions of $F(x,0;q)$ in rising powers of $\varepsilon = -\log q $. Such expansions were first obtained by Ramanujan \cite[Ch. 27, Entries 6-7]{RIV} and later by others, see, e.g., \cite{Moak1984,D1994,UN1994,K1995,Prellberg95}. For future reference we need the first two terms of the expansion based on the Euler-Maclaurin summation formula. This expansion holds in the complex $x$-plane  cut along the real axis from $1$ to $+\infty$ \cite{Prellberg95}.
\begin{lemma}[\cite{Prellberg95, K1995}]
\label{lem_qPochas}
Let $\varepsilon=-\log q$. For every $q\in(0,1)$ and $x\in \mathbb{C}\backslash [1,+\infty)$
\begin{align}
\log(x;q)_\infty = -\frac{1}{\varepsilon}{\Li}_2(x)+\frac{1}{2}\ln(1\!-\!x)
+  \varepsilon R_{1}(x,q)
\label{eq_asy_qprod}
\end{align}
where ${\Li}_2(x)$ is the principal branch of the dilogarithm function defined by the integral in \eqref{dilog} and
\begin{align}\label{R1}
|R_{1}(x,q)| \le \frac{1}{6} \int_0^1 \frac{|x| du}{|1-xu|^2} \:.
\end{align}
\end{lemma}

\subsection{Integral representation}
For  $t\ne 0$ the series $F(x,t;q)$ is no longer a $q$-product and one needs a different strategy for the development of its asymptotic expansions. We use a contour integral representation of $F(x,t;q)$ and its saddle point approximations. An alternative approach based on a functional-differential equation for $F(x,t;q)$ is discussed in Appendix \ref{Appendix DEs} .

\begin{lemma}
\label{lemma2}
For every $q\in (0,1)$, $t\ge 0$ and $x\in \mathbb{C}\backslash\{0\}$,
\begin{align}\label{Fcont}
F(x,t;q)=\frac{(q;q)_\infty}{2\pi i}\oint_{\mathcal{C}}\frac{e^{-\log x\log z/\log q}}{(z;q)_\infty}e^{t(1-z)}dz
\end{align}
where ${\mathcal{C}}$ is a clockwise Hankel contour around the positive real axis from $\infty-i0$ to $\infty+i0$
that can be deformed to run along the straight line $(\rho-i\infty,\rho+i\infty)$ for $0<\rho<1$.
\end{lemma}
\begin{proof}
Equation \eqref{Fcont} is a simple extension to $t>0$ a contour integral for $F(x,0;q)$ given in \cite[Lemma 2.1]{Prellberg95}. It is derived from recognising that the residues of the poles of $1/(z;q)_{\infty}$,
    \begin{align*}
        \Res_{z=q^{-k}}  \frac{1}{(z;q)_{\infty}} =\frac{(-1)^{k+1}q^{\binom k2}}{(q;q)_k(q;q)_\infty}\:,
    \end{align*}
 contain much of the structure of the $q$-series (\ref{Fext}), and we can write
 \begin{align*}
F(x,t;q) =-(q;q)_\infty\sum_{k=0}^\infty\mathrm{Res}_{z=q^{-k}} \frac{x^{-\log z/\log q}e^{t(1-z)}}{(z;q)_\infty}\, .
\end{align*}
The result follows by writing the individual residues in terms of contours encircling the poles at $z=q^{-k}$ and
deforming these to a joint Hankel contour, absorbing the minus sign in a change of orientation.
\end{proof}

We now apply the approximation of Lemma \ref{lem_qPochas} to the integral in Lemma \ref{lemma2}.

\begin{lemma}
\label{Lemma:intrep}
\label{contour_t}
Let $\varepsilon = - \log q$. Then for every $x\in \mathbb{C}\backslash\{0\}$ and $t\ge 0$
\begin{align}\label{int_rep}
F(x,t;q)\!\mathop{=}_{\,\,\,\, q\to 1^{-}}
\!\frac{(q;q)_\infty}{2\pi i}\!\!\int_{\rho-i\infty}^{\rho+i\infty}
\!\!\!\exp\left[{\frac{\log x\log z\!+\!{\Li}_2(z)}{\varepsilon}\!+\!t(1\!-\!z)}\right]
\frac{dz}{\sqrt{1\!-\!z}}
\!\times \!\big(1\!+\!O(\varepsilon)\big)
\end{align}
for any $\rho \in (0,1)$.
\end{lemma}

\begin{proof}
This follows directly from Lemma \ref{lemma2}, noting that  $R_1(x,q)$ \eqref{eq_asy_qprod} remains uniformly bounded along the contour of integration.
\end{proof}

The function $F(x,t;q)$ can be evaluated asymptotically when $q$ is close to 1 by using the saddle-point method and asymptotic expansion for $(q;q)_{\infty}$  of which we only need several first terms \cite{Watson1936},
\begin{align}\label{eq_qqas}
\log (q;q)_{\infty} = -\frac{\pi^2}{6 \varepsilon} + \frac 12 \log \frac{2\pi}{\varepsilon}  + O(\varepsilon)\:, \quad \varepsilon=-\log q\to 0^+\;.
\end{align}
The asymptotic expansion for $(q;q)_{\infty}$ is a classical result in number theory and can be obtained from the transformation formula for the Dedekind eta-function, see e.g., \cite[pp 47--48]{Apostol1976}.

\subsection{Central Limit Theorem and Large Deviations.} We first consider $F(x,t;q)$ for a fixed $t \!>\! 0$.
\begin{theorem}
\label{afixed_t}
For every
$x\in \mathbb{C}\backslash \big((-\infty, 0] \cup [1, +\infty)\big)$
and $t\ge 0$,
\begin{align}\label{F1}
F(x,t;q)=\exp \left[
{-\frac 1 \varepsilon\, {\Li}_2(x)+\frac{1}{2} \log (1-x) + t x + O(\varepsilon)}
\right]
\;, \quad \varepsilon=-\log q \to 0^+
\end{align}
Equation \eqref{F1} also holds true as $x$ approaches zero at an angle $\phi $ with the positive real axis provided that $\phi \not= \pi$.
\end{theorem}

\begin{proof} We use the integral representation of $F(x,t;q)$ of Lemma \ref{contour_t}. This is a straightforward application of the saddle-point method. The integrand in \eqref{int_rep}  is analytic in the complex $z$-plane and has branch cuts $(-\infty,0]$ and $[1,+\infty)$ along the real axis. The function $f(z)=\log x\log z\!+\!{\Li}_2(z)$ has a single saddle at $z=1-x$. If $x\notin (-\infty, 0] \cup [1, +\infty)$ then the saddle is off the branch cuts and the straightforward application of the saddle point method leads to
\begin{align}\label{t}
F(x,t;q)=(q;q)_\infty(\varepsilon/2\pi)^{1/2}(1-x)^{1/2}e^{tx}e^{\frac1\varepsilon(\Li_2(1-x)+\log x\log(1-x))}
\times\left(1+O(\varepsilon)\right)\;.
\end{align}
Replacing $(q;q)_\infty$ by its asymptotic form \eqref{eq_qqas} and recalling that
\begin{align*}
{\Li}_2(x)+{\Li}_2(1-x)=\frac{\pi^2}{6}-\log x\log(1-x), \quad x\notin (-\infty, 0] \cup [1, +\infty),
\end{align*}
transforms  \eqref{t} to equation \eqref{F1}.

When $x\to 0$ the saddle $z=1-x$ approaches the branch point at $z=1$ and the above calculation breaks down. In this case we use the substitution $z=1-sx$ to transform \eqref{int_rep} to
\begin{align*}
F(x,t;q)=\frac{(q;q)_\infty}{2\pi i}\sqrt{x} \int_{\tilde{\mathcal C}}
\exp\left[
{\frac{\log x\log(1-sx)+\Li_2(1-sx)}{\varepsilon}+tsx}
\right]
\frac{ds}{\sqrt s} \times\left(1+O(\varepsilon)\right).
\end{align*}
The integrand now is an analytic function in the complex $s$-plane with two branch cuts
\begin{align*}
C_1=\{s=re^{i(\pi-\arg(x))}: \,\, r\ge 0\}\quad \text{and} \quad  C_2=\{s=re^{-i \arg(x) }: \,\, r\ge |x|^{-1}\}
\end{align*}
and has the saddle at $s=1$. If $|x|<1$ and $|\arg (x)|<\pi$ then the saddle is off the branch cuts and the saddle-point method again leads to \eqref{F1}.
\end{proof}

\begin{remark}
\label{remark:4.6}
Equation \eqref{F1} also holds for $x$ on the real negative axis. This can be shown by exploiting the fact that the coefficients of the series in power of $x$ are positive in this case. However, the asymptotic expansion in \eqref{F1} breaks down when $x\to1$. We will return to this matter later.
\end{remark}

\smallskip
Theorem \ref{thm:1.9} which asserts that suitably normalised counting function has normal distribution in the limit when $q\to 1^{-}$ is a simple corollary of Theorem \ref{afixed_t}.

\begin{proof}[Proof of Theorem \ref{thm:1.9}.] Consider the characteristic function $\E_{\varphi} \!\big\{ \exp({-x\widetilde N_q}) \big| \varphi (0)\!=\!a\big\}$ of the conditional distribution of
\begin{align*}
\widetilde N_q={(N_q-\mu_q(a))}/{\sigma_q(a)}
\end{align*}
with $\mu_q(a)$ and $\sigma_q(a)$ being as defined in equations \eqref{muqa} and \eqref{sigmaqa}. Evidently,
\begin{align*}
\E_{\varphi} \!\big\{ \exp({-x\widetilde N_q}) \big| \varphi (0)\!=\!a\big\}\! =\! e^{x{\mu_q(a)}/{\sigma_q(a)}} \E_{\varphi}\! \big\{ (1\!-\!x_q)^{N_q} \big| \varphi (0)\!=\!a\big\} \!=\! e^{x{\mu_q(a)}/{\sigma_q(a)}} F(x_q, |a|^2;q),
\end{align*}
where
\begin{align*}
x_q=1-\exp(-x/\sigma_q) \mathop{=}_{\,\,\, q\to 1^{-}}
 \sqrt{1-q}\, \big(\sqrt{2}\, x - x^2\sqrt{1-q} +O(1-q)\big).
\end{align*}
From Theorem \ref{afixed_t}, by expanding the dilogarithm in powers of $x_q$, we find that
\begin{align*}
F(x_q, |a|^2;q)  = \exp\left[
{-\frac{x_q}{1\!-\!q} - \frac{x_q^2}{4(1\!-\!q)} +O(1\!-\!q)}
\right]
=  \exp\left[
{-\frac{\sqrt{2} \, x }{\sqrt{1\!-\!q}} +\frac{x^2}{2} +O(\sqrt{1\!-\!q})}
\right].
\end{align*}
Also,
\begin{align*}
\exp \left[x\, \frac{\mu_q(a)}{\sigma_q(a)}\right]= \exp \left[\frac{\sqrt{2}\, (1+ O(1-q))}{\sqrt{1-q}}\right].
\end{align*}
Multiplying through these two asymptotic relations we find that for every $a$
\begin{align*}
\lim_{q\to 1^-}\E_{\varphi} \!\big\{ \exp({-x\widetilde N_q}) \big| \varphi (0)\!=\!a\big\} = e^{x^2/2}.
\end{align*}

By using Theorem \ref{afixed_t} one can also obtain the asymptotic behaviour of tails of the (conditional) distribution of $N_q$ around its expected value. For notational simplicity, it is preferable to use the leading asymptotic form of $\mu_q(a)$ (and also of $\sigma_q(a)$). The calculation presented in the proof of Theorem \ref{thm:1.9} works equally well if $\mu_q(a)$ and $\sigma_q(a)$ are replaced by the leading asymptotics and we find that
\begin{align*}
\lim_{q\to 1^{-}}\Prob \left\{ \left|N_q\!-\! \frac{1}{1\!-\!q}\right| \! \ge\!
\frac{\xi}{\sqrt{1\!-\!q}}  \,  \Big| \,  \varphi(0)\!=\!a \right\} = \frac{1}{\sqrt{\pi}}\int_{|u|\ge \xi}e^{-u^2} du\, .
\end{align*}
In the unconditional case (effectively $a=0$ above), the probability of large deviations from the normal law was found by Waknin in \cite{W23} by using the G\"artner-Ellis theorem, the standard tool of the large deviation theory. This requires asymptotic calculation of the suitably scaled logarithmic moment generating function of $N_q$ which in \cite{W23} (Lemma A.1) was achieved by exploiting the product structure of the moment generating function \eqref{pgfNq}. For $a\not=0$ the product structure is not there. Instead, in this case Theorem \ref{afixed_t} (see also Remark \ref{remark:4.6}) yields the desired rescaled logarithmic moment generating function which, for moderate and large deviations, turns out to be exactly the same as in the unconditional case, arriving at
\begin{align}\label{W<1}
\Prob \Big\{
\Big| N_q-\frac{1}{1-q} \Big| \ge \frac{\xi}{(1-q)^{\alpha}}\,  \Big|\,  \varphi (0)=a
\Big\}
= \exp\left[ -\frac{\xi^2\, (1+o(1))}{(1-q)^{2\alpha-1}}\right], \quad \alpha\in (1/2,1),
\end{align}
and
\begin{align}\label{W1}
\Prob \left\{
\left| N_q-\frac{1}{1-q} \right| \ge \frac{\xi}{(1-q)}\,  \Big|\,  \varphi (0)=a
\right\}
= \exp\left[
-\frac{(\xi+1) \Psi (1+\xi) (1+o(1))}{(1-q)}
\right]
\end{align}
where
\begin{align*}
\Psi(s)=s+W_0(-se^{-s}) +\frac{1}{s}\Li_2\! \big(1-e^{s+W_0(-se^{-s})}\big).
\end{align*}
\end{proof}
The asymptotic relations \eqref{W<1} and \eqref{W1} coincide with  those obtained by Waknin \cite{W23} for the unconditional case but are written in a slightly different form. We will return to \eqref{W<1} and \eqref{W1} in the next section and obtain more precise asymptotics which include pre-exponential factor.

In the unconditional case Waknin also obtained the large deviation rate function  in the regime when the deviations are in the scale much greater than the variance of $N_q$ (i.e., for $\alpha>1$). Extending that result to the conditional case requires asymptotic analysis of $F(x,t;q)$ in a scaling regime when $x$ approaches 1 at the same time as $q$ approaches 1, which is not currently available.

\subsection{Gumbel limit.}
Guided by Theorem \ref{formal}, we are led to consider the case when $x$ is real positive and $t$ is large. Correspondingly, we set $t=u/\varepsilon$ in \eqref{int_rep}. This gives a saddle point integral
\begin{align}\label{saddlepint}
F\left(x,\frac{u}{\varepsilon};q\right)=\frac{(q;q)_\infty}{2\pi i}\int_{\rho-i\infty}^{\rho+i\infty}\frac{\exp\left[{\varepsilon^{-1} f(z) }\right]}{\sqrt{1-z}}\;dz
\times\left(1+O(\varepsilon)\right)\;, \quad \varepsilon=-\log q,
\end{align}
where $\rho\in (0,1)$ and
\begin{align}
\label{leading}
f(z)=\log x\log z+{\Li}_2(z)+u(1-z).
\end{align}
As before, the integrand is analytic in the complex $z$-plane and has branch cuts $(-\infty,0]$ and $[1,+\infty)$ along the real axis unless $x=1$ in which case the only branch cut is $[1,+\infty)$. The saddle point equation for $f(z)$ is
\begin{align}
\label{saddleeqn}
\frac1 z \log x - \frac1 z \log(1-z)-u =0.
\end{align}
We find that it can be resolved in terms of the Lambert-$W$ function, the function that solves the equation $W\exp W=h$ for $W$. In the complex $h$-plane
this function is multivalued with infinitely many branches $W_k(h)$, $k\in \mathbb{Z}$, see e.g., \cite{CGHJK96}. We will prove below that the saddle points of $f(z)$ are expressed through $W_{0}$ and $W_{-1}$. These are the only branches that take real values; $W_0(h)$ is real-valued on the interval $ [-e^{-1},+\infty)$ and monotone increasing there, and $W_{-1}(h)$ is real-valued on $[-e^{-1},0)$ and monotone decreasing there, see Fig. \ref{Fig:6}.

\begin{figure}[ht]
\includegraphics[width=.6\linewidth]{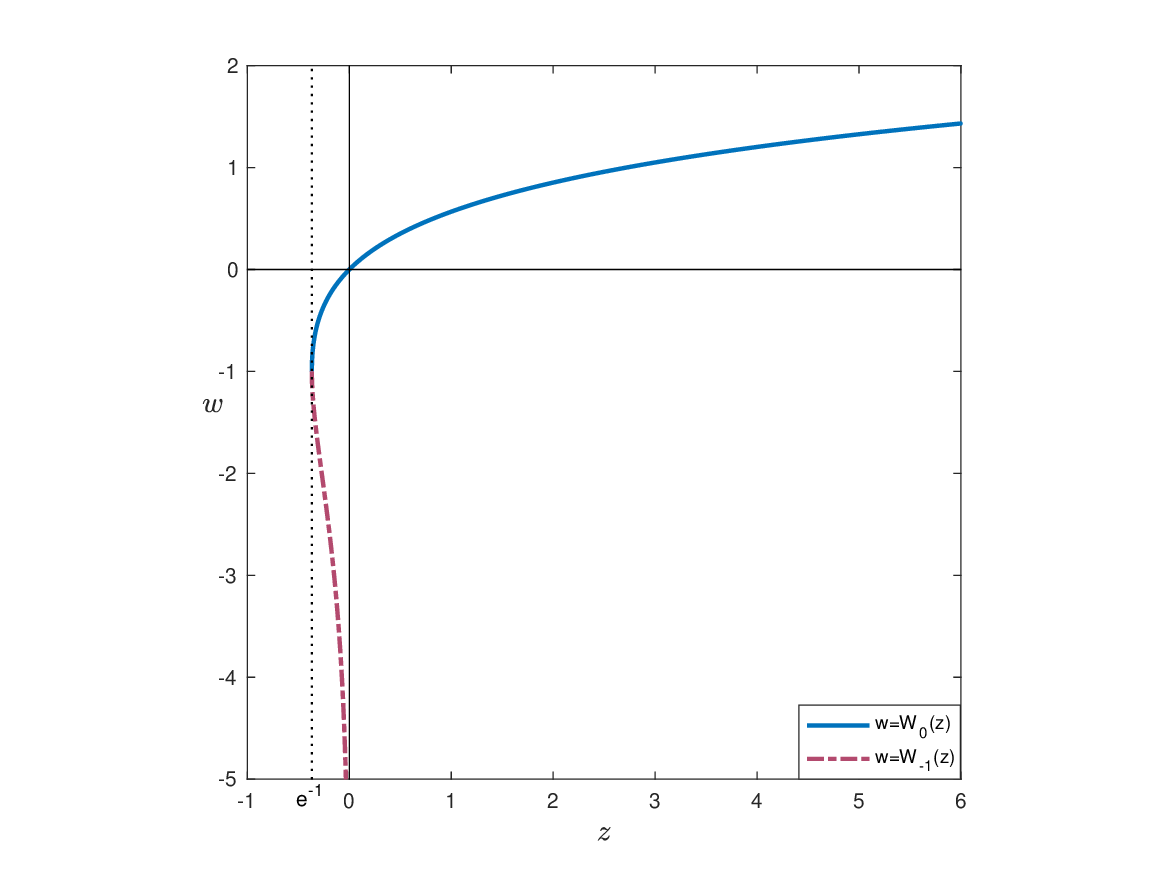}
\caption{
Two real branches of the Lambert-$W$ function. }
\label{Fig:6}
\end{figure}

\begin{lemma}
\label{Lemma:saddles}
Assume that $x>0$. Let
$u_k(x)=-W_{k} \left(\!-\frac{1}{ex}\right)$ and
\begin{align}\label{roots}
z_k(x,u)=1+\frac{1}{u}W_{k}(-xue^{-u})\:.
\end{align}
\begin{itemize}
\item[(a)] If $x>1$ and $u>0$ then the function $f(z)$
has two saddles $z_0(x,u)$ and $z_{-1}(x,u)$. These saddles are complex conjugated if $u_0(x)<u<  u_{-1}(x)$, otherwise they are real; furthermore, $z_0(x,u)=z_{-1}(x,u)$ if $u=u_0(x)$ or $u= z_{-1}(x,u)$.
\item[(b)] If $x=1$ then for every $u>0$ the function $f(z)$ has one saddle given by $z_{-1}(1,u)$ if $0<u\le 1$ and by $z_{0}(1,u)$ if $u\ge  1$.
This saddle is real.
\item[(c)]
For every $0<x<1$ and $u>0$ the function $f(z)$ has two saddles $z_0(x,u)$ and $z_{-1}(x,u)$. These saddles are real.
\end{itemize}
\end{lemma}
\begin{proof}
Evidently, $z=0$ is a root of \eqref{saddleeqn} if and only if $x=1$ and $u=1$. This particular case is captured in part (b) (note that $W_0(-1/e)=W_{-1}(-1/e)=1$). It remains to find non-zero roots \eqref{saddleeqn}, and, from now on,  we assume that $z\not=0$. In this case equation \eqref{saddleeqn} is equivalent to
\begin{align}\label{saddleeqn1}
\log x = \log(1-z)+ uz\, .
\end{align}
On exponentiating and multiplying through by $-u\exp(-u)$,  this equation is transformed to
\begin{align}\label{saddleeqn2}
-uxe^{-u}=-(1-z)ue^{-(1-z)u}\;.
\end{align}
Every root of \eqref{saddleeqn1} is a root of equation \eqref{saddleeqn2} for $z$ but not every root of \eqref{saddleeqn2} is a root of equation \eqref{saddleeqn1}.
The roots of equation \eqref{saddleeqn2} are given by
\begin{align} \label{saddleW}
z=1+\frac1uW(-ux e^{-u}),
\end{align}
where $W(h)$ is the Lambert-$W$ function. Since this function has infinitely many branches $W_k$, our task now is to verify which of the roots $z_k(x,u)=1+ u^{-1}W_k(-ux e^{-u})$ solve equation \eqref{saddleeqn1}. On substituting $z_k(x,u)$ for $z$ in \eqref{saddleeqn1}, we obtain
\begin{align}\label{saddleeqn3}
\log x = \log(-u^{-1}W_k(-ux e^{-u})) + u + W_k(-ux e^{-u})\, .
\end{align}
Since $u$ is positive real,   $\log(-u^{-1}w)= -\log u +  \log(-w)$ for every complex $w$, and \eqref {saddleeqn3} can equivalently be written as
\begin{align}\label{saddleeqn4}
\log|h|  = W_k(h) +\log(-W_k(h)), \quad h=-xue^{-u}\, .
\end{align}
The branches of the Lambert-$W$ function can be discriminated by the relation (see Eq. 12 in \cite{JHC96})
\begin{align}\label{Wbranch}
W_k(h) + \log W_k(h) =
\begin{cases}
\log h, & \text{for } k=-1 \text{ and } h\in(-1/e, 0);\\
\log h + 2\pi i k,  &\text{ otherwise};
\end{cases}
\end{align}
where, as before, $\log z = \log|z| + i \arg (z) $ is the principal branch of the logarithm. This is almost the r.h.s. of \eqref{saddleeqn4}, except for the sign inside the logarithm. Since $\log(ab)$ may differ from  $\log a+ \log b$ in the complex plane, care needs to be taken when expressing $\log(-W_k(h))$ in terms of $\log W_k(h)$. To this end, we note that
\begin{align}\label{log}
\log(-w)  =
\begin{cases}
\log w - i \pi,  & \text{ if } \im w >0  \text{ or } w< 0,\\
\log w + i \pi, & \text{ if } \im w < 0 \text{ or } w > 0,
\end{cases}
\end{align}
where we use the convention that $\arg w \in (-\pi, \pi]$. The two main branches, $W_0$ and $W_{-1}$, are real-valued on the $[-e^{-1}, 0)$ and take there negative values. Therefore, by making use of \eqref{Wbranch} and \eqref{log}, for $h\in [-e^{-1}, 0)$,
\begin{align}\label{x=1}
W_k(h) + \log(-W_k(h))
= \log h -i\pi = \log|h|,  \quad k=0, -1.
\end{align}
If $h\in (-\infty, -e^{-1})$, i.e. $h$ is on the branch cut for $W_0(h)$, then $\Im W_{0}(h)>0$ and $\Im W_{-1}(h)<0$. Therefore, for $h<-1/e$
\begin{align}\label{x1a}
W_0(h) + \log(-W_0(h)) &
= \log h -i\pi =\log|h|\, ,\\
\label{x1b}
W_{-1}(h) + \log(-W_{-1}(h)) &
= \log h + i\pi =\log|h|\, .
\end{align}
The rest of the branches of $W(h)$ are complex-valued on the negative semi-axis: if $h\in (-\infty, 0)$ then $\Im W_k (h)<0$ for $k\le -2$ and $\Im W_k (h)> 0$ for $k\ge 1$. Therefore, for $h<0$
\begin{align*}
W_k(h) + \log(-W_k(h))
&= \log|h| + i2\pi(k+1), \quad k= -2, -3, \ldots\, ,\\
W_k(h) + \log(-W_k(h))
&= \log|h| + i2\pi k ,  \quad\quad \quad\,\, k= 1, 2, \ldots \, .
\end{align*}
The last two relations are evidently incompatible with equation  \eqref{saddleeqn4}. Therefore, only the two main branches of the Lambert-$W$ function in \eqref{saddleW} can yield roots of the saddle point equation.

If $x=1$ then $h\in [-e^{-1}, 0)$ in \eqref{saddleeqn3} and  \eqref{x=1} implies that both $z_0(1,u)=1+u^{-1}W_0(-u e^{-u})$
and $z_{-1}(1,u)=1+u^{-1}W_{-1}(-ue^{-u})$ solve equation \eqref{saddleeqn1}. However, $z_0(1,u)=0$ for every $0<u<1$ and is not a saddle, and similarly $z_{-1}(1,u)=0$, not a saddle, for every  $u>1$. The case of $u=1$ has already been dealt with. In this case $z=0$ is a saddle. This proves part (b).

If $x>1$ and $u\in (u_1(x), u_2(x))$ then $h<-1/e$ in \eqref{saddleeqn3}, and relations \eqref{x1a} -- \eqref{x1b} verify that both $z_0(x,u)$ and $z_{-1}(x,u)$ are saddles. These saddles are complex conjugated.

Finally, if $0<x<1$ or if $x>1$ and $u\in (0, u_1(x)) \cup (u_2(x),+\infty)$ then $h\in [-e^{-1}, 0)$ in \eqref{saddleeqn3}, and \eqref{x=1} verifies that both $z_{0}(x,u)$ and $z_{-1}(x,u)$ are saddles. These saddles are evidently real.
\end{proof}

With the saddle point integral \eqref{saddlepint} in mind, we now take a look at the location of the real saddles of $f(z)$. For illustration, Fig. \ref{Fig:7} depicts real saddles $z_s$ of $f(z)$ in the $(u,z_s)$-plane in the three characteristic cases of $0<x<1$, $x=1$ and $x>1$.

\begin{figure}[ht]
\includegraphics[width=.6\linewidth]{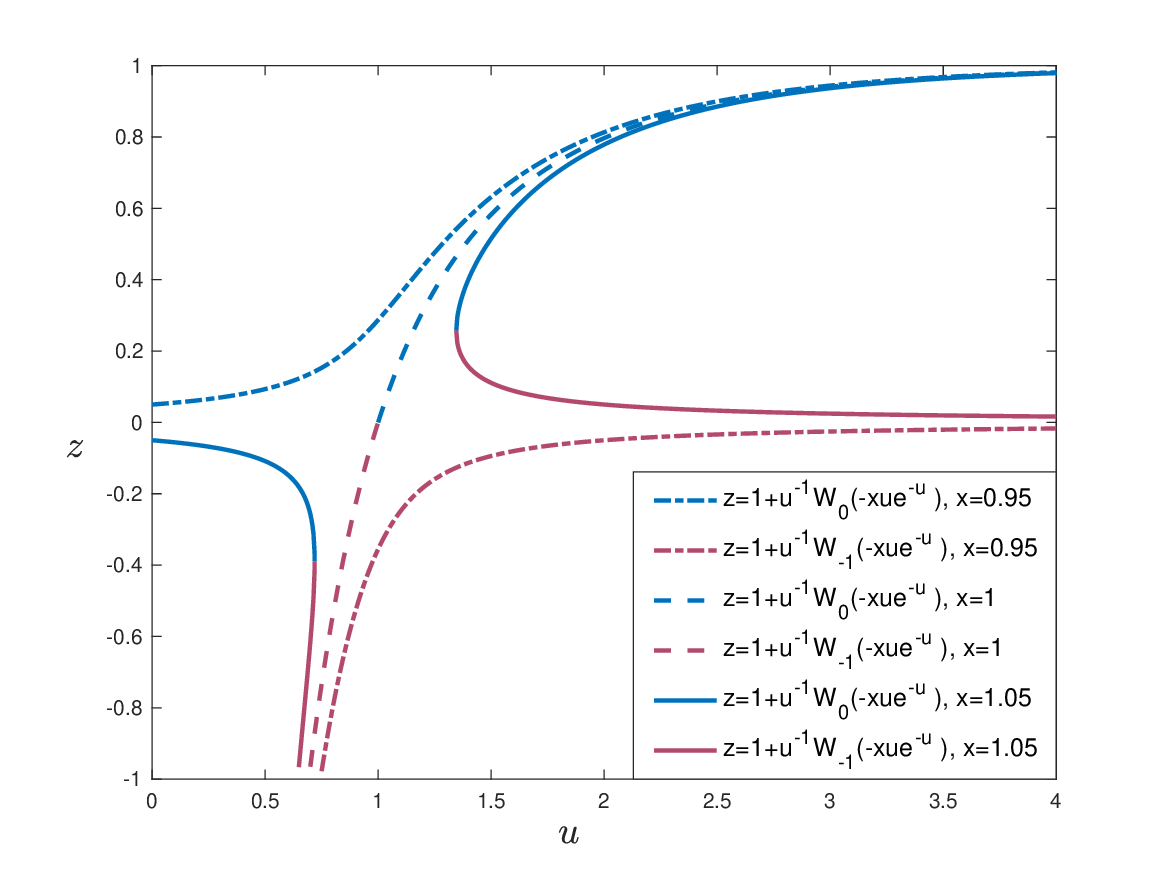}
\caption{
Graphs of the real saddles $z_s$ of $f(z)$ as functions of $u$ in the $(u,z_s)$-plane  for $x=0.95$ (dot-dashed line), $x=1$ (dashed line), and $x=1.05$ (solid line).
}
\label{Fig:7}
\end{figure}

Consider first the case when $x>1$. Note that $u_0(x)$ and $u_{-1}(x)$ are the two real roots of the equation $-xue^{-u}=-e^{-1}$,  $u_0(x)< 1 <u_{-1}(x)$, see Fig. \ref{Fig:6a}. Since $W_0(h)$ is monotone increasing  and $W_{-1}(h)$ is monotone decreasing on the interval $-1/e\le h<0$ and $W_0(-ue^{-u})=-u$ if $u<1$ and $W_{-1}(-ue^{-u})=-u$ if $u>1$,  we find that
\begin{align}\label{82}
    & z_{-1}(x,u) < 1-\frac 1 u < z_0(x,u)< 0 \quad \quad  \text{if $0< u < u_0(x)$, }\\
    \label{83}
0<& z_{-1}(x,u) <  1-\frac 1 u <  z_0(x,u) <1  \quad \quad  \text{if $u>  u_{-1}(x)$, }
\end{align}
and $z_{-1}(x,u)=1-\frac 1u=z_{0}(x,u)$  if and only if $u=u_0(x)$ or $u=u_{-1}(x)$.
\begin{figure}[ht]
\includegraphics[width=.6\linewidth]{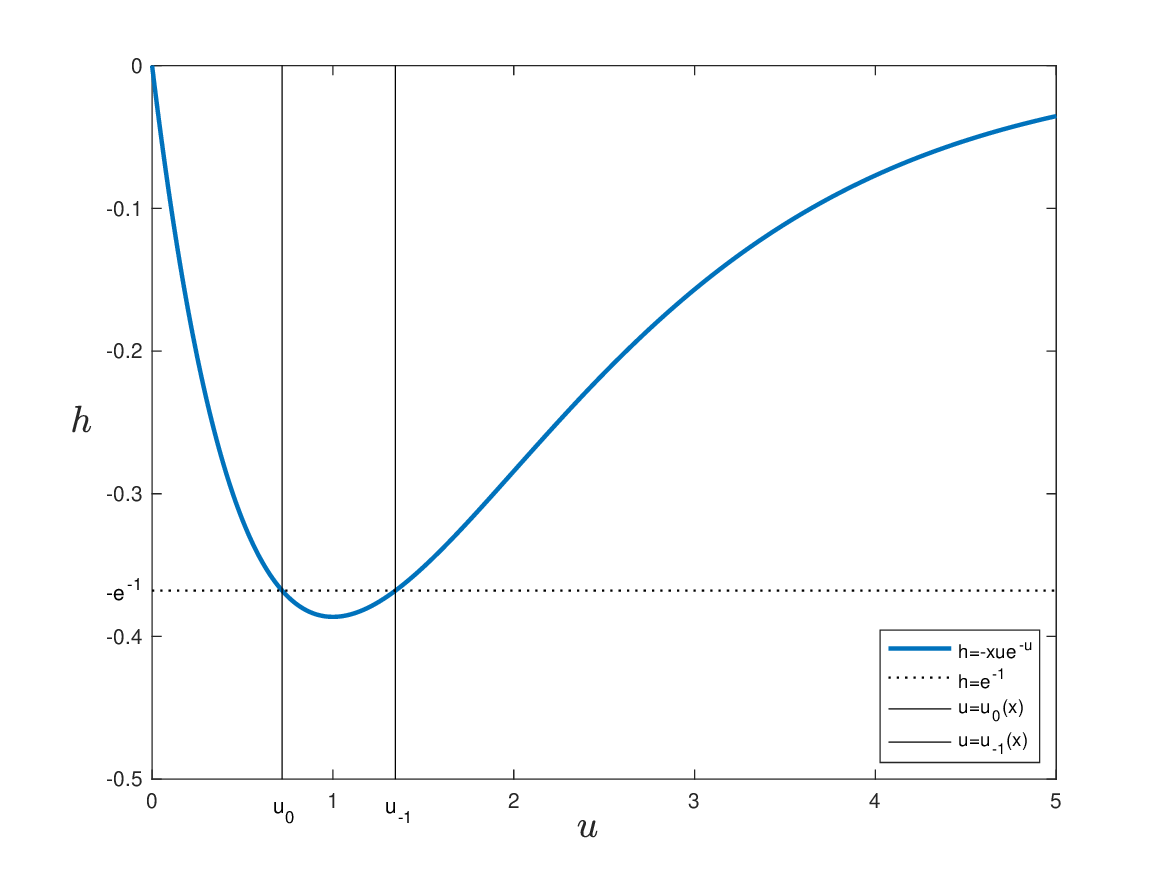}
\caption{
Graph of the function $h(u)=-xue^{-u}$ for $x=1.05$
}
\label{Fig:6a}
\end{figure}
When $0<u\le u_0(x)$, both $z_{-1}(x,u)$ and $z_{0}(x,u)$ are on a branch cut of $f(z)$. Therefore, in this case there are no suitable saddle points to work with. In contrast, when $u\ge u_{-1}(x)$ both saddle points are off the branch cuts of the integrand in \eqref{saddlepint}.

Turning to $0<x<1$, in this case $-xue^{-u} > -ue^{-u}$ for all $u>0$ and we find that
\begin{align}\label{84}
z_{-1}(x,u) < 1-\frac{\max (u,1)}{u} \le 0 \le 1-\frac{\min (u,1)}{u} < z_0(x,u) < 1.
\end{align}
We see that in this case the saddle point $z_0(x,u)$ is suitable while $z_{-1}(x,u)$ is not, as the latter lies on a branch cut of $f(z)$.

And, finally, when $x=1$ the integrand in \eqref{saddlepint} has only one branch cut $[1,+\infty)$ and any saddle in the interval $(-\infty, 1)$ would be good. In this case $f(z)$ has only one saddle (Lemma \ref{Lemma:saddles} part (b)). Exploiting the inequalities below
\begin{align}\label{85}
&0< 1-\frac{1}{u} < z_0(1,u) < 1 \quad \quad\quad \quad \quad \text{if $u>1$,}\\
\label{86}
&\phantom{0< 1-\frac{1}{u} < } z_{-1}(1,u) < 1-\frac{1}{u}<0 \quad \text{if  $0<u<1$,}
\end{align}
and noting that $z_0(1,1)=z_{-1}(1,1)=0$, this saddle can be seen to be off the branch cut, and hence is suitable.

\smallskip
With the above, we now perform the saddle-point approximation to arrive at
\begin{theorem}\label{main}
Let $\varepsilon=-\log q$. If $(x,u)\in D=(0,1]\!\times \!(0,+\infty)\cup (1,+\infty)\! \times \![u_{-1}(x), +\infty)$
 then, in the scaling limit $\varepsilon\to  0^{+}$, $t \varepsilon=u>0$,
\begin{align}\label{qas1}
F\left(x,\frac u \varepsilon ;q\right)= \displaystyle{\frac{\exp\left\{{\varepsilon^{-1} \big[\log x\log z_s+{\Li}_2(z_s)+u(1-z_s)-\pi^2/6\big]}\right\}}{\sqrt{1+(1-u)(1-z_s)/z_s}}} \left(1+O(\varepsilon)\right),
\end{align}
where
\begin{align}
\label{saddle}
z_s=
\begin{cases}
z_0(x,u) &  \text{if }  x\not=1 \lor (x=1\land u\ge 1) \\
z_{-1}(x,u) & \text{if }  x=1\land u\in (0,1)
\end{cases}
\end{align}

\end{theorem}

\begin{proof} We shall use the saddle point integral approximation \eqref{saddlepint} of function $F(x,t;q)$.
Expanding $f(z)$ to quadratic order in $(z-z_s)$,
\begin{align} \label{quadratic1}
f(z)=f(z_s)+f^{\prime\prime} (z_s) \frac{(z-z_s)^2}2+O((z-z_s)^3),
\end{align}
we find
\begin{align} \label{quadratic}
f^{\prime\prime} (z_s)= \frac{1-u +uz_s}{z_s(1-z_s)}\, .
\end{align}

Consider first $(x,u)\in (0,1)\!\times \!(0,+\infty)$. In this case there is only one suitable saddle point $z_0=z_0(x,u)$, and it is evident from  \eqref{84} that $f^{\prime\prime} (z_0) >0$.

If $x=1$ then $f(z)$ has a single saddle point $z_s=z_0(x,u)$ if $u\ge 1$ and $z_s=z_{-1}(x,u)$ if $0<u<1$. It is evident from \eqref{85}--\eqref{86} that $f^{\prime\prime} (z_s) >0$.

If $ (1,+\infty) \times [u_{-1}(x), +\infty)$ then in this case there are two suitable saddle points,
\begin{align*}
 0< z_{-1}(x,u) < z_0(x,u)<1.
 \end{align*}
 It is evident from \eqref{83} that $f^{\prime\prime} (z_{-1}) < 0$ and $f^{\prime\prime} (z_0) > 0$. The derivative of $f(z)$ is negative on the interval $z_{-1} <z<z_0$. Therefore $f(z_0)<f(z_{-1})$ and hence the saddle point $z_{-1}$ is insignificant in that if the path of integration is deformed to cross both saddles, its contribution to the integral in the limit $\varepsilon \to 0^{+}$ will be exponentially small compared to that of $z_0$. Hence, the saddle at $z_{-1}(x,u)$ can be ignored.

Thus we have a suitable saddle for the entire parameter range $D$. Choosing now a suitable path of integration, as we cross the saddle point $z_s$ in the direction of the imaginary axis, the contribution of the saddle point to the integral in \eqref{saddlepint} is
\begin{align*}
i\, \sqrt{\frac{2\pi \varepsilon}{f^{\prime\prime} (z_s)}} \frac{\exp[\varepsilon^{-1} f(z_s)] }{\sqrt{1-z_s}} = i\,  \sqrt{2\pi \varepsilon} \, \frac{\exp\{\varepsilon^{-1}[ \log x \log z_s + {\Li}_2 (z_s) +u(1-z_s)]\} }{\sqrt{1+(1-u)(1-z_s)/z_s}},
\end{align*}
and recalling \eqref{eq_qqas} we arrive at \eqref{qas1}.
\end{proof}

The desired result can now be stated as a corollary.

\begin{corollary}
\label{gumbel}
\phantom{a}
\begin{itemize}
\item[(a)]
For $x>0$ and real $s$,
\begin{align}\label{g2}
\lim_{\varepsilon\to0^+}F(x,({s-\log\varepsilon})/{\varepsilon};e^{-\varepsilon})=\exp \left({-xe^{-s}}\right).
\end{align}
\item[(b)]
The counting function $N_q$ of the absolute squares of zeroes of the Gaussian power series ${\varphi (z)}=\sum_{k\ge 0} c_k z^k $ conditioned on the event that $|g(0)|^2=t$ converges in distribution, in the scaling limit
\begin{align}\label{sl_t}
q\to 1^{-}\,\, \text{ and }\,\, t=\frac{s-\log \log \frac1 q }{\log \frac 1 q},
\end{align}
 to Poisson$(e^{-s})$.
\end{itemize}
\end{corollary}

\begin{proof} Part (a).
Choose $u=s-\log\varepsilon$ in Theorem \ref{main}. For $\varepsilon\to 0^+$ with $s$ fixed, $u\to\infty$ and we expand
\begin{align*}
z_s=1+u^{-1} W_0(-xue^{-u})=1-xe^{-u}+ O(e^{-2u}),
\end{align*}
and
\begin{align*}
\log(x)\log(z_s) &= -x e^{-u} +O(e^{-2u}),\\
u(1-z_s) &= xue^{-u}  +O(ue^{-2u}), \\
1+(1-u)(1-z_s)/z_s & = 1+O(e^{-u}).
\end{align*}
Recalling that
\begin{align*}
{\Li}_2(1-\delta) = \pi^2/6 + \delta(\log \delta -1 ) +O(\delta^2), \quad \text{when } \delta \to 0^+,
\end{align*}
we find that
\begin{align*}
{\Li}_2(z_s) =   \pi^2/6 +x (\log (x) -1) e^{-u} - x u e^{-u} +O(ue^{-2u}).
\end{align*}
Therefore,
\begin{align}
F\left(x,{u}/{\varepsilon}; e^{-\varepsilon}\right) = \exp(-\varepsilon^{-1}xe^{-u}) \times
\big( 1+ \varepsilon^{-1} O(ue^{-2u})\big) \big(1+O(e^{-u})\big)
\end{align}
Now inserting $u=s-\log\varepsilon$, we find that
\begin{align*}
F\left(x,(s-\log\varepsilon)/\varepsilon;e^{-\varepsilon}\right)=\exp({-xe^{-s}})\times\left(1+O(\varepsilon\log\varepsilon)\right),
\end{align*}
and \eqref{g2} follows.

\smallskip

Part (b). The l.h.s. of \eqref{g2} is the conditional moment generating function of $N_q$ and the r.h.s. is the moment generating function of Poisson$(e^{-s})$.
\end{proof}

\begin{remark}
The scaling limit \eqref{sl_t} is essentially the same as the scaling limit \eqref{sl_p} considered in Section \ref{Sec:3}. To see that we note that $t\to+\infty$ in the scaling limit \eqref{sl_t}. In this range the relation
\begin{align*}
t=\frac{s-\log  \varepsilon}{ \varepsilon}, \quad \varepsilon=\log \frac 1 q,
\end{align*}
can be written as $\varepsilon=t^{-1}W_0(te^s) $. where $W_0$ is the main branch of the Lambert-$W$ function. Recalling that
\begin{align*}
W_0(t) = \log t -\log \log t +O ((\log t)^{-1} \log \log t ),
\end{align*}
we find that
\begin{align*}
\varepsilon t = s+\log t - \log \log t + O ((\log t)^{-1} \log \log t )
\end{align*}
Hence,
\begin{align*}
q= 1- \frac{\log t - \log \log t +s}{t}  + O\left(\frac{\log \log t}{t\log t}\right)
\end{align*}
which, asymptotically, is equivalent to \eqref{sl_p}.

\end{remark}

\subsection{Hole Probability: Proof of Theorem \ref{Thm:hole_pr}.}
Next, we return to the case of fixed $t>0$ and investigate the asymptotic behaviour of the series $F(x,t;q)$ \eqref{Fext} for $x=1$ in the limit $q\to 1^{-}$. Writing
\begin{align}\label{F1a}
F(1,t;q)=  I(t;q) (q;q)_\infty \big(1\!+\!O(\varepsilon)\big), \quad \varepsilon = - \log q\, ,
\end{align}
see Lemma \eqref{Lemma:intrep}, we will focus on the saddle-point integral
\begin{align*}%\label{int_rep1}
I(t;q)\!=\!\frac{1}{2\pi i}\!\!\int_{\rho-i\infty}^{\rho+i\infty}
\!\!\!\exp\left[{\frac{{\Li}_2(z)}{\varepsilon}\!+\!t(1\!-\!z)}\right]
\frac{dz}{\sqrt{1\!-\!z}}, \quad\quad \rho \in (0,1).
\end{align*}
The function $\exp t(1-z)$ behaves violently in the neighbourhood $z=\infty$, the saddle point of ${\Li}_2(z)$, and  needs to be included in our considerations. Correspondingly, we write
 \begin{align}\label{saddlepint1}
I(t;q)\!=\!\frac{1}{2\pi i}\!\!\int_{\rho-i\infty}^{\rho+i\infty}
\frac{\exp\left[\varepsilon^{-1} f(z)\right]}{\sqrt{1\!-\!z}}\, dz,
\end{align}
where
\begin{align}\label{leading1}
f(z)={\Li}_2(z) + u(1-z), \quad\quad u=t\varepsilon.
\end{align}
We are back in the domain of Theorem \ref{main},
albeit now with $u=O(\varepsilon)$. It follows from Lemma \ref{Lemma:saddles}, part (b), that for $\varepsilon$ small enough, $f(z)$ has a unique saddle point given by
\begin{align*}
z_s=1+\frac{W_{-1}(-ue^{-u})}{u}, \quad u=t\varepsilon.
\end{align*}
In the limit $\varepsilon \to 0^{+}$, $z_s$ diverges to $-\infty$ as $\log (u)/u$. We will need a more precise asymptotic form of $z_s$. By expanding  $W_{-1}(-ue^{-v})$ in powers of $v$ about $v=0$, we obtain
\begin{align}\label{zse}
z_s=\frac{w}{u}+ \frac{1}{1+w} +\frac{wu}{2(1+w)^3} + O(u^2) = \frac{w}{u}\left[ 1+ \frac{u}{w(1+w)} + \frac{u^2}{2(1+w)^3}+O(u^3) \right],
\end{align}
where
\begin{align*}
w=W_{-1}(-u) = \log u - \log |\!\log u |  + O\left(\frac{\log |\!\log u | }{\log u}\right).
\end{align*}
Variable $w$ accounts for the additional scale $\log u$.

The asymptotic form of the dilogarithm at the saddle point can be worked out  by recalling the identity
\begin{align*}
{\Li}_2 (z) + {\Li}_2(1/z)=-\frac{\pi^2}{6} -\frac{1}{2} \log^2 (-z), \quad\quad z\notin [0,+\infty).
\end{align*}
By the definition of $w$, $we^w=-u$ and, hence, $\log (-w/u)=-w$. Now, taking into account \eqref{zse}, we find that
\begin{align*}
\log^2 (-z_s)= w^2-\frac{2u}{1+w} + O(u^2),
\end{align*}
which together with ${\Li}_2 (1/z_s) = u/w +O(u^2)$ and  $u(1-z_s)=-w+uw/(1+w) +O(u^2)$ leads to the asymptotic form of $f(z)$ at its saddle point
\begin{align*}
f(z_s) = -\frac{w^2}{2}-w -\frac{\pi^2}{6} +u -\frac{u}{w}  + O(u^2)\, .
\end{align*}
Keeping in mind the saddle point approximation of the integral in \eqref{saddlepint1}, we revisit the quadratic expansion \eqref{quadratic1} of $f(z)$  about $z_s$. We find that
\begin{align}\label{r1}
f^{\prime\prime} (z_s) &= & \frac{1}{(1-z_s) z_s} -\frac{u}{z_s} &=  -(1+w) \left(\frac{u}{w}\right)^2 + \frac{1}{1+w} \left(\frac{u}{w}\right)^3 +O\left(u^4\right) \\
\label{r2}
f^{\prime\prime\prime} (z_s)&= & \frac{1}{(1-z_s)^2 z_s}-\frac{2}{(1-z_s) z_s^2} +\frac{2 u}{z_s^2} &=  (2w+3) \left(\frac{u}{w}\right)^3 +O\left(u^4\right)
\end{align}
We see that that the quadratic term $\frac{1}{2\varepsilon}f^{\prime\prime} (z_s)(z-z_s)^2$ vanishes in the limit $\varepsilon \to 0^{+}$.  Reflecting on \eqref{r1} -- \eqref{r2},  we are prompted to transform the variable of integration in \eqref{saddlepint1}. Introducing
\begin{align*}
\delta={w}/{u}, \quad  \zeta =z/\delta, \quad \text{and} \quad \hat f (\zeta) = f\left(\delta \zeta\right),
\end{align*}
we write
 \begin{align}\label{saddlepint2}
I(t;q)=\frac{-\delta}{2\pi i}\!\!\int_{\rho-i\infty}^{\rho+i\infty}
\frac{\exp\big[\varepsilon^{-1} \hat f(\zeta)\big]}{\sqrt{1- \delta \zeta}}\, d\zeta, \quad  \rho \in (-\infty, 0),
\end{align}
where
\begin{align*}
\hat f(\zeta)={\Li}_2\left(\delta \zeta\right) + u(1-\delta \zeta).
\end{align*}
Evidently, the function $\hat f(\zeta)$ has a unique saddle point determined by the relation $z_s=\delta \zeta_s$,
\begin{align*}
\zeta_s=  1+ \frac{u}{w(1+w)} + O(u^2).
\end{align*}
Expanding the exponent in integrand in \eqref{saddlepint2} about the saddle point, we obtain
\begin{align*}
\hat f(\zeta) = \hat f(\zeta_s) +\frac{1}{2} \hat f^{\prime\prime} (\zeta_s) (\zeta-\zeta_s)^2 + O((\zeta-\zeta_s)^3),
\end{align*}
with $\hat f  (\zeta_s) = -\frac{w^2}{2}-w -\frac{\pi^2}{6} +u -\frac{u}{w}  + O(u^2) $ and
$\hat f^{\prime\prime}  (\zeta_s)  = -(1+w) +O\left(\frac{u}{w^2}\right)$.
Now, it is plain to see that the large parameter in the saddle point integral \eqref{saddlepint2} is given by $(1+w)/\varepsilon$,
\begin{align*}
-(1+w)/\varepsilon \sim - \log (\varepsilon) / \varepsilon ,  \quad\quad \text{as } \varepsilon \to 0^{+},
\end{align*}
so that by choosing a suitable path of integration to pass through the saddle point and using the quadratic approximation in the exponent in  \eqref{saddlepint2}  we arrive at
\begin{align*}
I(t;q)
 & =\sqrt{\frac{\varepsilon}{2\pi}} \sqrt{\frac{1}{u (1-(\delta u)^{-1})} \frac{1}{\zeta_s(1-(\delta \zeta_s)^{-1}}} \exp \left[ \frac{1}{\varepsilon} \hat f (\zeta_s) \right] \left(1+ O \left(\frac{\varepsilon}{\log \varepsilon} \right) \right)\\
 &= \sqrt{\frac{\varepsilon}{2\pi}} \frac{1}{\sqrt{u}}  \exp \left[ \frac{1}{\varepsilon} \left( -\frac{w^2}{2}-w -\frac{\pi^2}{6} + u -\frac{u}{w}  \right) \right] (1+O(u))
\end{align*}
Recalling now \eqref{F1a} and \eqref{eq_qqas} and noting that $u=t\varepsilon$ and $w=W_{-1}(-u)$, we finally arrive at the desired asymptotic form \eqref{gappr2m} of $F(1,t;q)$ in the limit $\varepsilon = -\log q \to 0^{+}$. This completes our proof of Theorem \ref{Thm:hole_pr}.

\section{Tail asymptotics for $N_q$}
\label{section5}
Let $0<q<1$ and $t\geq0$. In this section we analyse the probabilities
\begin{align*}
P_k(q)= \Prob \{N_q\!=\!0 \}=\left.\frac{(-1)^k}{k!}\frac{d^k}{dx^k}\right|_{x=1}\!\!F(qx,0;q)
\end{align*}
and
\begin{align*}
 P_k(t;q) =  \Prob \{N_q\!=\!0 \big| |{\varphi (0)}|^2 \!=\!t\}= \left.\frac{(-1)^k}{k!}\frac{d^k}{dx^k}\right|_{x=1}\!\!F(x,t;q)
\end{align*}
in three regimes: (i) $q$ is fixed and $k\to \infty$, (ii) $k$ is fixed and $q\to 1^{-}$, and (iii) $k\to \infty$ and $q\propto 1/k $. We note that $P_0(q)=(q;q)_\infty$ and
$P_{k+1}(0;q)=P_{k}(q)$ for $k\ge 0$.

\smallskip

Most of our analysis is based on the integral representation of $F(x,t;q)$ given in Lemma \ref{lemma2}.
\begin{lemma}
For every $q\in (0,1)$ and $k\ge 0$
\begin{align*}
P_k(q)
=\sum_{m=k}^\infty(-1)^{m-k}\binom mk\frac{q^{\binom{m+1}2}}{(q;q)_m}
=\frac1{k!}\frac{(q;q)_\infty}{2\pi i}\int_{\mathcal{C}}\frac{\prod_{l=0}^{k-1}\log(zq^l)}{z(z;q)_\infty(\log q)^k}dz
\end{align*}
and
\begin{align*}
P_k(t;q)
=\sum_{m=k}^\infty(-1)^{m-k}\binom mk\frac{q^{\binom m2}}{(q;q)_m}e^{t(1-q^{-m})}
=\frac1{k!}\frac{(q;q)_\infty}{2\pi i}\int_{\mathcal{C}}\frac{\prod_{l=0}^{k-1}\log(zq^l)}{(z;q)_\infty(\log q)^k}e^{t(1-z)}dz
\end{align*}
\end{lemma}
\begin{proof}
The infinite series are obtained by differentiation of the infinite series $F(qx,0;q)$ and $F(x,t;q)$, and the contour integrals by differentiation of the contour integral representation.
\end{proof}

\subsection{Scaling as $k\to\infty$ for $q$ fixed}

\begin{proposition}
\label{prop:5.3}
For every $q\in (0,1)$ it holds in the limit $k\to\infty$ that
\begin{align*}
P_k(q)=\frac{q^{\binom{k+1}2}}{(q;q)_\infty}\left(1+O(q^k)\right) \quad \text{and} \quad
P_k(t;q)=\frac{q^{\binom k2}}{(q;q)_\infty}e^{t(1-q^{-k})}\left(1+O(q^k)\right).
\end{align*}
\end{proposition}
\begin{proof}
We find upon inspection that the first term in the series expansion for $P_k(q)$ dominates the sum for large $k$. More precisely, we can write
\begin{align*}
P_k(q)
%&
=\frac{q^{\binom{k+1}2}}{(q;q)_k}\sum_{m=0}^\infty(-1)^m\binom{m+k}k\frac{q^{\binom{m+1}2+km}}{(q^{k+1};q)_m}
%\\&
=\frac{q^{\binom{k+1}2}}{(q;q)_k}\left(1+O(q^k)\right)=\frac{q^{\binom{k+1}2}}{(q;q)_\infty}\left(1+O(q^k)\right)
\end{align*}
An analogous argument holds for $P_k(t;q)$.
\end{proof}

\subsection{Scaling as $q\to1^{-}$ for $k$ fixed}
\begin{proof}[Proof of Theorem \ref{thm:1.10}.]
As $q\to1^{-}$, $\prod\limits_{l=0}^{k-1}\log(zq^l)\to(\log z)^k$, and hence
\begin{align*}
P_k(q)
=\frac1{k!}\frac{(q;q)_\infty}{2\pi i}\int_{\mathcal{C}}\frac{\prod_{l=0}^{k-1}\log(zq^l)}{z(z;q)_\infty(\log q)^k}dz
\sim\frac1{k!}\frac{(q;q)_\infty}{2\pi i}\int_{\mathcal{C}}\frac{(\log z)^k}{z(z;q)_\infty(\log q)^k}dz
\end{align*}
Further approximating the q-product $(z;q)_\infty\sim e^{\Li_2(z)/\log q}\sqrt{1-z}$, see the Lemma \ref{lem_qPochas}, we find, upon writing $\varepsilon=-\log q$ and substituting $z=\varepsilon s$,
\begin{align*}
P_k(q) \sim\frac{(q;q)_\infty}{k!(\log q)^k}\frac1{2\pi i}\int_{\mathcal{C}}\frac{e^{\frac{1}{\varepsilon}\Li_2(z)}(\log z)^k}{z\sqrt{1-z}}dz\\
 \MoveEqLeft[15.7]
=\frac{(q;q)_\infty}{k!(\log q)^k}\frac1{2\pi i}\!\int_{\mathcal{C}}\frac{e^{\frac{1}{\varepsilon}\Li_2(\varepsilon s)}(\log(\varepsilon s))^k}{s\sqrt{1-\varepsilon s}}ds \sim\frac{(q;q)_\infty}{k!(\log q)^k}\frac1{2\pi i}\!\int_{\mathcal{C}}\frac{e^s(\log\varepsilon)^k}sds
=\frac{(q;q)_\infty(\log\varepsilon)^k}{k!(\log q)^k}
\end{align*}
as $\frac1{2\pi i}\int_{\mathcal{C}}\frac{e^s}sds=1$. The integrand for $P_k(t;q)$ only differs by a factor $ze^{t(1-z)}$, and we find instead
\begin{align*}
P_k(t;q) \sim\frac{(q;q)_\infty}{k!(\log q)^k}\frac1{2\pi i}\!\int_{\mathcal{C}}\frac{e^{\frac{1}{\varepsilon}\Li_2(z)+t(1-z)}(\log z)^k}{\sqrt{1-z}}dz\\
 \MoveEqLeft[18.3]
=\frac{(q;q)_\infty}{k!(\log q)^k}\frac{\varepsilon}{2\pi i}\!\int_{\mathcal{C}}\frac{e^{\frac{1}{\varepsilon} \Li_2(\varepsilon s)+t(1-\varepsilon s)}(\log(\varepsilon s))^k}{\sqrt{1-\varepsilon s}} ds
\sim\frac{(q;q)_\infty}{k!(\log q)^k}\frac\varepsilon{2\pi i}\!\int_{\mathcal{C}}e^{s+t}(\log\varepsilon\!+\!\log s)^kds.
\end{align*}
In this case, the $(\log\varepsilon)^k$-term vanishes upon integration, and instead we have
\begin{align*}
P_k(t;q)
\sim\frac{(q;q)_\infty}{k!(\log q)^k}\frac{\varepsilon}{2\pi i}\!\int_{\mathcal{C}}e^{s+t}k(\log\varepsilon)^{k-1}\log s\,ds\\
 \MoveEqLeft[10]
=\frac{(q;q)_\infty(\log\varepsilon)^{k-1}e^t}{(k-1)!(\log q)^{k-1}}\frac{(-1)}{2\pi i}\!\int_{\mathcal{C}}e^s\log s\,ds =\frac{(q;q)_\infty(\log\varepsilon)^{k-1}}{(k-1)!(\log q)^{k-1}}
\end{align*}
as $\frac1{2\pi i}\int_{\mathcal{C}}e^s\log s\,ds=-1$.
\end{proof}

\subsection{Scaling as $q\to1$ and $k\to\infty$}
\begin{proof}[Proof of Theorem \ref{maintheorem}]
We write
\begin{align*}
P_k(q)=\frac1{2\pi i}\int_{\mathcal{C}}e^{f(z)}dz
\end{align*}
with
\begin{align*}
e^{f(z)}=e^{f(k,q;z)}
=\frac1{k!}\frac{(q;q)_\infty}{(z;q)_\infty}\frac{\prod_{l=0}^{k-1}\log(zq^l)}{z(\log q)^k}
=\frac1{k!}\frac1z\frac{(q;q)_\infty}{(q^kz;q)_\infty}\prod_{l=0}^{k-1}\frac{\log(zq^l)}{(1-zq^l)\log q}
\end{align*}
where in the last step we have matched the removable singularities at $z=1,q^{-1},\ldots,q^{-k+1}$. Thus, in addition to the branch point at $z=0$ the integrand has simple poles at $q^{-k}, q^{-k-1},\ldots$. We note that for $q=e^{-s/k}$ these poles, in the limit $k\to \infty$,  accumulate at $e^s$ and there are no singularities in the interval $(0,e^s)$.  While we are mainly concerned with identifying and exploiting the saddle point on $(0,e^s)$, the explicit asymptotic expansion we give below is valid uniformly on compact subsets of the cut complex plane avoiding $(-\infty,0]$ and $[e^s,\infty)$.

The proof turns out to be a standard application of the saddle-point method, involving a single saddle point. Approximating the sums occurring in $f(k,q;z)$ using the Euler-Maclaurin formula, we compute to leading order
\begin{equation}
\label{goal}
f(k,e^{-s/k};z)=k f_0(s;z)+f_1(s;z)+O(1/k)
\end{equation}
 Identifying a unique saddle point $z_s$ on the positive real axis using $f^{\prime}_0(s;z_s)=0$
and checking that $f^{\prime\prime}_0(s;z_s)>0$ as required, we find
\begin{align*}
P_k(e^{-s/k})=\frac{e^{f_1(s;z_s)}}{\sqrt{2k\pi f_0''(s;z_s)}}e^{kf_0(s;z_s)}(1+O(1/k)).
\end{align*}
We start the actual calculation with
\begin{align*}
f(k,e^{-s/k};z)
=&-\log k! -\log z
+\log(e^{-s/k};e^{-s/k})_\infty-\log(ze^{-s};e^{-s/k})_\infty %\nonumber
\\
&+\sum_{l=0}^{k-1}\log\left(\frac{ls-k\log z}{1-ze^{-ls/k}}\right)-k\log s.
\end{align*}
The $q$-products above can be expanded using \eqref{eq_asy_qprod} and \eqref{eq_qqas},
\begin{align*}
\log(e^{-s/k};e^{-s/k})_\infty&=-\frac ks\frac{\pi^2}6+\frac12\log \frac{2\pi k}{s}+O(1/k)\\
\log(ze^{-s};e^{-s/k})_\infty&=-\frac ks\Li_2(ze^{-s})+\frac12\log(1-ze^{-s})+O(1/k).
\end{align*}
We find
\begin{align*}
f(k,e^{-s/k};z)
=&-k\log k+k-\frac12\log(2\pi k)-\log z-k\log s \\
&-\frac ks\frac{\pi^2}6+\frac12\log(2\pi k/s)+\frac ks\Li_2(ze^{-s})-\frac12\log(1-ze^{-s})\nonumber\\
&+\int_0^k\log\left(\frac{vs-k\log z}{1-ze^{-vs/k}}\right)dv+\frac12\log\left(\frac{(-\log z)(1-ze^{-s})}{(1-z)(s-\log z)}\right) \\
&+O(1/k)
\end{align*}
where in the last step we used
\begin{align*}
\sum_{l=0}^{k-1}g(l)=\int_0^{k}g(v)dv+\frac12(g(0)-g(k))+R
\end{align*}
 with $|R|\le\frac k{12}\sup_{v\in[0,k]}|g''(v)|$.  For $g(v)=\log\left((vs-k\log z)/(1-ze^{-vs/k})\right)$ we need to take some care to estimate $|g''(v)|$. Letting $\rho=e^{vs/k}/z$, we find $g''(v)=\frac{s^2}{k^2}\left(\rho/(\rho-1)^2-1/(\log\rho)^2\right)$.  Some easy estimates show that $|g''(v)|\le\frac{s^2}{12k^2}$ uniformly for complex values of $z$ avoiding the cuts $(-\infty,0]$ and $[e^s,\infty)$, and thus $|R|\le\frac{s^2}{144k}$.  Performing the integral results in
%\begin{multline}
\begin{align*}
\nonumber
\int_0^k\log\left(\frac{vs-k\log z}{1-ze^{-vs/k}}\right)dv%\\
=&-\frac ks\left(\log(ze^{-s})\log\left(\frac ks\frac{\log(ze^{-s})}{ze^{-s} - 1}\right)-\log z\log\left(\frac ks\frac{\log z}{z - 1}\right) \right)\\
& -\frac ks\left(- \Li_2(ze^{-s}) - \log(ze^{-s}) + \Li_2(z) + \log z\right).
\end{align*}
%\end{multline}
Using $\Li_2(ze^{-s})+\Li_2(1-ze^{-s})=\frac{\pi^2}6-\log(ze^{-s})\log(1-ze^{-s})$, we find after a few steps the claimed form (\ref{goal}) with
\begin{align*}
f_0(s;z)=&\frac1s\log z\log \frac{\log z}{(z-1)(s-\log z)} - \frac1s\Li_2(1 - z) +\log \frac{s-\log z}s \\
f_1(s;z)=&\frac12\log \frac{\log z}{s(z-1)(s-\log z)} -\log z\, .
\end{align*}
For the derivatives of $f_0(s;z)$ we find
\begin{align*}
f_0'(s;z)=&\frac1{sz}\log \frac{\log z}{(z-1)(s-\log z)}\\
f_0''(s;z)=&-\frac1{sz^2}\log \frac{\log z}{(z-1)(s-\log z)}
+\frac1{sz(1-z)}+\frac1{sz^2\log z}+\frac1{sz^2(s-\log z)}
\end{align*}
The saddle-point equation $f_0'(s;z)=0$ now reads
\begin{equation}\label{saddleLDP}
s=\frac{z\log z}{z-1},
\end{equation}
and we seek its solutions in the $z$-plane with branch cuts $(-\infty, 0]$ and $[e^s, +\infty)$ along the real line. By making the substitution $\omega= \log (z)/s$, one transforms this equation to
\begin{align*}
\omega s + \log (1-\omega)=0
\end{align*}
which (momentarily identifying $\omega$ with $z$) coincides with the particular case of equation \eqref{saddleeqn1} corresponding to $x=1$. Its solution is given by Lemma \ref{Lemma:saddles}, part (b), and, converting back to variable $z=e^{s\omega}$ we conclude that equation \eqref{saddleLDP} has a unique root in the cut $z$-plane. This root is real and given by
\begin{align*}
z_s=e^{s+W(-se^{-s})}=-\frac{s}{W(-se^{-s})}
\end{align*}
where we use the convention that $W(-se^{-s})$ is real analytic in $s>0$,
\begin{align}\label{z_s_sc}
W(-se^{-s})=\begin{cases} W_{-1}(-se^{-s}), & \text{when } s\leq1,\\ W_{0}(-se^{-s}), &\text{when } s\geq1.
\end{cases}
\end{align}
Having identified the suitable saddle point $z_s$, we compute
\begin{align}
f_0(s;z_s)=&-\frac1s\Li_2(1 - z_s) +\log \frac{s-\log z_s}s\, ,\\
f_1(s;z_s)=&-\frac12\log s-\log z_s\, , \\
f_0''(s;z_s)=&\frac1{sz_s(1-z_s)}+\frac1{sz_s^2\log z_s}+\frac1{sz_s^2(s-\log z_s)}\, .
\end{align}
To confirm the sign of $f_0''(s;z_s)$, we compute
\begin{align}
f_0''(s;z_s)=&\frac{z_s-s}{s^2z_s(z_s-1)}=\frac{z_s-1-\log z_s}{(z_s\log z_s)^2}>0\, .
\end{align}
After some further simplification we find
\begin{align}
f_0(s;z_s)& =-\log z_s-\frac1s\Li_2(1 - z_s)=-\Psi(s)\, ,\\
\frac{e^{f_1(s;z_s)}}{\sqrt{2f_0''(s;z_s)}}&=\sqrt{\frac{\log z_s}{2(z_s-s)}}=A(s)\, .
\end{align}
This completes the proof for $P_k(e^{-s/k})$. The result for $P_k(t;e^{-s/k})$ follows from recognizing that the integrand $e^{f(z)}$ of the contour integral gets multiplied by a factor $ze^{t(1-z)}$. As $t\geq0$, this simply leads to a multiplicative factor $z_se^{t(1-z_s)}$ in the asymptotic result.
\end{proof}

\smallskip
Theorem \ref{maintheorem} is a refinement of the large deviation result \eqref{W1} obtained in Section \ref{Sec:4} by using standard techniques of the large deviation theory.  To see this, we write the scaling limit $q=e^{-s/k}$, $k\to\infty$  in the form
\begin{align*}
k=\frac{1+\xi}{1-q}+O(1), \quad q\to 1^{-},
\end{align*}
where $s=1+\xi$.
Substituting this expression for $k$ into \eqref{eq:mthm1} we find that for every $1+\xi >0$
\begin{align}\label{eq:mthm3}
\Prob \Big\{ N_q=\frac{1\!+\!\xi\!}{1\!-\!q}\!+\!O(1)\, \Big| \, \varphi (0)\!=\!a \Big\}
%\\ \MoveEqLeft[5]
= \exp\Big[-\frac{(1\!+\!\xi)\Psi (1\!+\!\xi)}{1-q}+  \frac{1}{2}\log (1-q) +O(1)\Big],
\end{align}
where the remainder term $O(1)$ depends on $\xi$ and $a$. To compare this with \eqref{W1}, we note that
\begin{align*}
\Big| N_q-\frac{1}{1-q} \Big| \ge \frac{\xi}{(1-q)} \iff
N_q \le \frac{2-s}{(1-q)}  \lor
N_q \ge \frac{s}{(1-q)}\, ,
\end{align*}
and, since $(1+\xi)\Psi (1+\xi) <(1-\xi)\Psi (1-\xi)$ for $\xi\in (0,1)$,
\begin{align*}
\Prob \Big\{
N_q \le \frac{1\!-\!\xi}{1\!-\!q}  \,  \Big|\,  \varphi (0)\!=\!a
\Big\}
\ll  \Prob \Big\{
N_q \ge \frac{1\!+\!\xi}{1\!-\!q}\,  \Big|\,  \varphi (0)\!=\!a
\Big\}
\end{align*}
in the limit $q\to 1^{-}$. Therefore, equation \eqref{W1} effectively claims that
\begin{align*}
\Prob \Big\{ N_q\ge \frac{1\!+\!\xi}{1\!-\!q} \, \Big| \, \varphi (0)\!=\!a \Big\}
= \exp\Big[-\frac{(1\!+\!\xi)\Psi(1\!+\!\xi)}{1-q} +o\Big(\frac {1}{1-q} \Big)\Big],
\end{align*}
in agreement with \eqref{eq:mthm3}.

Theorem \ref{maintheorem}  can also be used to obtain a refinement of the moderate deviation result \eqref{W<1}. To this end, set
\begin{align*}
k=\frac{1+\xi (1-q)^{1-\alpha}}{1-q}, \quad \alpha \in (0,1),
\end{align*}
in \eqref{eq:mthm1}.  In the limit $q\to 1^{-}$, for given values of $k$ and $\xi$, the parameter $s$ is determined from the relation $q=e^{-s/k}$. We find that
\begin{align*}
s=1+\xi (1-q)^{1-\alpha} +O(1-q).
\end{align*}
Using the asymptotic expansions for $\Psi(s)$, $A(s)$ and $z_s$ developed in Appendix  \ref{app:C}, we obtain
\begin{align*}
k\Psi(s)=\frac{\xi^2}{(1-q)^{2\alpha-1}} +
\begin{cases}
O\big((1-q)^{2-3\alpha}\big) & \text{if } \alpha \in [1/2,1)\\
O\big((1-q)^{1-\alpha}\big) & \text{if } \alpha \in (0, 1/2]
\end{cases}
\end{align*}
and
\begin{align*}
\frac{A(s)}{\sqrt{\pi k}} = \sqrt{\frac{1-q}{\pi}} (1+O(1-q)), \,\,\, z_se^{t(1-z_s)} = 1+O(1-q).
\end{align*}
By multiplying these three factors together we find that
\begin{align*}
\Prob \Big\{ N_q\!=\!\frac{1 \!+\! \xi (1\!-\!q)^{1-\alpha} }{1\!-\!q}\,\Big| \varphi (0)\!=\!a \Big\} = \\
\\ \MoveEqLeft[5]
\sqrt{\frac{1\!-\!q}{\pi}}
\exp
\left[
-\xi^2 (1\!-\!q)^{1-2\alpha} + O\big((1\!-\!q)^{1-\alpha} \big)
\right], \quad \text{when } \alpha \in (0,1/2],
\end{align*}
and
\begin{align*}
\Prob \Big\{ N_q\!=\!\frac{1 \!+\! \xi (1\!-\!q)^{1-\alpha} }{1\!-\!q}\,\Big| \varphi (0)\!=\!a \Big\} = \\
\\ \MoveEqLeft[5]
\sqrt{\frac{1\!-\!q}{\pi}}
\exp
\left[
-\frac{\xi^2}{(1\!-\!q)^{2\alpha-1}} + O\big((1\!-\!q)^{2-3\alpha} \big)
\right], \quad \text{when } \alpha \in [1/2,1).
\end{align*}
We note that for $\alpha\in (2/3,1)$ the remainder term in the exponential effectively makes the pre-exponential factor obsolete.

\appendix
\section{Multipoint factorial moments of conditional counting function}
\label{appendix:A}
In this Appendix we calculate the multipoint factorial moments
\begin{align}\label{MFM}
%\E \prod\nolimits_{j=1}^{k} ( N_{q_j}^{(a)} - j+1)  =
\E_{\varphi} \big\{ N_{q_1}(N_{q_2}-1) \ldots (N_{q_k}-k+1) \, \big| \, {\varphi (0)}=a \big\}
\end{align}
of the counting function of absolute squares of the zeros of the Gaussian power series ${\varphi (z)}=\sum_{k\ge 0} c_k z^k $ conditioned on the event that ${\varphi (0)}=a$. The coefficients $c_k$ are independent standard complex normals.

For this calculation we are going to use the machinery of the eigenvalue correlation functions in the random matrix ensemble $A_{n,\tau}$ \eqref{A}, and in this framework it is more convenient to work with the point process of eigenvalue moduli rather than the absolute squares of the eigenvalues. Correspondingly, we consider
%Let $r_1<r_2< \ldots < r_k$. Consider
the factorial moments:
\begin{align}\label{mpfm}
F_{n,\tau}^{(k)}(r_1, \ldots, r_k) = \E_{U(n)} \left\{ N_{n,\tau}({r_1^2}) (N_{n,\tau}({r_2^2})-1) \ldots (N_{n,\tau}({r_k^2})-k+1)\right\},
\end{align}
of $N_{n,\tau}(r)$, the number of eigenvalue moduli $r_j$ of $A_{n, \tau}$ in the disk
$D_r = \{z\in \mathbb{C}: \quad |z|\le r \}$.
The factorial moments in \eqref{MFM} are $F_{n,\tau}^{(k)}(\sqrt{q_1}^, \ldots, \sqrt{q_k})$ in the scaling limit $n \tau= |a|^2$, $n\to\infty$,
\begin{align*}
\E_{\varphi} \big\{ N_{q_1}(N_{q_2}-1) \ldots (N_{q_k}-k+1) \, \big| \, {\varphi (0)}=a \big\} = \lim_{n\to\infty}F_{n,\tau}^{(k)}(\sqrt{q_1}^, \ldots, \sqrt{q_k}).
\end{align*}

For the computation of $F_{n,\tau}^{(k)}(r_1, \ldots, r_k)$ we recall that these factorial moments are just the eigenvalue correlation functions
\begin{align}\label{rhon}
R_{k}^{(n,\tau)}(z_1, \ldots , z_k) = \frac{n!}{(n-k)!} \int \ldots \int p_{n,\tau} (z_1, \ldots z_n )  \, d^2z_{k+1} \cdots d^2 z_n \, ,
\end{align}
integrated over $D_{r_1} \times \ldots \times D_{r_k}$:
\begin{align}\label{fm}
%\mathbb{E} \left\{ \big(N_r^{(n,\tau)}\big)_k\right\}
F_{n,\tau}^{(k)}(r_1, \ldots, r_k)
=\int_{D_{r_1}}  \ldots \int_{D_{r_k}}  R_{k}^{(n,\tau)}(z_1, \ldots , z_k) \, d^2z_1 \ldots d^2z_k\, .
\end{align}
In \eqref{rhon}, $p_{n,\tau}(z_1, \dots, z_n)$ is the joint eigenvalue density \eqref{eq:-2} in the random matrix ensemble \eqref{A}.

Now, consider
\begin{align*}
f_n(\tau)= \theta (1-\tau) (1-\tau)^{n-1} F_{n,\tau}^{(k)}(r_1, \ldots, r_k)\, .
\end{align*}
The Mellin transform of this function has determinantal form  \cite{F00},
\begin{align}\label{f1a}
{\mathcal M} \left[ f_n ( \tau);s  \right] %= \int_0^1 \tau^{s-1} (1-\tau)^{n-1}  R_k^{(n,\tau)}(z_1, \ldots , z_k)  d\tau \\
= B(n,s)  \int_{{D_{r_1}}}  \ldots \int_{{D_{r_k}}} \det \left[ K^{(n)}_s(z_i, \overline{z}_j) \right]_{i,j=1}^k d^2z_1 \ldots d^2z_k\, , %\, , \quad \text{where } K^{(n)}_s (z,w) =\frac{1}{\pi} |zw|^{s-1} \sum_{l=0}^{n-1}(l+s) z^l w^l
\end{align}
with the kernel $K^{(n)}_s (z,w) =\frac{1}{\pi} |zw|^{s-1} \sum_{l=0}^{n-1}(l+s) z^l w^l$.
When $n$ approaches infinity,  the determenantal kernel $K^{(n)}_s (z,w) $ converges to
 \begin{align}\label{K_s}
K_s (z,w) = \frac{1}{\pi} |zw|^{s-1} \sum_{l=0}^{\infty}(l+s) z^l w^l\, ,
\end{align}
uniformly in $z$ and $w$ inside the unit disk $D$. Hence,
 \begin{align*}
 \lim_{n\to\infty }  {\mathcal M} \left[ f_n\left( \frac{t}{n}\right);s  \right]
 &\!=\!
\lim_{n\to\infty} n^s B(n,s)
\int_{{D_{r_1}}} \! \ldots \!\int_{{D_{r_k}}} \!\!
\det \left[K_s^{(n)} (z_i,\overline{z}_j) \right]_{i,j=1}^k
\, d^2z_1 \ldots d^2z_k
 \\& = \Gamma(s)  \int_{{D_{r_1}}} \! \ldots \!\int_{{D_{r_k}}} \!\! \det \left[K_s (z_i,\overline{z}_j) \right]_{i,j=1}^k \, d^2z_1 \ldots d^2z_k\, .
 \end{align*}
Thefore, in the scaling limit \eqref{sl}, the factorial moments \eqref{fm} can be obtained by first performing the integration in
 \begin{align}\label{pv1}
 \int_{{D_{r_1}}}  \ldots \int_{{D_{r_k}}}  \det \left[K_s (z_i,\overline{z}_j) \right]_{i,j=1}^k \, d^2z_1 \ldots d^2z_k
 \end{align}
then multiplying the obtained expression by $\Gamma (s)$ and applying the inverse Mellin transform, and, lastly, multiplying the inverse Mellin transform  by $e^{t}$.

Let employ this strategy and calculate the expected value and variance of $N_t({r^2})$ in the scaling limit \eqref{sl}. For the calculation of the expected value,
\begin{align*}
 \int_{D_r} K_s (z, \overline{z}) \, d^2 z
  =\frac{r^{2s}}{1-r^2}\, .
\end{align*}
Multiplying the r.h.s. by $\Gamma (s)$, inverting the Mellin transform with the help of identity
\begin{align}\label{imvm}
  {\mathcal M} \left[  \exp \left( -\frac{t}{q^{k}} \right); s\right] =  q^{ks} \, \Gamma (s), \quad \Re s>0 \, ,
\end{align}
 and multiplying the resulting expression by $e^{t}$, we obtain
\begin{align}\label{fm1}
\lim_{n\to\infty}\mathbf{E} \left\{ N_t({r^2})\right\} = \frac{1}{1-r^2}\, \exp \left( t - \frac{t}{r^2}\right)\, .
\end{align}
This expression is in full agreement with the expression for $\mathbf{E} \left\{ N_{q} \right\}$ \eqref{muqa}.
As far as the second factorial moment is concerned, we have
\begin{align*}
 \int_{D_{r_1}} \int_{D_{r_2}}
 \left[
 K_s (z_1, \overline{z}_1)K_s (z_2, \overline{z}_2) -
 K(z_1, \overline{z}_2) K(z_2, \overline{z}_1)
 \right] d^2 z_1 d^2 z_2
 & =  \\
  \MoveEqLeft[7]
  \frac{(r_1r_2)^{2s}}{(1-r_1^2)(1-r_2^2)} -   \frac{(r_1r_2)^{2s}}{1-(r_1r_2)^2}\, .
\end{align*}
Again, multiplying the r.h.s here by $\Gamma (s)$, inverting Mellin transform and multiplying the resulting expression by $e^{t}$, we arrive at
\begin{align*}%\label{fm2}
\lim_{n\to\infty} \E_{U(n)} \left\{ N_t({r_1^2}) (N_t({r_2^2})-1) \right\}=
%& =  \\
%  \MoveEqLeft[7]
 \left( \frac{1}{(1-r_1^2)(1-r_2^2)} -   \frac{1}{1-(r_1r_2)^2}\right)\,
\exp \left( t - \frac{t}{(r_1r_2)^2}\right)
\end{align*}
in full agreement with  \eqref{sigmaqa}.

To obtain the higher order factorial moments, we follow a calculation in \cite{PV04} where the integral in \eqref{pv1} was evaluated at $s=1$. That calculation easily extends to values of $s$ other than unity and goes as follows.  Write the determinant in \eqref{pv1} as the sum over all permutations of matrix indices $1, 2, \ldots, n$ and factorise each permutation $\sigma$ into disjoint cycles. On integration, each cycle $\nu=(i_1, \ldots, i_{m})$ contributes
\begin{align*}
\frac{1}{\pi^{m}}  \int_{D_{r_{i_1}}}  \ldots \int_{D_{r_{i_m}}} \sum_{l=0}^{\infty} (l+s)^{|\nu|} |z_{i_1}|^{2(l+s-1)} \cdots |z_{i_m}|^{2(l+s-1)}  d^2z_{i_1} \cdots d^2z_{i_m} = \frac{(r_{i_1} \cdots r_{i_m})^{2s}}{1-(r_{i_1} \cdots r_{i_m})^2}
\end{align*}
to the product.
Hence, the integral in \eqref{pv1} is
\begin{align}\label{fm3}
\sum_{\sigma\in S_k} \prod_{\nu\in \sigma } (-1)^{|\nu|-1} \frac{r_{\nu}^{2s}}{1-r_{\nu}^2}= (r_1\cdots r_k)^{2s}\sum_{\sigma\in S_k} \prod_{\nu\in \sigma } (-1)^{|\nu|-1} \frac{1}{1-r_{\nu}^2} \, ,
\end{align}
where the sum is over all permutations of $k$ indices, the product is over all disjoint cycles $\nu$ of permutation $\sigma$, $|\nu|$ is the length of $\nu$ and $r_{\nu}=\prod_{j\in \nu} r_j$. Multiplying \eqref{fm3} by $\Gamma (s)$ before applying to it the inverse Mellin transform, see \eqref{imvm}, and then multiplying the resulting expression by $e^{t}$,  one obtains higher order factorial moments:

\begin{proposition} For every $t>0$, the multipoint factorial moments $F_{n,\, t/n}^{(k)}(r_1, \ldots, r_k) $\eqref{mpfm} of the eigenvalue moduli of random subunitary matrices $A_{n,\, t/n}$ converge, in the limit of infinite matrix dimension $n\to\infty$, to
\begin{align}\label{fm4app}
F_{t}^{(k)}(r_1, \ldots, r_k)=\exp\left[t-\frac{t}{(r_1\cdots r_k)^2} \right]\sum_{\sigma\in S_k} \prod_{\nu\in \sigma } (-1)^{|\nu|-1} \frac{1}{1-r_{\nu}^2}\, ,
\end{align}
where the sum is over all permutations of $k$ indices, the product is over all disjoint cycles $\nu$ of permutation $\sigma$, $|\nu|$ is the length of $\nu$ and $r_{\nu}=\prod_{j\in \nu} r_j$.
\end{proposition}
This Proposition and Forrester's and Ipsen's work imply that the multipoint factorial moments \eqref{MFM} of the counting function of the limiting point process $\mathcal{Q}_a$ are given by the expression on the r.h.s. in \eqref{fm4app} with $r_j^2$ replaced by $q_j$, $j=1, \ldots, k$.

\section{A functional-differential equation and formal asymptotic expansion for $F(t,x;q)$}
\label{Appendix DEs}
In this appendix, we will restrict ourselves to real non-negative $x$ and $t$.  We first express the $q$-series $F(x,t;q)$ \eqref{Fext} as solution of a functional-differential equation, which will be useful for formal asymptotic expansion.
\begin{lemma}
\label{lemma1}
    $F(x,t;q)$ is the unique solution of
    \begin{align}
        xF(x,t;q)=e^{(1-q)t}\frac{\partial}{\partial t}F(x,qt;q)
        \label{dfeq}
    \end{align}
    that is regular at $x=0$ and satisfies $F(0,t;q)=1$.
\end{lemma}

\begin{proof}
    Inserting a power series in $x$ into \eqref{dfeq} leads to recursion for the series coefficients which has a unique solution given by \eqref{Fext}.
\end{proof}
Formal asymptotic expansions for $F(x,t;q)$ can be obtained from Lemma \ref{lemma1} by inserting a series Ansatz into the functional-differential equation and solving for the coefficients recursively. While necessarly non-rigorous, this approach allows us to obtain expansions to arbitrary order and to understand the nature of the asymptotic expansions. It is also useful for numerical work.

However, to obtain rigorous expansions, and to fix constants that remain undetermined using the formal approach, we turn to Lemma \ref{lemma2} as a basis for a saddle point expansion.

We first shall illustrate this approach to derive an asymptotic result involving the Gumbel distribution.

\begin{theorem}
\label{formal}
$F(x,t;q)$ admits a formal expansion
\begin{align*}
F(x,(s-\log\varepsilon)/\varepsilon;\exp(-\varepsilon))\sim e^{-xe^{-s}}\Big(1+\sum_{l=1}^\infty G_l(x,\log\varepsilon,s)\Big),
\end{align*}
where
\begin{align*}
G_l(x,\log\varepsilon,s)=\sum\limits_{m=1}^{2l}x^me^{-ms}g_{l,m}(\log\varepsilon,s)
\end{align*}
with $g_{l,m}(\cdot,\cdot)$ a bivariate polynomial of degree $m$ in each of the arguments.
\end{theorem}

\begin{proof}[Sketch of proof] A substitution of the given expansion into (\ref{dfeq}) and collecting terms in powers of $\varepsilon$ leads to differential equations for $G_l$ in $s$ that can be solved iteratively. The polynomial form follows similarly.
\end{proof}

Note that the appearance of the Gumbel distribution is quite straightforward if the variables $t$ and $\varepsilon=-\log q$ are scaled
jointly such that $s=\varepsilon t+\log\varepsilon$ is fixed as $\varepsilon\to 0$.

We now consider the limit of $q\to 1^{-}$ for fixed $t$. Inserting
\begin{align*}
\log F(x,t;q)=f(x,t; q)/\log q
\end{align*}
into \eqref{dfeq} and rearranging terms, we find
\begin{align*}
%\label{qdiff}
\frac{f(x,t;q)-f(x,qt;q)}{(1-q)t}=\log q+\frac{\log q}{(1-q)t}\log\left(\frac{f_t^\prime (x,qt;q)}{x\log q}\right)\;.
\end{align*}
The l.h.s here is a $q$-derivative, which tends to the usual derivative $f_t^\prime (x,t; 1)$ as $q\to 1^{-}$. Thus,  to leading order
\begin{align*} %\label{DE1}
f_t^\prime (x,t;e^{-\varepsilon})=-\frac1t\log\left(-\frac{f_t^\prime (x,t;e^{-\varepsilon})}{x\varepsilon}\right), %\quad \varepsilon=-\log q ,
\end{align*}
a relation that shows a logarithmic dependence on $\varepsilon=-\log q$, similarly to the expansion in Theorem \ref{formal}.
Recalling the Lambert-$W$ function, we find that to leading order
\begin{align*}%\label{DE2}
f_t^\prime (x,t;e^{-\varepsilon}) = \frac{W(-x\varepsilon t)}{t}\, .
\end{align*}
The integration in $t$ can be performed by noticing that $W^\prime (z) + W^\prime (z) W(z)=\frac{W(z)}{z}$. We find that to leading order
\begin{align*}
F(x,t; e^{-\varepsilon}) = e^{-\frac{1}{\varepsilon} \big[\frac12W(-\varepsilon tx)^2  + W(-\varepsilon tx) +C\big]},
\end{align*}
where $C$ is a constant of integration with respect to $t$ and may thus still depend on $x$ and $\varepsilon$.
The Lambert-$W$ function is multivalued and the branch that is needed here is the branch $W_{-1}(z)$ which approaches negative infinity like $-\log(-z)-\log(-\ln(-z))$ as $z\to0^-$ (as opposed to the branch $W_0(z)$ which is regular at the origin).  Note that $C$ is a constant of integration with respect to $t$ only and may thus still depend on $x$ and $\varepsilon$.

This approach works as we have a separation of scales between $\varepsilon$ and $\log\varepsilon$. If one includes further terms in the expansion, one can see that the appearance of $\exp(\varepsilon)$ might still require some re-arranging between these terms, but numerical checks are convincing. Working a bit harder, one could now state a similar result to Theorem \ref{formal}, with the drawback that every step involves another constant of integration that one is not able to determine within this method.

As $t$ is fixed and $\varepsilon\to0$, one might be tempted to further approximate $W(-\varepsilon t x)\sim-\log(\varepsilon tx)-\log(-\log(\varepsilon tx))$ using \eqref{W_ae}, but this leads to a worse approximation. The reason for this is that the ``correct'' asymptotic scale here is given by the Lambert-$W$ function, as is evident from the saddle point approximation of Section \ref{Sec:4}.

\section{Asymptotics of $\Psi(s)$ and $A(s)$}
\label{app:C}
We summarise here the asymptotic behaviour of
\begin{align*}
%\Lambda =-2\log z_s-2\Li_2(1-z_s), \quad
\Psi (s) &=\log z_s+\frac1s\Li_2(1-z_s) \\
A(s) &=\sqrt{\frac{\log z_s}{2(z_s-s)}}
\end{align*}
where $z_s=e^{s+w_s}$, $w_s=W(-se^{-s})$ with the branches of the Lambert-$W$ function chosen as in the main text. It is useful to parametrise %$z_s=e^\lambda$, $\lambda\in \mathbb{R}$.
\begin{align*}
z_s=e^\lambda, \quad s=\frac\lambda{1-e^{-\lambda}},\,\, \lambda\in \mathbb{R} \, .
\end{align*}
In this parametrisation
\begin{align*}
%\Lambda=-2\lambda-2\Li_2(1-e^\lambda), \quad
\Psi=\lambda+\frac1\lambda(1-e^{-\lambda})\Li_2(1-e^\lambda), \quad
A=\sqrt{\frac{\lambda(e^\lambda-1)}{2e^\lambda(e^\lambda-1-\lambda)}}\, .
\end{align*}
As $z_s=e^{s+w_s}$ (we note in passing that $w_s<0$ and therefore $z_se^{-s}<1$), there is yet another way of writing the $s$-parametrisation using
\begin{align*}
s=\frac\lambda{1-e^{-\lambda}}, \quad \quad  w_s&=-\frac\lambda{e^\lambda-1}
\end{align*}
and find
\begin{align*}
%\Lambda&=-2s-2w_s-2\Li_2(1-e^{s+w_s})\\
\Psi=s+w_s+\frac1s\Li_2(1-e^{s+w_s}), \quad
A=\sqrt{\frac{s+w_s}{2(e^{s+w_s}-s)}}\, .
\end{align*}
This parametrisation is not unique, one could for example eliminate the exponential functions by using $e^{s+w_s}=-\frac s{w_s}$.
%We finally note $\lambda=s+w_s$, which will become relevant for the expansions below.

%\subsection{Expansions}

\subsection{Expansions in the limit $s\to1$ ($\lambda\to 0$)}
 An expansion of $s$, $z_s$, and $w_s$ in small values of $\lambda$ gives
\begin{align*}
s&=1+\frac12\lambda+\frac1{12}\lambda^2-\frac1{720}\lambda^4+O(\lambda^5)\\
z_s&=1+\lambda+\frac12\lambda^2+\frac16\lambda^3+\frac1{24}\lambda^4+O(\lambda^5)\\
w_s&=-1+\frac12\lambda-\frac1{12}\lambda^2+\frac1{720}\lambda^4-O(\lambda^5)\, .
\end{align*}
The expansion for $z_s$ is just the exponential series, and we note that $s+w_s=\lambda$. We are interested in expansions in $s$ about $s=1$. We find correspondingly
\begin{align*}
\lambda&=2(s - 1) - \frac23(s - 1)^2 + \frac49(s - 1)^3 - \frac{44}{135}(s - 1)^4+O((s-1)^5\\
z_s&=1+2 (s-1)+\frac43(s-1)^2+\frac49 (s-1)^3+\frac{16}{135} (s-1)^4+O((s-1)^5)\\
w_s&=-1 + (s - 1) - \frac23(s - 1)^2 + \frac49(s - 1)^3 - \frac{44}{135}(s - 1)^4 +O((s-1)^5)\, .
\end{align*}
Hence
\begin{align}
%\Lambda&=\frac12\lambda^2 + \frac1{18}\lambda^3+O(\lambda^5)\\
\Psi&=\frac14\lambda^2-\frac5{72}\lambda^3+\frac1{72}\lambda^4+O(\lambda^5)\\
A&=1 - \frac5{12}\lambda + \frac3{32}\lambda^2 - \frac{269}{17280}\lambda^3 + \frac{1663}{829440}\lambda^4+O(\lambda^5)\, ,
\end{align}
and, correspondingly,
\begin{align*}
%\Lambda&=2(s - 1)^2 - \frac89(s - 1)^3+\frac23(s - 1)^4+O((s-1)^5)\\
\Psi&=(s-1)^2-\frac{11}9 (s-1)^3+\frac43 (s-1)^4+O((s-1)^5), \\
A&=1 - \frac56(s - 1) + \frac{47}{72}(s - 1)^2 - \frac{403}{720}(s - 1)^3 + \frac{8653}{17280}(s - 1)^4+O((s-1)^5)\, .
\end{align*}

\subsection{Expansions in the limit $s\to\infty$ ($\lambda\to +\infty$)}  An expansion of $s$, $z_s$, and $w_s$ in large positive values of $\lambda$ gives
\begin{align*}
s=\frac\lambda{1-e^{-\lambda}}&=\lambda(1+e^{-\lambda}+e^{-2\lambda}+\ldots)\\
z_s&=e^\lambda\\
w_s=-\frac{\lambda e^{-\lambda}}{1-e^{-\lambda}}&=-\lambda(e^{-\lambda}+e^{-2\lambda}+e^{-3\lambda}+\ldots)\, .
\end{align*}
As far as large values of $s$ are concerned,  the argument $-se^{-s}$ of  $W_0(-se^{-s})$ tends to zero for large $s$, so we can use the series expansion
$$
W_0(x)=\sum_{k=1}^\infty\frac{(-k)^{k-1}}{k!}x^k
$$
to express $w_s=W_0(-se^{-s})$, $\lambda=s+w_s$ and $z_s=e^{s+w_s}$ as expansions in $s$. We find
\begin{align*}
\lambda&=s-se^{-s}-s^2e^{-2s}-\frac32s^3e^{-3s}-\frac83s^4e^{-4s}+\ldots\\
z_s&=e^s - s - \frac12s^2e^{-s} - \frac23s^3e^{-2s} - \frac98s^4e^{-3s}+\ldots\\
w_s&=-se^{-s}-s^2e^{-2s}-\frac32s^3e^{-3s}-\frac83s^4e^{-4s}+\ldots
\end{align*}
Hence
\begin{align*}
%\Lambda&=\lambda^2-2\lambda+\frac{\pi^3}3-2(\lambda+1)e^{-\lambda}-\frac12(2\lambda+1)e^{-2\lambda}+\ldots\\
\Psi&=\frac12\lambda-\frac{\pi^2}6\frac1\lambda+\frac{ 3\lambda^2 +6\lambda + \pi^2 + 6}{6\lambda}e^{-\lambda} - \frac{2\lambda + 3}{4\lambda}e^{-2\lambda}+\ldots\\
A&=\sqrt{\frac{\lambda e^{-\lambda}}2}\left(1 + \frac {\lambda}2e^{-\lambda}+ \frac{\lambda(3\lambda + 4)}8e^{-2\lambda}
%+ \frac{\lambda(5\lambda^2 + 12\lambda + 8)}{16}e^{-3\lambda}
+\ldots\right)\, ,
\end{align*}
and correspondingly
\begin{align*}
%\Lambda&=s^2-2s + \frac{\pi^2}3 - 2(s^2 + 1)e^{-s} -\frac12 (4s^3 - 2s^2 + 2s + 1)e^{-2s}+\ldots\\
\Psi&=\frac12s - \frac{\pi^2}6\frac1s + \frac{s+1}se^{-s} + \frac{2s^2 + 2s + 1}{4s}e^{-2s}+\ldots\\
A&=\sqrt{\frac{se^{-s}}2}\left(1 + \frac{2s - 1}2e^{-s}+ \frac{14s^2 - 8s - 1}8e^{-2s}
%+ \frac{172s^3 - 102s^2 - 18s - 3}{48}e^{-3s}
+\ldots\right)\, .
\end{align*}

\subsection{Expansions in the limit $s\to 0$ ($\lambda\to-\infty$)} An expansion of $s$, $z_s$, and $w_s$ in large negative values of  $\lambda$ gives
\begin{align*}
s=-\frac{\lambda e^\lambda}{1-e^\lambda}&=-\lambda(e^{\lambda}+e^{2\lambda}+e^{3\lambda}+\ldots)\\
z_s&=e^\lambda\\
w_s=\frac\lambda{1-e^\lambda}&=\lambda(1+e^{-\lambda}+e^{-2\lambda}+\ldots)
\end{align*}
We are also interested in expansions as $s\to0$. Here, the argument $-se^{-s}$ of $W_{-1}(-se^{-s})$ tends to zero for small s, but the expansion \eqref{W_ae}, i.e.,
\begin{align*}
W_{-1}(x)\sim-\eta-\log\eta+\sum_{n=1}^{\infty}%
\frac{1}{\eta^{n}}\sum_{m=1}^{n}\genfrac{[}{]}{0.0pt}{}{n}{n-m+1}\frac{\left(-\log\eta\right)^{m}}{m!}
\end{align*}
with $\eta=-\log(-x)$ is not particulary helpful: for $x=-se^{-s}$ we find an expansion in $\eta=-\log s+s$ and $\log\eta=\log(-\log s + s)$. An alternative is to use $w=W_{-1}(s)$ explicitly as an additional asymptotic scale (similar to expansions in $s$ and $\log s$, for example). In doing so, we find
\begin{align*}
W_{-1}(-se^{-s})=w\left(1 - \frac1{w + 1}s + \frac1{(w + 1)^3}\frac{s^2}2 +\frac{(2w - 1)}{(w+1)^5}\frac{s^3}6+O(s^4)\right)\, .
\end{align*}
This allows us to write $w_s=W_{-1}(-se^{-s})$, $\lambda=s+w_s$ and $z_s=e^{s+w_s}$ as expansions in $s$ and $w$.
We find
\begin{align*}
\lambda&=w+\frac1{w + 1}s + \frac w{(w + 1)^3}\frac{s^2}2 +\frac{w(2w - 1)}{(w+1)^5}\frac{s^3}6+O(s^4)\\
z_s&=-\frac sw - \frac1{w(w + 1)}s^2 - \frac{(2w + 1}{2w(w + 1)^3}s^3 - \frac{6w^2 + 4w + 1}{w(w + 1)^5}\frac{s^4}6+O(s^5)\\
w_s&=w- \frac w{w + 1}s + \frac w{(w + 1)^3}\frac{s^2}2 +\frac{w(2w - 1)}{(w+1)^5}\frac{s^3}6+O(s^4)\, .
\end{align*}
Hence
\begin{align*}
%\Lambda&=-2\lambda-\frac{\pi^2}3-2(\lambda-1)e^{\lambda}-\frac12(2\lambda-1)e^{2\lambda}+\ldots\\
\Psi&=-\frac{\pi^2}{6\lambda}e^{-\lambda}+\frac{\pi^2+6}{6\lambda}-1+\lambda +\frac{2\lambda-3}{4\lambda}e^\lambda +\frac{6\lambda - 5}{36\lambda}e^{2\lambda}+\ldots\\
A&=\sqrt{\frac{\lambda e^{-\lambda}}{2(\lambda+1)}}\left(1 - \frac{\lambda}{2(\lambda+1)}e^\lambda- \frac{\lambda(4 + \lambda)}{8(\lambda+1)^2}e^{2\lambda}+\ldots\right)\, ,
\end{align*}
and correspondingly
\begin{align*}
%\Lambda&=-2w-\frac{\pi^2}3+ \frac{2(w^2 - w - 1)}{(w + 1)w}s + \frac{2w^4+w^3+w^2+w+1}{2w^2(w + 1)^3}s^2+O(s^3)\\
\Psi&=\frac{\pi^2}6\frac1s + w-1+\frac1w + \frac{2w - 1}{4w^2}s + \frac{3w^2 - 4w + 2}{18w^3(w + 1)}s^2+O(s^3)\\
A&=\sqrt{\frac{w^2}{-2(w+1)s}}\left(1 + \frac1{2w(w + 1)^2}s - \frac{2w^3 + 4w + 1}{8w^2(w + 1)^4}s^2+O(s^3)\right)\, .
\end{align*}

\bibliography{FKP.bib}
\bibliographystyle{ieeetr}

\end{document}